\documentclass[11pt,twoside]{article}
\usepackage{amsfonts,amssymb,amsmath,amsthm}
\usepackage{color}
\usepackage{graphicx}
\usepackage{psfrag}
\usepackage{enumerate}
\usepackage{datetime}
\usepackage{fancyhdr}
 
\newtheorem*{maintheorem}{\bf Main Theorem}
\newtheorem*{MainTheoremRevisited}{\bf Main Theorem Revisited}
\newtheorem*{theorem*}{\bf Theorem}
\newtheorem*{claim*}{Claim}
\newtheorem*{addendum*}{\bf Addendum}
\newtheorem{mytheorem}{\bf Theorem}
\newtheorem{definition}{\bf Definition}[section]
\newtheorem{theorem}[definition]{\bf Theorem}
\newtheorem{lemma}[definition]{\bf Lemma}

\newtheorem{proposition}[definition]{\bf Proposition}
\newtheorem{corollary}[definition]{\bf Corollary}
\newtheorem{remark}[definition]{\bf Remark}

\newtheorem{question}[definition]{\bf Question}

\newtheorem*{conjecture*}{\bf Conjecture}
\newtheorem{problem}[definition]{\bf Problem}
\newtheorem{claim}[definition]{\bf Claim}

\newcommand\TMO{T_OM}

\def\QQ{\mathbb Q}
\def\NN{\mathbb N}
\def\ZZ{\mathbb Z}
\def\RR{\mathbb R}

\def\cO{\mathcal O}
\def\cZ{\mathcal Z}
\def\cV{\mathcal V}
\def\cR{\mathcal R}
\def\cU{\mathcal U}
\def\cA{\mathcal A}
\def\cB{\mathcal B}
\def\cC{\mathcal C}
\def\cD{\mathcal D}
\def\cG{\mathcal G}
\def\cK{\mathcal K}

\def\eps{\epsilon}
\def\NDS{{\mathcal E_\omega(M)}}

\newcommand\Grass{\operatorname{Grass}}
\newcommand\Jac{\operatorname{Jac}}
\newcommand\Prob{{\mathbb P}}
\newcommand\Probe{{\mathbb P}_{\rm erg}}
\newcommand\Proberg{{\mathbb P}_{\rm erg}}
\newcommand\diam{\operatorname{diam}}
\newcommand{\Diff}{{\operatorname{Diff}}}
\newcommand\Id{\operatorname{Id}}
\newcommand\Lip{\operatorname{Lip}}

\newcommand\supp{\operatorname{supp}}
\renewcommand\top{{\operatorname{top}}}

\newcommand\loc{{\operatorname{loc}}}

\newcommand\BA{\mathbf A}\newcommand\BB{\mathbf B}\newcommand\BC{\mathbf C}\newcommand\BD{\mathbf D}

\newcommand\ignorethis[1]{}

\ignorethis{

\definecolor{darkgreen}{rgb}{0.0, 0.3, 0.0}

\usepackage{soul}

 }

\date{\today}
\title{The entropy of $C^1$-diffeomorphisms \\ without a dominated splitting}

\author{J\'er\^ome Buzzi, Sylvain Crovisier, Todd Fisher\footnote{S.C was partially supported by the ERC project 692925 \emph{NUHGD}.  T.F.\ is supported by Simons Foundation grant \# 239708.}}

\makeatletter
\renewcommand{\l@section}{\@dottedtocline{2}{3.8em}{3.2em}}
\renewcommand{\l@subsection}{\@dottedtocline{3}{3.8em}{3.2em}}
\newcommand{\subsectionruninhead}{\@startsection{subsection}{2}{0mm}{-\baselineskip}{-0mm}{\bf\large}}
\newcommand{\subsubsectionruninhead}{\@startsection{subsubsection}{3}{0mm}{-\baselineskip}{-0mm}{\bf\normalsize}}
\makeatother

\begin{document}

\maketitle

\begin{abstract} 
A classical {construction due to Newhouse creates} horseshoes from hyperbolic periodic orbits with large period and weak domination through local $C^1$-perturbations.  Our main theorem shows that,  when one works in the $C^1$ topology, the entropy of such horseshoes can be made arbitrarily close to an upper bound {following} from Ruelle's inequality, i.e., the sum of the positive Lyapunov exponents (or the same for the inverse diffeomorphism, whichever is smaller). 

This {optimal entropy creation} yields a number of consequences for $C^1$-generic diffeomorphisms, {especially in the absence of a dominated splitting.}
{For instance, in the conservative settings, we find formulas for the topological entropy,
deduce that the topological entropy is continuous, but not locally constant at the generic diffeomorphism
and we prove that these generic diffeomorphisms have no measure of maximum entropy. In the dissipative setting, we {show} the locally generic existence of infinitely many homoclinic classes with entropy bounded away from zero.}
\end{abstract}

{{\renewcommand{\thefootnote}{}%
}}

\tableofcontents


\section{Introduction}

Topological and measure theoretic entropy are fundamental dynamical invariants.  Although there is a number of deep theorems about the entropy of smooth dynamics (see, e.g., \cite{Katok, LY, Newhouse1989, DN}), some basic questions remain open.
The present work addresses some of {these questions} for $C^1$-diffeomorphisms of compact manifolds. We focus especially in the conservative setting or when there is no dominated splitting (in some sense {this is} the opposite of the better understood uniformly hyperbolic diffeomorphisms). We build on the classical article of Newhouse \cite{Newhouse} and strengthen and generalize the recent work of Catalan and Tahzibi and others \cite{CT,C1,AACS} to higher dimension as well as to the volume-preserving case.

First, we try and understand the \emph{mechanisms underlying the entropy}, and in particular, 
the Lyapunov exponents of periodic points or invariant measures, and
 horseshoes. For diffeomorphisms, general results of Ruelle  and Katok  relate the entropy of measures to their Lyapunov exponents and to the existence of horseshoes. We study corresponding global variational principles.

Second, we are interested in  the
\emph {continuity properties and local variation}
of the map $f\mapsto h_\top(f)$. 
In general, this map is neither upper, nor lower semicontinuous, but within certain important classes of diffeomorphisms these properties may hold, see for example \cite{Katok, Newhouse1989}.  Also, local constancy may occur even for non structurally stable maps \cite{HSX,BFSV,U,FPS,BF}. We show that, in the conservative case, when there is no dominated splitting, the entropy is never locally constant but is continuous at the generic diffeomorphism.

Third, we ask about the \emph{localization of the entropy}: i.e., the entropy of homoclinic classes and of invariant measures.  In the dissipative setting, when there is no dominated splitting, we show the coexistence of infinitely many homoclinic classes with entropy bounded away from zero. In the conservative setting, also without a dominated splitting,  there is no measure maximizing the entropy and the topological entropy is a complete invariant of Borel conjugacy up to the periodic orbits.

All our results rely on a strengthening of the classical construction of Newhouse \cite{Newhouse} identifying a mechanism for entropy, namely,  the creation of horseshoes from hyperbolic periodic orbits with homoclinic tangencies through $C^1$-perturbations. When there is no dominated splitting, our main theorem produces horseshoes with entropy arbitrarily close to an upper bound following from Ruelle's inequality and this optimality allows the computation of entropy (and related quantities).

{Our perturbations are local: the horseshoes are obtained in an arbitrarily small neighborhood of a well chosen periodic orbit. This relies on the {perturbation} techniques that have been developed
for the $C^1$-topology by Franks~\cite{franks}, Gourme\-lon~\cite{Gourmelon2014}, Bochi and Bonatti~\cite{BochiBonatti}, among others.}

\subsection{Horseshoes arising from homoclinic tangencies}
Before stating our main theorem we need to review a few definitions.  Let us fix a closed (compact and boundaryless) connected Riemannian manifold $M$ of dimension $d_0$ and a $C^1$-diffeo\-mor\-phism $f$.
Any periodic point $p$ has $d_0$ Lyapunov exponents, repeated according to multiplicity,
$\lambda_1(f,p)\leq \dots\leq \lambda_{d_0}(f,p)$.
We set $\lambda^+:=\max(\lambda,0)$ and $\lambda^-:=\max(-\lambda,0)$.

If $N\geq 1$ is an integer,
we say that an invariant compact set $\Lambda$
has an \emph{$N$-dominated splitting} if there exists a non-trivial decomposition
$TM|_{\Lambda}=E\oplus F$ of the tangent bundle of $M$ above $\Lambda$ in two invariant continuous subbundles such that for all $x\in \Lambda$, all $n\geq N$ and all unit vectors $u\in E(x)$ and $v\in F(x)$ we have,
 $$
    \|Df^nu\|\leq \|Df^n v\|/2.
 $$
 In the case $E$ is uniformly contracted and $F$ uniformly expanded, the set $\Lambda$
 is \emph{(uniformly) hyperbolic}. A \emph{horseshoe} is a transitive hyperbolic set, that is locally maximal,
 totally disconnected, and not reduced to a periodic orbit.

We now state our main result, see Theorem~\ref{t.newhouse} for a more detailed version. 

\begin{maintheorem}
For any $C^1$-diffeomorphism $f$ of a compact $d_0$-dimensional manifold $M$ and any $C^1$-neighborhood $\cU$
of $f$, there exist $N,T\geq1$ with the following property.
{If} $p$ {is a periodic point for} $f$ with period at least $T$ and whose orbit has no $N$-dominated splitting, {then}
there exists $g\in\cU$ containing a horseshoe $K$ such that
 $$
   h_\top(g,K) \geq \Delta(f,p):=\min\bigg(\sum_{i=1}^{d_0} \lambda_i^+(f,p), \sum_{i=1}^{d_0} \lambda_i^-(f,p)\bigg).
 $$
 Moreover $K$ and the support of $g\circ f^{-1}$ can be chosen in an arbitrarily small
 neighborhood of the orbit of $p$, and 
if $f$ preserves a volume or symplectic form, then one can choose $g$ to preserve the volume or symplectic form respectively.
\end{maintheorem}
The above inequality is, in some sense, optimal -- see Remark~\ref{r.sharp}.

\paragraph{\it Extension to subbundles and higher smoothness.}
The perturbation method used in the proof of the Main Theorem extends to the case where the lack of dominated splitting
occurs  inside an invariant subbundle.  More specifically,
if $E$ is an invariant sub-bundle over the orbit $O$ of a periodic point $p$,
we denote by $\lambda_{E,1}(f,p)\leq \dots\leq \lambda_{E,\ell}(f,p)$ where
$\ell=\dim(E)$, the Lyapunov exponents corresponding to vectors tangent to $E$.
\begin{MainTheoremRevisited}
Let $f,\cU$ and $N,T\geq1$ be as above.
For any periodic point $p$ of $f$ with period at least $T$,
and any invariant bundle $E$ over the orbit of $p$
which does not admit an $N$-dominated splitting,
there exists $g\in \cU$ having a horseshoe $K$ such that
$$h_\top(g,K)\geq \Delta_E(f,p):=\min\bigg(\sum_{i=1}^{\ell} \lambda_{E,i}^+(f,p), \sum_{i=1}^{\ell} \lambda_{E,i}^-(f,p)\bigg).$$
 Moreover $K$ and the support of $g\circ f^{-1}$ can be chosen in an arbitrarily small
 neighborhood of the orbit of $p$, and 
if $f$ preserves a volume or symplectic form, then one can choose $g$ to preserve the volume or symplectic form respectively.
\end{MainTheoremRevisited}

Newhouse's result extends in the $C^r$-topology (assuming the existence of a saddle point with a homoclinic tangency),
but the entropy of the resulting horseshoes is smaller --this is unavoidable given the continuity properties of the topological entropy for smooth maps (see Newhouse \cite{Newhouse1989}  and, for  surfaces, \cite{burguet}).
Because of a program developed by Gourmelon~\cite{gourmelon-announce},
we believe that a version of our result is also true in higher smoothness,
in the setting of a \emph{cycle of basic sets}.
\begin{conjecture*}
For any $C^r$-diffeomorphism $f$ of a manifold $M$,
and for any hyperbolic periodic point $p$ which belongs to a cycle of basic sets
having no dominated splitting of any index,
there exists $g$, arbitrarily close to $f$ for the $C^r$-topology,
with a horseshoe $K$ such that $h_\top(g,K) \geq \Delta(f,p)/r$.
 \end{conjecture*}

A statement similar to our Main Theorem but restricted to the symplectic setting  has been published by Catalan~\cite{catalan-2016} while this paper was being written. Catalan's proof is different from ours and is specific to the symplectic setting. He uses a global perturbation in contrast to our construction.

\subsection{Consequences for conservative diffeomorphisms}
In the conservative setting
(i.e., in the volume-preserving or symplectic setting), generic irreducibility properties, see for example \cite{BC}, yield the strongest consequences to the Main Theorem. 

Let $\omega$ be a smooth volume or symplectic form on the manifold $M$ with dimension at least $2$. Let $\NDS$ be the interior of the set of $C^1$-diffeomorphisms of $\Diff^1_\omega(M)$ with no dominated splitting. This is always nonempty\footnote{In order to build a diffeomorphism that has robustly no dominated splitting,
it is enough to require for each $1\leq i< d_0$ the existence
of a periodic point $p$ whose tangent dynamics at the period has
$d_0$ simple eigenvalues
$\lambda_1,\dots,\lambda_{d_0}$ such that $|\lambda_k|<|\lambda_{k+1}|$ for each $k\neq i$
and such that $\lambda_i,\lambda_{i+1}$ are non-real conjugated complex numbers.}.
As is usual, we say that the \emph{generic diffeomorphism} in a Baire space of diffeomorphisms
$\mathcal{D}$ has some property, if it holds for all elements of some dense G$_\delta$ subset of $\mathcal{D}$.
\paragraph{a -- Entropy formulas.}
Let $\Grass_k(TM)$ denote the set of $k$-dimensional linear subspaces $E\subset TM$ and $\Jac(f,E)$ the Jacobian of $Df:E\to Df(E)$ with respect to some (arbitrary, fixed) Riemannian metric on $M$.

\begin{mytheorem}\label{t.entropycons}
For the generic $f\in\NDS$,  the topological entropy is equal to
 \begin{enumerate}[(i)]
  \item\label{i.entHS} the supremum of the topological entropies of its horseshoes,
  \item\label{i.entPer} $\Delta(f):=\sup\{\Delta(f,p):p$ periodic point$\}$, and

  \item\label{i.entSubmultiplicative} 
  $\max_{0\leq k\leq \dim(M) }\;\lim_{n\to\infty} \frac 1 n \log \left(\sup _{E\in\Grass_k(TM)}  \Jac(f^n,E)\right).$
 \end{enumerate}
 \end{mytheorem}

\medskip

We observe that formulas \eqref{i.entPer}~and~\eqref{i.entSubmultiplicative} fail for an open and dense set of Anosov diffeomorphisms, since their topological entropy is locally constant.
For a (possibly dissipative) $C^1$ map $f\colon M\to M$ of a compact manifold, the supremum in item \eqref{i.entHS} is still a lower bound whereas Ruelle's inequality shows that the {expression} \eqref{i.entSubmultiplicative} becomes an upper bound:
\begin{equation*}
\begin{split}
\sup\{h_\top(f|K),\; &K\text{ horseshoe}\}\leq \\
&\leq h_\top(f)\leq \max_{0\leq k\leq \dim(M)}\lim_{n\to\infty} \frac 1 n \log \left(\sup _{E\in\Grass_k(TM)}  \Jac(f^n,E)\right).
\end{split}
\end{equation*}

{We note that} items~(\ref{i.entHS}) and (\ref{i.entPer}) above generalize a formula proved by Catalan and Tazhibi \cite{CT} for surface conservative diffeomorphisms and a weaker inequality from \cite{CH} in the symplectic setting (see also \cite{catalan-2016} for the symplectic case).
Item~\eqref{i.entSubmultiplicative}
looks very similar to a result obtained by Kozlovski \cite{Kozlovski1996} for all $C^\infty$, possibly dissipative, maps:
 $$
      h_\top(f) = \max_{0\leq k\leq \dim(M)} \lim_{n\to\infty} \frac1n \log \int_M \sup_{E\in\Grass_k(T_xM)} \Jac(f^n,E) \, dx.
 $$

\paragraph{b -- Measures of maximal entropy.}
The variational principle asserts that the topological entropy is the supremum
of the entropy of the invariant probability measures. A measure is a \emph{measure of maximal
entropy} if it realizes the supremum.

\begin{mytheorem}
\label{t.nomax}
The generic $f\in\NDS$ has no measure of maximal entropy.
\end{mytheorem}

Note that $C^\infty$ diffeomorphisms always admit a measure of maximal
entropy by a theorem of Newhouse \cite{Newhouse1989}. {In general,} this fails in any finite smoothness \cite{misiurewicz,buzziNoMax}.

\paragraph{c -- Continuity, stability.}
We now discuss how the topological entropy depends on the diffeomorphism.

\begin{corollary}\label{c.continuity}
The generic diffeomorphism in $\NDS$ is a continuity point of the map $f\mapsto h_\top(f)$, defined over $\Diff^1_\omega(M)$.
\end{corollary}

According to a classical result of Ma\~n\'e \cite{Mane}, structural stability is equivalent to hyperbolicity. As hyperbolic systems are not dense, this suggests studying weaker forms of stability, e.g.,  local constancy of the entropy (which does hold for some non-hyperbolic diffeomorphisms \cite{BFSV,BF}).

\begin{corollary}\label{c.unstability}
The map $f\mapsto h_\top(f)$ is nowhere locally constant in $\NDS$.
More precisely, for any $f$ and any G$_\delta$ set $\cG$ in $\NDS$,
there exists $g\in \cG$ arbitrarily close to $f$ such that $h_{top}(f)\neq h_{top}(g)$.
\end{corollary}

{We note} the two following immediate consequences: 
\begin{itemize}
\item[(1)] there is an open and dense subset of $\Diff^1_\omega(M)$ on which local constancy of the entropy implies the existence of a dominated splitting; and 
\item[(2)] {the set $h_\top(\mathcal G)$ of the topological entropies of generic diffeomorphisms in $\NDS$ is uncountable.}
\end{itemize}

\paragraph{d -- Borel classification.}
The topological entropy is a {complete invariant among these diffeomorphisms for the following type of conjugacy}.
Recall that an automorphism $T$ of a Borel space $X$ is a bijection of $X$, bimeasurable with respect to a given $\sigma$-field. A typical example is a homeomorphism of a possibly non compact, complete metric space with its $\sigma$-field of Borel subsets.  The \emph{free part} of $T$ is the Borel subsystem $$X^{\operatorname{free}}:=\{x\in X:\{T^nx:n\geq0\}\textrm{ is infinite}\}.$$  {Two such} automorphisms $(X_1,T_1)$,
$(X_2,T_2)$ are \emph{Borel conjugate} if there is a Borel isomorphism $\psi:X_1\to X_2$ with $\psi\circ T_1=T_2\circ\psi$ 
(see \cite{Weiss1,Weiss2} for background).

Using a recent result of Hochman \cite{Hochman2}, Theorems \ref{t.entropycons} and \ref{t.nomax} yield the classification (we refer to Proposition \ref{p.Borel} for a more detailed statement):

\begin{corollary}\label{c.almostBorel}
The free parts of generic diffeomorphisms $f$ of $\NDS$ are classified, up to Borel conjugacy, by their topological entropy $h_\top(f)$. 
Moreover,  they are Borel conjugate to free parts of topological Markov shifts. 
\end{corollary}

Thus, two generic diffeomorphisms in $\NDS$ are Borel conjugate themselves if and only if they have equal entropy and the same number of periodic points of each period (see \cite{Weiss1}). We do not believe  this to hold assuming only equal entropy.

\paragraph{e -- Entropy structure, tail entropy, symbolic extensions.}
For a diffeomorphism $f:M\to M$ of a compact manifold we let $\mathbb{P}(f)$ be the set of $f$-invariant Borel probability measures.  The entropy function for $\mu\in\mathbb{P}(f)$ is defined as $\mu\mapsto h_{\mu}(f)$. 
An \emph{entropy structure} is a non-decreasing sequence of functions on
$\mathbb{P}(f)$ converging to the entropy function in a certain manner. It captures many dynamical properties beyond the entropy function, see Section \ref{s.entropystructure}  and especially \cite{DownarowiczBook} for details.

For instance, the entropy structure determines the entropy at arbitrarily small scales. More precisely, for each $\eps,\delta>0$ and $n\geq 1$,
one considers the maximal cardinality $s_f(n,\delta,\eps)$ of sets
$\{x_1,\dots,x_\ell\}\subset M$ satisfying:
\begin{enumerate}[(a)]
\item for each $0\leq k<n$ and each $i,j$, one has $d(f^k(x_i),f^k(x_j))\leq \eps$,
\item for each $i\neq j$, there is $0\leq k< n$ such that $d(f^k(x_i),f^k(x_j))>\delta$.
\end{enumerate}
The \emph{tail entropy} is then defined as:
\begin{equation}\label{eq-def-tail-h}
 h^*(f):=\lim_{\eps\to 0}\lim_{\delta\to 0} \limsup_{n\to +\infty} \frac 1 n \log s_f(n,\delta,\eps).
\end{equation}

For $C^\infty$ diffeomorphisms, the tail entropy vanishes~\cite{buzzi}. Therefore, by \cite{BFF}, there always is a \emph{symbolic extension}, that is, a topological extension $\pi\colon (K,\sigma)\to (M,f)$ such that $(K, \sigma)$ is a subshift on a finite alphabet. This also holds for uniformly hyperbolic $C^1$-diffeomorphisms (using expansivity and Markov partitions).
This is in contrast to $C^1$-generic systems, for which the lack of dominated splitting prevents the existence of a
symbolic extension (see~\cite{CT,C1,AACS}).

The following strengthens these results inside $\NDS$
(see Proposition \ref{p.ualpha} for a complete statement).

\begin{mytheorem}\label{t.ualphageneral}
For the generic diffeomorphism in $\NDS$, there is no symbolic extension and, more precisely,
\begin{enumerate}
\item[(i)] the order of accumulation of the entropy structure is the first infinite ordinal, and
\item[(ii)] the tail entropy is equal to the topological entropy.
\end{enumerate} 
\end{mytheorem}

{This theorem strenghtens results of \cite{ABC2,CT,CH,AACS} showing the lack of a symbolic extension.}

\paragraph{f -- When there is some dominated splitting.}
Our revisited Main Theorem yields a formula for the tail entropy.
For diffeomorphisms in $\Diff^1_\omega(M)\setminus \NDS$,
there is a unique finest dominated splitting (see~\cite[Proposition B.2]{BDV}):
\begin{equation}\label{e.finest-DS}
TM=E_1\oplus E_2\oplus\dots\oplus E_\ell.
\end{equation}
We define the finest dominated splitting of a diffeomorphism in $\NDS$ to be the trivial splitting, i.e., $\ell=1$ and $E_1=TM$.

\begin{mytheorem}\label{t.tail-general}
There is a dense G$_\delta$ set $\mathcal G$ of diffeomorphisms $f\in\Diff^1_\omega(M)$ with the following properties.
Given the finest dominated splitting $E$ as {above}, the tail entropy coincides with
$$h^*(f)=\Delta^*(f):=
\sup\{\Delta_{E_i}(f,p)\; : \; p \text{ periodic },\;  1\leq i \leq \ell\}.$$
Moreover, the tail entropy is upper semicontinuous at any diffeomorphism in $\mathcal G$.
\end{mytheorem}

Note that {when} there exists a dominated splitting, it gives $C>0$ such that
for any periodic point $p$ we have
$C+\max_i\Delta_{E_i}(f,p)<\Delta(f,p)$.
Hence, the above theorem (together with Theorem~\ref{t.entropycons}) gives
a characterization of diffeomorphisms with a dominated splitting.

\begin{corollary}\label{c.caracterize-DS}
The generic diffeomorphism $f$ in $\Diff^1_\omega(M)$
has a (non-trivial) dominated splitting if and only if
$h^*(f)<\Delta(f)$.
\end{corollary}

{Catalan {\it et al.} \cite{C1,AACS} have shown that, $C^1$-generically, the existence of a ``good" dominated splitting (see item (i) below) is equivalent to zero tail entropy (in the volume-preserving setting and in the dissipative setting for isolated homoclinic classes). In the conservative setting, we show that this is also equivalent to both the continuity  of the tail entropy and its local constancy.

\begin{corollary}\label{c.unstability-hstar}
There is a dense G$_\delta$ set of diffeomorphisms $f$ in $\Diff^1_\omega(M)$, such that the following are equivalent:
\begin{itemize}
\item[(i)] any sub-bundle in the finest dominated splitting of $f$ is uniformly contracted, uniformly expanded or one-dimensional;
\item[(ii)] the tail entropy is {constant on a neighborhood of $f$};
\item[(iii)] {the tail entropy is continuous at $g=f$;}
\item[(iv)] $h^*(f)=0$.
\end{itemize}
\end{corollary}

We also obtain:

\begin{corollary}\label{c.loc-cte}
For the generic $f$ in $\Diff^1_\omega(M)$,
whenever $h_\top(f)=h^*(f)$, the topological entropy fails to be locally constant.
\end{corollary}

The following questions are thus natural.

\begin{question}
Consider a generic $f\in \Diff^1_\omega(M)$ with a dominated splitting.\\
i. Is it possible to have $h^*(f)=h_{top}(f)$?\\
ii. When $h^*(f)<h_{top}(f)$, is the map $g\mapsto h_{top}(g)$ locally constant near $f$?\\
iii. Does $h_{top}(f)$ coincide with the supremum of the entropy of the horseshoes?
\end{question}

\subsection{Consequences for dissipative diffeomorphisms}\label{ss.dissipative}
For dissipative (i.e. non-necessarily conservative) diffeomorphisms, the chain-recurrent dynamics decomposes into
disjoint invariant compact sets -- the chain-recurrence classes~\cite[Section 9.1]{robinson}.
For $C^1$-generic diffeomorphisms, each chain recurrent class that contains a periodic orbit $O$
coincides with its \emph{homoclinic class $H(O)$} (see Section~\ref{s.prelim}).  In particular,
if the class is not reduced to $O$ it contains an increasing sequence of
horseshoes whose union is dense in $H(O)$.
The other chain-recurrence classes are called \emph{aperiodic classes},
see~\cite{BC}.

\paragraph{a-- Entropy of individual chain-recurrence classes.}
Some of the previous consequences extend immediately to
homoclinic classes lacking any dominated splitting.

\begin{corollary}
Let $f$ be a generic diffeomorphism in $\Diff^1(M)$
and let $H(O)$ be a homoclinic class with no dominated splitting.
Then 
$$\Delta(f_{|H(O)}):=\sup\{\Delta(f,p):  p \text{ periodic point in } H(O)\}$$
is equal to the supremum of the entropy of horseshoes in $H(O)$.
Moreover it is a lower bound for the tail entropy $h^*(f_{|H(O)})$.
\end{corollary}

Note that a homoclinic class may contain periodic points with different stable dimensions.
The equality in the previous corollary still holds if one restricts
to periodic points and horseshoes having a given stable dimension.

However we don't know if a homoclinic class $H(O)$ may support a measure
which is only approximated in the vague topology by periodic orbits \emph{outside} $H(O)$.
This is the main obstacle to the following generalization of Theorem \ref{t.entropycons}:

\begin{problem}
For a generic diffeomorphism $f\in \Diff^1(M)$
with a homoclinic class $H(O)$, does $h_{top}(f_{|H(O)})$
coincides with the supremum of the entropy of horseshoes in $H(O)$?
\end{problem}

It is known that the generic diffeomorphism can exhibit aperiodic classes~\cite{bonatti-diaz-aperiodic},
but the examples are limits of periodic orbits with very large periods
and have zero entropy. A major question is thus:

\begin{problem}
For a generic diffeomorphism $f\in \Diff^1(M)$
does the entropy vanish on aperiodic classes?
\end{problem}

\paragraph{b-- Infinitely many chain-recurrence classes with large entropy.}
It has been known for sometime \cite{bonatti-diaz-aperiodic} that coexistence of infinitely many non-trivial homoclinic classes is locally generic. We show that this is still the case even for homoclinic classes with a given lower bound on entropy.

\begin{mytheorem}\label{t.infiniteHC}
For the generic $f\in \Diff^1(M)$,
any homoclinic class $H(O)$ with no dominated splitting
is accumulated by infinitely many disjoint homoclinic classes
with topological entropy larger than a constant $h>0$.
\end{mytheorem}

In fact, this theorem can be proved using Newhouse's original argument. The use of Theorem \ref{t.newhouse} instead improves the lower bound $h$ on the entropy of the accumulating homoclinic classes from
$$\sup\{\min_i|\lambda_i(f,p)|:  p \text{ periodic point in } H(O)\}$$
to $\Delta(f_{|H(O)})$.

\paragraph{c-- Approximation of measures by horseshoes with large entropy.} 
Katok's horseshoe theorem says that any hyperbolic ergodic measure of a $C^{1+\alpha}$-diffeomorphism ($\alpha>0$) can be approximated in entropy by
horseshoes.
We show a similar result for generic $C^1$ diffeomorphisms with no dominated splitting:

\begin{mytheorem}\label{t.katok}
For a generic $f\in\Diff^1(M)$
(or $f\in\Diff^1_\omega(M)$),
if $\mu$ is  an ergodic measure such that $\supp\mu$ has no dominated splitting, then there exists a sequence of horseshoes $(K_n)$ which converge to $\mu$ in the three following senses:
 \begin{itemize}
  \item[(i)] in the weak-* topology,
  \item[(ii)] with respect to the Hausdorff distance to the support,
  \item[(iii)] with respect to entropy:  $\lim_n h_\top(f,K_n)=h(f,\mu)$.
\end{itemize}
\end{mytheorem}

Let us point out that this approximation property also holds for $C^1$-diffeomorphisms \emph{with} a dominated splitting \cite{gelfert}, though for completely different reasons.

\paragraph{d-- Dimension of homoclinic classes.}
Newhouse already noticed~\cite{Newhouse} that homoclinic tangencies
of surface diffeomorphisms allows one to build horseshoes with large Hausdorff dimension.
In this direction we obtain:

\begin{mytheorem}\label{t.dimension}
For a generic $f\in\Diff^1(M)$
and a periodic point $p$ such that
$$\sum_{i=1}^{\ell} \lambda_{i}^-(f,p)\geq \sum_{i=1}^{\ell} \lambda_{i}^+(f,p),$$
if the homoclinic class  containing $p$ has no dominated splitting,
then its Hausdorff dimension is strictly larger than the
unstable dimension of $p$
\end{mytheorem}

\subsection{Outline}

The paper proceeds as follows. In Section \ref{s.prelim} we state background results and definitions.  In Section \ref{s.perturbative}  we {list the perturbative tools}.  Section \ref{s.main} is {devoted to our Main Theorem, i.e., the construction of the linear horseshoe}.  In Section \ref{s.entropy}  we prove Theorem \ref{t.entropycons} on the computation of entropy in the conservative setting and {its} Corollaries \ref{c.continuity} and \ref{c.unstability}.  In Section \ref{s.non} we prove Theorem \ref{t.nomax} on the nonexistence of measures of maximal entropy
and the Borel classification, Corollary \ref{c.almostBorel}.
In Section \ref{s.entropystructure} we prove Theorems~\ref{t.ualphageneral} and~\ref{t.tail-general} and Corollaries \ref{c.unstability-hstar}, \ref{c.loc-cte} on the entropy structure, tail entropy, and symbolic extensions.
In Section \ref{s.nonconservative} we prove Theorems \ref{t.infiniteHC}, \ref{t.katok}
and  and \ref{t.dimension} {about} homoclinic classes with no dominated splitting in the dissipative setting.
\\

\noindent{\bf Acknowledgements:}

{The authors wish to thank Nicolas Gourmelon for discussions about the conservative versions of some of his perturbation techniques and also David Burguet as well as Tomasz Downarowicz for discussions about the entropy structure of diffeomorphisms.
 We also thank Yongluo Cao for pointing out a gap in the original proof of Theorem \ref{t.bound-tail}.

\section{Preliminaries}\label{s.prelim}

In this section we review some facts of hyperbolic theory, perturbation tools and generic results in the $C^1$ setting and extend some of them as needed in the rest of the paper. {We refer to \cite{robinson} for further details.}

\paragraph{Symplectic linear algebra.}
We use the Euclidean norm on $\RR^{d_0}$ and the operator norm on matrices.
When $d_0=2d$,
one considers the canonical symplectic form is $\omega(u,v):=u^TJv$ with $J:=\begin{pmatrix} 0 & I_{d}\\ -I_{d} & 0\end{pmatrix}$. For $d\times d$-matrices $\BA,\BB,\BC,\BD$,
 \begin{equation}\label{e.symplecticmatrix}
   \begin{pmatrix} \BA & \BB\\ \BC& \BD\end{pmatrix}\in Sp(2d,\RR) \iff
    \left\{\begin{array}{l} \BA^T\BC=\BA\BC^T,\\ \BB^T\BD=\BD^T\BB,\text{ and }\\
    \BA^T\BD-\BC^T\BB=I_d.\end{array}\right. 
 \end{equation}

The \emph{symplectic complement} $E^\omega$ of a subspace $E\subset \RR^{2d}$ is the linear subspace of vectors $v$ such that
$\omega(v,u)=0$ for all $u\in E$.

A subspace $E\subset \RR^{2d}$ is \emph{symplectic} if the restriction $\omega|E\times E$ of the symplectic form is non-degenerate (i.e. symplectic).
Two symplectic subspaces $E,E'$ are \emph{$\omega$-orthogonal} if for any $u\in E$, $u'\in E'$ one has $\omega(u,u')=0$.

A $d$-dimensional subspace $E$ is \emph{Lagrangian} if the restriction of the symplectic form vanishes, i.e. $E^\omega=E$.
Obviously, the stable (resp. unstable) space of a linear symplectic map is Lagrangian.

\begin{proposition}\label{p.change-basis2}
There exists $C>0$ (only depending on $d$)
such that for any Lagrangian space $L\subset \RR^{2d}$,
there exists $A\in Sp({2d},\RR)$
satisfying
$$A(L)=\RR^{d}\times \{0\}^{d} \text{ and }
\|A\|,\|A^{-1}\|<C.$$
\end{proposition}
\begin{proof}
We first prove:
\begin{claim*}
For any $\eps>0$, there exists $\delta>0$ such that
if $(e_1,\dots,e_{d})$ is $\delta$-close to the first $d$ vectors of the standard basis $(E_1,\dots,E_{2d})$ of
$\RR^{2d}$ and generates a Lagrangian space,
then there exists $A\in Sp(2d,\RR)$ that is $\eps$-close to the identity
such that $A(e_i)=E_i$ for $i=1,\dots,d$.
\end{claim*}
Indeed, if $(e_1,\dots,e_d)$ is the matrix $(m_{ij})_{1\leq i\leq 2d,1\leq j\leq d}$ with $e_j=\sum_{i=1}^{2d} m_{ij}E_i$ for $j=1,\dots,m$, one defines a $2d\times 2d$ matrix as:
 $$
   \begin{pmatrix}\BA& 0\\ \BC&\BD\end{pmatrix} \text{ such that }\begin{pmatrix} \BA \\ \BC\end{pmatrix}=(e_1,\dots,e_d)
    \text{ and }\BD=(\BA^T)^{-1}.
  $$
This matrix sends $E_i$ to $e_i$ for $1\leq i\leq d$ and is symplectic since $(\BA^T\BC-\BC^T\BA)_{ij}=\omega(e_i,e_j)=0$. Hence its inverse has the claimed properties.

It is well-known that the symplectic group acts transitively on the Lagrangian spaces (by a variation of the preceding proof).
The proposition follows from the claim and the compactness of the set of Lagrangian spaces.
\end{proof}

\paragraph{Distance for the $C^1$ topology.}
Let $M$ be a compact connected boundaryless Riemannian manifold with dimension $d_0$.
The tangent bundle is endowed with a natural distance: if $\nabla$ is the Levi-Civita connection,
the distance between $u\in T_xM$ and $v\in T_yM$ is the infimum of $\|u-\Gamma_\gamma v\|+\text{Length}(\gamma)$ ($\Gamma_\gamma$ denoting the parallel transport) over $C^1$-curves $\gamma$ between $x$ and $y$.
Let $\Diff^1(M)$ denote the space of $C^1$-diffeomorphisms of $M$.
{We will use the following standard} distance defining the $C^1$-topology:
$$d_{C^1}(f,g)=\sup_{v\in T^1M} \max\big( d(Df (v),Dg (v)), \; d(Df^{-1}(v),Dg^{-1}(v))\big).$$
We say that $g$ is an \emph{$\eps$-perturbation} of $f$ when $d_{C^1}(g,f)<\eps$. 
{We say that a property holds \emph{robustly} if it holds for all $\eps$-perturbations for $\eps>0$ small enough.}

\paragraph{Conservative diffeomorphisms.}
The vector space $\RR^{d_0}$ is endowed with the standard volume form
and with the standard symplectic form described above (when $d_0$ is even).
If $\omega$ is a volume or a symplectic form on $M$, one denotes by $\Diff^1_\omega(M)$ the subspace of diffeomorphisms which preserve $\omega$.
The charts $\chi \colon U\to \RR^{d_0}$ of $M$ that we will consider will always send $\omega$ on the standard Lebesgue volume
or symplectic form of $\RR^{d_0}$; it is well-known that any point admits a neighborhood with such a chart.

\paragraph{Invariant measures.}
We denote by $\Prob(f)$ the set of all Borel probability measures that $f$ preserves,
by $\Probe(f)$ the set of those which are ergodic and by $d_*$
a distance on the space of probability measures of $M$ compatible with the weak star topology.

\paragraph{Dominated splitting.}
The definition has been stated in the introduction. {The \emph{index} of a dominated splitting $E\oplus F$ is the dimension of $E$, its least expanded bundle.}
We note:

\begin{lemma}\label{l.domination}
If $(f_n)$ converges to $f$ for the $C^1$-topology
and if $(\Lambda_n)$ are compact sets in $M$ which converge to $\Lambda$ for the Hausdorff topology,
such that $\Lambda_n$ is $f_n$-invariant and has {an $N$-dominated splitting} at index $i$ for $f_n$,
then $\Lambda$ is $f$-invariant and has an $N$-dominated splitting at index $i$ for $f$.
\end{lemma}

In particular, if $\Lambda$ has no $N$-dominated splitting at index $i$ for $f$,
then for any diffeomorphism $g$ sufficiently $C^1$-close,
any invariant compact set $\Lambda'$ close to $\Lambda$ has no $N$-dominated splitting at index $i$ either.

\begin{definition}\label{def-TNweak} 
Given positive integers $T, N$,
we say that a periodic point $p$ is \emph{$T, N$-weak} if the period of $p$ is at least $T$ and its orbit has no $N$-dominated splitting.
\end{definition}

\paragraph{Uniform hyperbolicity.}
Let $\Lambda$ be a compact invariant subset of $M$. If there is a splitting $TM|\Lambda=E\oplus F$ with $E$ uniformly contracted and $F$ uniformly expanded, one says that $\Lambda$ is a \emph{hyperbolic set} for $f$.  
If $\Lambda$ is a hyperbolic set for a $C^r$ diffeomorphism $f$, and $x\in\Lambda$, the \emph{stable manifold of $x$}, denoted
$W^s(x)=\{ y\in M\, :\, d(f^nx, f^ny)\to 0, n\to \infty\}$, is a $C^r$ immersed submanifold.  Likewise, we can define the unstable manifold of $x$ using backward iteration. 

A hyperbolic set $\Lambda$ is \emph{locally maximal} if there exists a neighborhood $U$ of $\Lambda$ such that $\Lambda=\bigcap_{n\in\mathbb{Z}}f^n(U)$.  An invariant set $\Lambda$ is transitive for $f$ if there exists a point in $\Lambda$ whose orbit closure is $\Lambda$.
We recall that a \emph{horseshoe} for a diffeomorphism is a locally maximal transitive hyperbolic set that is homeomorphic to the Cantor set.
\medskip

If $p$ is a periodic point for $f$,
we denote by $\pi(p)$ the period of $p$ and by $\cO(p)$  its orbit.
It is a \emph{saddle} if the orbit $\cO(p)$ is hyperbolic and is neither a sink nor a source (so both $E$ and $F$ are nontrivial).

\paragraph{Homoclinic classes.}
Two hyperbolic periodic orbits $O_1,O_2$ are \emph{homoclinically related}
if 
\begin{itemize}
\item[--] $W^s(O_1)$ has a transverse intersection point with $W^u(O_2)$, and 
\item[--] $W^u(O_1)$ has a transverse intersection point with $W^s(O_2)$.
\end{itemize}
Equivalently, there exists a horseshoe containing $O_1$ and $O_2$.
The \emph{homoclinic class} of a hyperbolic periodic point $p$ is the closure of the union of
hyperbolic periodic orbits homoclinically related to the orbit of $p$; it is denoted $H(p)$.
\medskip

\paragraph{Lyapunov exponents.}
Let $f\in \mathrm{Diff}^{1}(M)$ and $\mu$ be an $f$-invariant Borel probability measure.  Then there is a full measure set of points 
such that $T_x M = E_1 (x) \oplus \cdots E_{k(x)}(x)$ where each $E_i(x)$ is a linear subspace and $v\in E_i(x)$ if and only if $v=0$ or 
$$\lim_{n\to \pm \infty} \frac{\log \|D_x f^n v\|}{n}=\mu_i(x)$$
where $\mu_1(x)< \mu_2(x)<\cdots< \mu_{k(x)}(x)$. Each $\mu_i(x)$ is a \emph{Lyapunov exponent of $x$}. It has some \emph{multiplicity} $\dim E_i(x)$. It is convenient to define $\lambda_1(x)\leq\lambda_2(x)\leq\dots\leq\lambda_{d_0}(x)$ by repeating each $\mu_i(x)$ according to the multiplicity. 

If $\mu$ is an ergodic measure then the previous functions $k,\mu_i,\lambda_j$ are constant almost everywhere.  In this case we will denote the Lyapunov exponents by $\mu_i(\mu)$ or $\lambda_i(\mu)$. For an ergodic measure we let $\Delta^+(\mu)$ be the sum of the positive Lyapunov exponents counted with multiplicity.  Likewise, $\Delta^-(\mu)\leq0$ is the sum of the negative Lyapunov exponents counted with multiplicity, and 
$$
\Delta(\mu)=\min\{ \Delta^+(\mu), -\Delta^-(\mu)\}.$$
The notations $\Delta^+(p)$, $\Delta^-(p)$, and $\Delta(p)$ are defined {in the obvious way.}

When there exists a dominated splitting
$TM=E\oplus F\oplus G$, one defines
$\Delta^+_F(\mu)$ (resp. $\Delta^-_F(\mu)$) as the sum of the positive (resp. negative)
Lyapunov exponents associated to directions inside $F$.
One sets $\Delta_F(\mu)=\min(\Delta_F^+(\mu),-\Delta^-_F(\mu))$.

We sometimes include in the notations the dependence of the previous objects on the diffeomorphism, e.g.,  writing $\cO(f,p)$ instead of $\cO(p)$.

\section{Perturbative tools}\label{s.perturbative}

We will need to perturb the differential or to linearize the map along a periodic orbit as well as to create homoclinic tangencies. In order to restrict ourselves to local perturbations, we do not use the powerful technique of transitions introduced in~\cite{BDV} and applied in many subsequent works. Instead we follow mostly
the approach developed in~\cite{Gourmelon2010,Gourmelon2014,BochiBonatti}.
Note that we also want to preserve some homoclinic connections and volume or symplectic forms.
In this section we state all the results that we need and that have an independent interest.
The proofs in the the dissipative case are already known. The proof of conservative versions
are given in a separate paper \cite{BCF}.

\begin{definition}\label{def-evNpert}
Consider $f\in\Diff^1(M)$, a finite set $X\subset M$, a neighborhood $V$ of $X$, and $\eps>0$.
A diffeomorphism $g$ is an \emph{$(\eps, V, X)$-perturbation} of $f$ if $d_{C^1}(f,g)<\eps$ and $g(x)=f(x)$ for all $x$ outside of $V\setminus X$.
\end{definition}

\begin{definition}
Given $f\in \Diff^1(M)$, a periodic point $p$ of period $\ell =\pi(p)$ and $\eps>0$,
an \emph{$\eps$-path of linear perturbations at $\cO(p)$}
is a family of paths $(A_{i}(t))_{t\in[0,1]}$, $1\leq i\leq \ell$, of linear maps $A_i(t)\colon T_{f^i(p)}M\to T_{f^{i+1}(p)}M$
satisfying:
\begin{itemize}
\item[--] $A_{i}(0)=Df(f^i(p))$,
\item[--] $\sup_{t\in [0,1]} \big(\max{\|Df(f^i(p))-A_{i}(t)\|, \|Df^{-1}(f^i(p))-A^{-1}_{i}(t)\|}\big)<\eps$.
\end{itemize}
\end{definition}
\medskip

For modifying the image of one point, one uses this classical proposition.

\begin{proposition}\label{p.elementary}
For any $C,\eps>0$,
there is $\eta>0$ with the following property.
For any $f\in \Diff^1(M)$ such that $Df$, $Df^{-1}$ are bounded by $C$,
for any points $x,y$ such that $r:=d(x,y)$ is small enough,
one can find an $(\eps, B(x,r/\eta),x)$ perturbation $g$ of $f$ satisfying $g(x)=f(y)$.
Furthermore, if $f$ preserves a volume or symplectic form, one can choose $g$ to preserve it.
\end{proposition}

\subsection{Franks' lemma.}
We will use the following two strengthening of the classical Franks' lemma.

\begin{theorem}[Franks' lemma with linearization]\label{t.linearize}
Consider $f\in \Diff^1(M)$, $\eps>0$ small, a finite set $X\subset M$ and a chart $\chi\colon V\to \RR^{d_0}$ with $X\subset V$.
For $x\in X$, let $A_x\colon T_xM\to T_{f(x)}M$ be a linear map such that 
$$\max(\|A_x-Df(x)\|, \|A^{-1}_x-Df^{-1}(x)\|)<\eps/2.$$
Then there exists an $(\eps, V,X)$-perturbation $g$ of $f$ such that for each $x\in X$ the map $\chi\circ g \circ \chi^{-1}$ is linear in a neighborhood of $\chi(x)$ and
$Dg(x)=A_x$. Moreover if $f$ preserves a volume or a symplectic form, one can choose $g$ to preserve it also.
\end{theorem}

The next statement is proved in~\cite{Gourmelon2014} in the dissipative case. It generalizes Franks' lemma while preserving a homoclinic relation.

\begin{theorem}[Franks' lemma with homoclinic connection]\label{l.gourmelon}
Let  $f\in \Diff^1(M)$ and $\eps > 0$ small.
Consider:
\begin{itemize}
\item[--] a hyperbolic periodic point $p$ of period $\ell =\pi(p)$,
\item[--] a chart $\chi\colon V\to \RR^{d_0}$ with $\cO(p)\subset V$,
\item[--] a hyperbolic periodic point $q$ homoclinically related to $\cO(p)$, and
\item[--] an $\eps/2$-path of linear perturbations $(A_{i}(t))_{t\in[0,1]}$, $1\leq i\leq \ell$, at $\cO(p)$ such that the composition $A_{\ell}(t)\circ \dots A_{1}(t)$ is hyperbolic for each $t\in [0,1]$.
\end{itemize}
Then there exists an $(\eps, V,\cO(p))$-perturbation $g$ of $f$ such that, for each $i$ the chart
 $\chi\circ g\circ \chi^{-1}$ is linear and coincides with $A_{i}(1)$ near $f^i(p)$, and $\cO(p)$ is still homoclinically related to $q$.
\smallskip
Moreover if $f$ and the linear maps $A_i(t)$ preserve a volume or a symplectic form,
one can choose $g$ to preserve it also.
\end{theorem}

\subsection{Perturbation of periodic linear cocycles}\label{ss.perturbative-cocycle}
We now explain how to obtain the paths of linear perturbations necessary to apply the previous theorem.
These results are well-known in the dissipative case.
The volume-preserving case follows
(since the modulus of the determinant is {almost} preserved along the paths of dissipative perturbations).
The symplectic case is more delicate and proved in \cite{BCF}. 

Let $G$ be a subgroup of $GL(d_0,\RR)$. A \emph{periodic cocycle} with period $\ell$ is an $\ell$-periodic sequence $(A_i)_{i\in \ZZ}$ in $G$.  The \emph{eigenvalues (at the period)} are the eigenvalues of $\cA:=A_{\ell}\dots A_{1}$. The cocycle is \emph{hyperbolic} if $\cA$ has no eigenvalue on the unit circle. The cocycle is \emph{bounded by $C>1$} if $\max (\|A_i\|,\|A_i^{-1}\|)\leq C$ for $1\leq i\leq \ell$. An $\eps$-path of perturbations is a family of periodic cocycles $(A_i(t))_{t\in [0,1]}$ in $G$ such that
$A_{i}(0)=A_i$ and
$\max{\|A_i-A_{i}(t)\|, \|A_i^{-1}-A^{-1}_{i}(t)\|}<\eps$
for each $t\in [0,1]$.

The next proposition allows us to obtain simple spectrum for the periodic cocycle. The proof in the dissipative setting is essentially contained in the claim of the proof of~\cite[Lemma 7.3]{BochiBonatti}.

\begin{proposition}[Simple spectrum]\label{p.simple}
For any $d_0\geq 1$ and $\eps>0$, any periodic cocycle in $GL(d_0,\RR)$ or $SL(d_0,\RR)$ admits an $\eps$-path of perturbations $(A_i(t))_{t\in [0,1]}$ with the same period $\ell$,
such that the composition
$\cA(t):=A_{\ell}(t)\dots A_{1}(t)$ satisfies:
\begin{itemize}
\item[--] the moduli of the eigenvalues of $\cA(t)$ are constant in $t\in[0,1]$;
\item[--] $\cA(1)$ has $d_0$ distinct eigenvalues; their arguments are in $\pi\QQ$.
\end{itemize}
The same result holds in the group $Sp(d_0,\RR)$ if the cocycle is hyperbolic.
\end{proposition}

We will now exploit the lack of an $N$-dominated splitting.
One can make all  
stable (resp. unstable) eigenvalues to have equal modulus by a perturbation.
This has been proved in the dissipative and volume-preserving settings  in~\cite[Theorem 4.1]{BochiBonatti}.

\begin{theorem}[Mixing the exponents]\label{t.BochiBonatti}
For any $d_0\geq 1$, $C>1$, $\eps>0$, there exists $N \geq 1$ with the following property.
For any hyperbolic periodic cocycle $(A_i)_{i\in \ZZ}$ in $GL(d_0,\RR)$, $SL(d_0,\RR)$ or $Sp(d_0,\RR)$ bounded by $C$, with period $\pi$
and no $N$-dominated splitting, there exists an $\eps$-path of perturbations $(A_i(t))_{t\in [0,1]}$ with period $\ell$, a multiple of $\pi$, such that:
\begin{itemize}
\item[--] $\cA(t):=A_\ell(t)\dots A_1(t)$ is hyperbolic for any $t\in[0,1]$,
\item[--] $t\mapsto |\det(\cA(t)_{|E^s})|$ and $t\mapsto |\det(\cA(t)_{|E^u})|$ are constant in $t\in[0,1]$,
\item[--] the stable (resp. unstable) eigenvalues of $\cA(1)$ have the same moduli.
\end{itemize}
\end{theorem}

\begin{remark}
In \cite{BochiBonatti} a stronger statement is proven for $GL(d_0,\RR)$ and $SL(d_0,\RR)$.  Namely,
there exists $T\geq 1$ which only depends on $d_0,C,\eps,$ and  $N$, such that
any multiple $\ell\geq T$ of $\pi$ satisfies the conclusion of Theorem~\ref{t.BochiBonatti}.
In particular, if the period $\pi$ is larger than $\ell$, one can choose $\ell=\pi$.
We do not know if this uniformity holds in $Sp(d_0,\RR)$.
\end{remark}

\subsection{Homoclinic tangencies}
The following theorem has been proved by Gourmelon in~\cite[Theorem 3.1]{Gourmelon2010} and~\cite[Theorem 8]{Gourmelon2014}
in the dissipative case.
\begin{theorem}[Homoclinic tangency]\label{t.gourmelon}
For any {$d_0\geq1$}, $C>1$, $\eps>0$, there exist $N,T\geq 1$ with the following property.
Consider
\begin{itemize}
\item[--] a diffeomorphism $f\in \Diff^1(M)$ {of a $d_0$-dimensional manifold $M$}
such that the norms of $Df$ and $Df^{-1}$ are bounded by $C$,
\item[--] a periodic saddle $O$ with period larger than $T$ such that {the splitting defined by the stable and unstable bundles is not an $N$-dominated splitting,} and
\item[--] a neighborhood $V$ of $O$.
\end{itemize}
Then there exists an $(\eps, V, O)$-perturbation $g$ of $f$ and $p\in O$ such that
\begin{itemize}
\item[--] $W^s(p)$, $W^u(p)$ have a tangency $z$ whose orbit is contained in $V$, and
\item[--] the derivative of $f$ and $g$ coincide along $O$.
\end{itemize}

Moreover if $O$ is homoclinically related to a periodic point $q$ for $f$, then the perturbation $g$
can be chosen to still have this property. If $f$ preserves a volume or a symplectic form, one can choose $g$ to preserve it also.
\end{theorem}

\begin{remark}\label{r.gourmelon}
As a consequence, one can also obtain a transverse intersection $z'$ between $W^s(O)$ and $W^u(O)$,
whose orbit is contained in $V$. This implies that $O$ belongs to a horseshoe of $g$ which is contained in $V$.
\end{remark}

\subsection{Approximation by periodic orbits with control of the Lyapunov exponents}
Ma\~n\'e's ergodic closing lemma {\cite{Mane82}} allows the approximation of any
ergodic measure by a periodic orbit, after a $C^1$-small perturbation.
In~\cite[Proposition 6.1]{ABC} a control of the Lyapunov exponents is added
(the proof given in the dissipative setting is still valid in the conservative cases).

\begin{theorem}[Ergodic closing lemma]\label{t.ergodic-closing}
For any $C^1$-diffeomorphism $f$ and any ergodic measure $\mu$,
there exists a small $C^1$-perturbation $g$ of $f$ having a periodic orbit $O$ that is close to
$\mu$ for the vague topology and whose Lyapunov exponents (repeated according to multiplicity)
are close to the Lyapunov exponents of $\mu$ for $f$.
If $f$ preserves a volume or symplectic form, $g$ also does.
Moreover $f$ and $g$ coincide outside a small neighborhood of the support of $\mu$.
\end{theorem}

\begin{corollary}\label{c.ergodic-closing}
For any generic diffeomorphism $f$
in $\Diff^1(M)$ or $\Diff^1_\omega(M)$,
and any ergodic measure $\mu$,
there exists a sequence of periodic orbits $(O_n)$ which converge to $\mu$
in the weak star topology and whose Lyapunov exponents (repeated according to multiplicity) converge to those of $\mu$.
\end{corollary}

This technique for controlling the exponents
(see also the proof of~\cite[Theorem 3.10]{ABC}) allows one to spread
a periodic orbit in a hyperbolic set and gives the next result.

\begin{proposition}\label{p.control-exponents}
For any generic diffeomorphism $f$
in $\Diff^1(M)$ or in $\Diff^1_\omega(M)$,
any horseshoe $K$, any periodic orbit $O\subset K$ and any $\eps>0$,
there exists a periodic orbit $O'\subset K$
that is $\eps$-close to $K$ in the Hausdorff distance
and whose collection of Lyapunov exponents is $\eps$-close to those
of $O$.
\end{proposition}

\section{Proof of the Main Theorem}\label{s.main}

We now state and prove a more detailed version of the Main Theorem (recall Definitions \ref{def-TNweak} and \ref{def-evNpert}).

\begin{theorem}\label{t.newhouse}
Given any integer $d_0\geq2$ and numbers  $C>0$ and $\eps>0$, there exist integers $N,T\geq1$ with the following property.
For any closed $d_0$-dimensional Riemannian manifold $M$, any diffeomorphism $f\in\Diff^1(M)$ with $\Lip(f),\Lip(f^{-1})\leq C$,  if $p$ is a $T,N$-weak periodic point and $V$ is a neighborhood of the orbit of $p$, then
there exists an $(\eps,V,\cO(p))$-perturbation $g$ such that
\begin{enumerate}
\item[--]
for any $x$ in $\mathcal{O}(p)$ we have $Df(x)=Dg(x)$ and
\item[--] $g$ contains a horseshoe $K\subset V$ such that $p\in K$ and
 $
   h_\top(g,K) \geq \Delta(g,p).$
\end{enumerate}

Moreover, if the orbit of $p$ is homoclinically related to a hyperbolic periodic point $q$ of $f$, one can choose $g$ such that the orbits of $p$ and $q$ are still homoclinically related.  Finally, 
if $f$ preserves a volume or symplectic form, then $g$ can be chosen to preserve the volume or symplectic form.  \end{theorem}

The strategy is similar to that of Newhouse in \cite{Newhouse}.  First, we construct a suitable homoclinic tangency using the lack of a strong dominated splitting.
Second, we create a horseshoe by suitable oscillations in a stable-unstable $2$-plane.  To exploit all the exponents we realize a circular permutation of the coordinates through the transfer map describing the homoclinic tangency.

\begin{remark}\label{r.sharp}
The theorem above is optimal in the following sense. It builds a diffeomorphism $g_0$ arbitrarily $C^1$-close to $f$ with $h_\top(g_0)\geq\Delta(f)$. However, 
 $$
  \limsup_{g\to f} h_\top(g)\leq \sup_{\mu\in \Probe(f)}\Delta(f,\mu),
 $$
as a consequence of Ruelle's inequality (see Sec. \ref{ss.periodichorseshoes}) and of the upper semicontinuity of  the right hand side  (see Lemma \ref{l-usc1}).
 \end{remark}

\subsection{A linearized homoclinic tangency}\label{s.preparation}

The first step in the proof of Theorem \ref{t.newhouse} is the creation of a homoclinic tangency that circularly permutes all the eigendirections of the periodic orbit, see Figure~\ref{f.permute}. 

\begin{proposition}\label{p.preparation}
For all $d_0\geq2,C,\eps>0$, there exist integers $N,T\geq1$ with the following property. 
For any diffeomorphism $f:M\to M$ with $\Lip(f),\Lip(f^{-1})\leq C$ of a $d_0$-dimensional closed manifold $M$ and any $T,N$-weak periodic saddle $p$ and any neighborhood $V$ of $\cO(p)$,
there exist
\begin{itemize}
\item[--]
an $(\eps,V, \cO(p))$-perturbation $g$ such that $Df(x)=Dg(x)$ for each $x\in \cO(p)$,
\item[--] a periodic orbit $O$ 
of $g$ such that $O\cup \cO(p)$ is included in a horseshoe in $V$ and satisfies
$\Delta^+(g,O)=\Delta^+(f,p)$ and $\Delta^-(g,O)=\Delta^-(f,p)$,
\item[--] a chart $\chi\colon W\to \RR^{d_0}$ with $O\subset W$, and
\item[--] a homoclinic point $\tau$ of $O$ for $g$, whose orbit is included in $V$, and an integer $m\geq 1$
\end{itemize}
such that the following hold:
 \begin{enumerate}
 \item $\chi\circ g\circ\chi^{-1}$ is linear on its domain and is diagonalizable;
more precisely, its restriction to
$\RR^k\times \{0\}^{d_0-k}$ is a homothety by a factor $\exp(\Delta^-(f,O)/k)<1$
and its restriction to $\{0\}^k\times \RR^{d_0-k}$
is a homothety by a factor $\exp(\Delta^+(f,O)/(d_0-k))>1$;
    \item $g^{-n}\tau\in W$ and $g^{n+m}\tau\in W$ for all $n\geq0$; and
  \item $\chi\circ g^{m}\circ\chi^{-1}$ is linear in a neighborhood of $\tau$ and satisfies
  $$\forall i=1,\dots,d_0\qquad D(\chi\circ g^{m}\circ\chi^{-1})(\tau) E_i=E_{i+1}$$
where $E_i:=\{0\}^{i-1}\times \RR\times \{0\}^{d_0-i}$  and $E_{{d_0}+1}:=E_1$.
 \end{enumerate}
Moreover, if $\cO(p)$ is homoclinically related to a  periodic point $q$ of $f$, then one can choose $g$ such that $\cO(p)$ and $q$ are homoclinically related for $g$.
Finally, if $f$ preserves a volume or a symplectic form, then one can choose $g$ to preserve the volume or symplectic form.
\end{proposition}

\begin{figure}
\begin{center}
\includegraphics[width=0.7\textwidth]{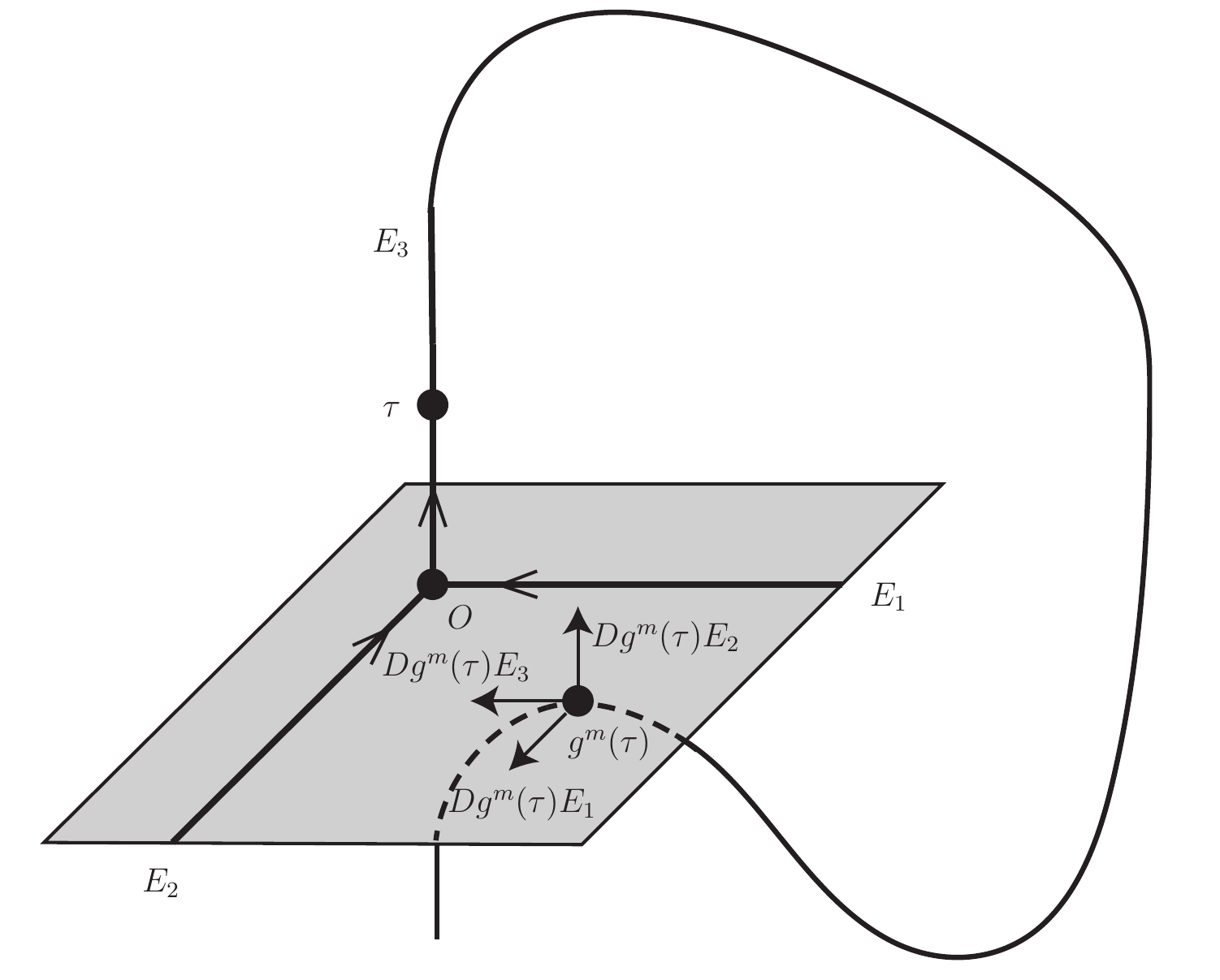}
\caption{Permutation of coordinates
}\label{f.permute}
\end{center}
\end{figure}

We first prove Proposition~\ref{p.preparation} in the dissipative case. The
adaptation to the volume-preserving and symplectic cases will be explained after.

\subsubsection{Dissipative setting}
Let $d_0,C$ and $\eps$ as in the statement of the proposition. Let $f\in\Diff^1(M)$ and let $p$ be a $T,N$-weak periodic saddle as above.

From Theorem~\ref{t.gourmelon}, we get integers $N,T\geq1$ (depending only on $d_0,C,\eps$)
such that for any $g\in B_{C^1}(f,\eps)$ and any $N,T$-weak periodic orbit $O$ of $g$, there exists an $\eps/10$-perturbation
$g'$ of $g$ which exhibits a homoclinic tangency.
Maybe after increasing $N$ (still depending only on $d_0,C,\eps$), Theorem~\ref{t.BochiBonatti} gives a perturbation of a repetition of the linear cocycle $Dg_O$ with all the stable (resp.\ unstable) eigenvalues with the same modulus.

Consider $p,V$ as in the statement of the proposition
and let $k$ be the stable dimension of $p$ and ${d_0}=\dim(M)$.
The proof splits into three perturbations. 
Each of them is supported in $V$ and preserves the orbit of $p$.
One takes a chart of a neighborhood of $\cO(p)$ so that the tangent maps at nearby points may be compared.

\medbreak
\noindent{\bf Step 1. Small horseshoes.}
\noindent{\it There exists an $(\eps/3,V,\cO(p))$-perturbation $g_1$ of $f$ in $\Diff^1(M)$
having a horseshoe $K_0\subset V$ that contains $p$.
Moreover, points $y\in K_0$ close to $p$ satisfy 
$$Dg_1(g_1^{i}(y))=Df(f^i(p))\textrm{ for any }i\in \ZZ,$$
and if $\cO(p)$ is homoclinically related to an orbit $\cO(q)$ of $f$, then one can choose $g_1$ such that $\cO(p)$ and $\cO(q)$ are still homoclinically related.}
\medbreak

From Theorem~\ref{t.gourmelon}
and Remark~\ref{r.gourmelon}, we can obtain an $(\eps/6,V,\cO(p))$-perturbation $g_0$ of $f$
with a horseshoe in $V$ which contains $p$. Moreover $Df$ and $Dg_0$ coincide on $\cO(p)$, and
Theorem~\ref{t.linearize} shows that $g_0$ may be chosen so that it is linear
in a neighborhood of the orbit $\cO(p)$.

Let us consider a point $z$ in the horseshoe which is homoclinic to $\cO(p)$.
Large (positive and negative) iterates of the point $z$ belong to the linearizing
neighborhood of $\cO(p)$.
Since $d_{C^1}(f,g_0)<\eps/6$ (and assuming $V$ small), Theorem~\ref{t.linearize} gives a further perturbation $g_1$
localized at the (finitely many) others iterates of $z$,
such that
\begin{itemize}
\item[--]  $d_{C^1}(g_0,g_1)<\eps/6$, 
\item[--] $g_1$ coincides with $g_0$ on the orbit of $z$, and
\item[--] for any $i\in \ZZ$ the map
$Dg_1$ is constant on a neighborhood of $g_1^{i}(z)$
where it coincides with $Df(f^i(p))$.
\end{itemize}
We define $K_0$ as the maximal invariant set in a neighborhood of
$$\cO(p)\cup\{g_1^i(z), i\in \ZZ\}.$$

Recall that Theorem~\ref{t.gourmelon} allows us to choose the perturbation so that it preserves the homoclinic relation with an orbit $\cO(q)$.
The perturbations to linearize near $\cO(p)$ and from the two applications of Theorem \ref{t.linearize} can be assumed to be sufficiently $C^1$-small given
the heteroclinic connection between $\cO(p)$ and $\cO(q)$.
Hence, these orbits are still homoclinically related for $g_1$.

\medbreak
\noindent {\bf Step 2. Stable and unstable homotheties.}
\noindent{\it There exists a $\pi(x)$-periodic point $x\in K_0$
such that for any neighborhood $V_x$ of $\cO(x)$
there exists a diffeomorphism $g_2$ satisfying the following:
\begin{itemize}
\item[--] $d_{C^1}(g_2,g_1)<\eps/3$, and $g_2=g_1$ on $\cO(x)\cup (M\setminus V_x)$,
\item[--] $x$ and $p$ belong to a horseshoe of $g_2$ contained in $V$,
\item[--] $\Delta^+(g_2,x) =  \Delta^+(f,p)$, $\Delta^-(g_2,x) =  \Delta^-(f,p)$ and
\item[--] $Dg_2^{\pi(x)}(x)$ is conjugate to 
$\Lambda^s\Id_{\RR^k}\otimes\Lambda^u\Id_{\RR^{{d_0}-k}}$,
where $$k\log \Lambda^s=\pi(x)\Delta^-(f,p) \text{ and } ({d_0}-k)\log \Lambda^u=\pi(x)\Delta^+(f,p).$$
\end{itemize}
Furthermore, if $\cO(p)$ is homoclinically related to an orbit $\cO(q)$ for $f$, then one can choose $g_2$ such that $\cO(p)$ and $\cO(q)$ are still homoclinically related.}
\medbreak

We will introduce a periodic point $x\in K_0$ close to $p$ with large period $\ell$ and
build $\eps/6$-paths of linear perturbations $(A_i(t))_{t\in [0,1]}$, $1\leq i\leq \ell$, at $\cO(x)$
such that the composition $\cA(t)=A_{\ell}(t)\circ\dots\circ A_1(t)$ is hyperbolic for each $t$.
Then Theorem~\ref{l.gourmelon} will realize $(A_i(1))_{i=1}^{\ell}$ by a perturbation $g_2$ of $g_1$ so that $Dg_2(g_2^i(x))=A_i(1)$ for each $i$ while preserving the orbit of $x$ and the homoclinic relation of $x$ to $p$. It remains to choose $x$ and define the paths $(A_i(t))$. We do it in three steps.

\smallbreak\noindent\emph{1. Equal modulus of all stable, resp.\ unstable, eigenvalues.}   
The lack of strong dominated splitting on the orbit of $p$ and Theorem~\ref{t.BochiBonatti} gives a multiple $n$ of the period of $p$ and paths $t\mapsto(B_j(t))_{1\leq j\leq n}$ such that  $B_j(0)=Dg_1(g^j_0(p))$ and the final composition $ \cB(1)=
 B(n)\circ\dots\circ B(1)$ has $k$ stable eigenvalues with the same modulus and ${d_0}-k$ unstable eigenvalues with the same modulus.
Moreover, the product of the moduli along the stable and the unstable spaces are unchanged during the deformation.
The stable moduli are thus equal to $\exp(n\Delta^-(f,p)/k)$ and the unstable ones to
$\exp(n\Delta^+(f,p)/({d_0}-k))$.

\smallbreak\noindent\emph{2. Simple eigenvalues with rational arguments.} Using Proposition~\ref{p.simple}, one builds paths $(\bar B_j(t))$,  arbitrarily close to the paths $( B_j(t))$, and
such that the composition $\bar \cB(1)= \bar B_{n}(1)\circ\dots\circ \bar B_1(1)$ has simple eigenvalues,
of the form $\lambda=\Lambda^\sigma e^{i\theta}$, where $\theta$ is rational and $\Lambda^\sigma\in \{\Lambda^s,\Lambda^u\}$.

\smallbreak\noindent\emph{3. Product of homotheties for a higher order periodic cocycle.}
We let $a\geq 1$ be an integer such that $\lambda^a$ is real and positive for any eigenvalue $\lambda$ of $\bar \cB(1)$, and consider a periodic point $x$ in $K_0$ close to $p$ whose period $\ell$ is a multiple of $an$.
Note that by the construction of $g_1$ we know that $Dg_1(g_1^i(x))$ coincides with $Dg_1(g_1^i(p))=Df(f^i(p))$.

The corresponding paths $(A_j(t))=(\bar B_{j\mathrm{mod} n}(t))$ for $1\le j\le \ell$ are such that the composition $\cA(1)=A_{\ell}(1)\circ\dots\circ A_1(1)$ has a stable, real and positive eigenvalue with multiplicity $k$ and an unstable, real and positive eigenvalue with multiplicity ${d_0}-k$. Moreover, it is a product of two homotheties.
\medbreak

\medbreak
\noindent {\bf Step  3. A homoclinic tangency.}
\noindent{\it There exists an $\eps/3$-perturbation $g_3$ of $g_2$
satisfying properties 1, 2, 3 of Proposition~\ref{p.preparation}.}
\medbreak

Starting from the diffeomorphism $g_2$ and its periodic orbit $O=\cO(x)$,
we use Theorem \ref{t.gourmelon} and then \ref{t.linearize} to perform an $\eps/6$-perturbation which 
creates a homoclinic tangency $z$ for $x$
whose orbit is contained in $V$ and again linearizes the dynamics in a neighborhood of the orbit of $\cO(x)$.  Furthermore, both perturbations keep the tangent map along $\cO(x)$ unchanged and
the homoclinic relation to $\cO(p)$.
Also, the perturbation is supported in an arbitrarily small neighborhood of $\cO(x)$ that is contained in $V$.

Since $Dg_2^{\pi(x)}(x)$ is a homothety along the stable and unstable spaces,
there exists an invariant decomposition $E_1\oplus\dots\oplus E_{d_0}$ into one-dimensional subbundles of $TM$ above the orbit of $x$ such that
$E_1\oplus \dots\oplus E_k$ coincides with the stable bundle,  and $E_{k+1}\oplus \dots\oplus E_{d_0}$ coincides with the unstable bundle.

One may assume that the intersection between the stable and unstable manifolds of $x$ at $z$ is one-dimensional.
One denotes by $F$ the tangent line to this intersection at $z$.
One then chooses ${d_0}-1$ lines $F_1^s,\dots,F_{k-1}^s$, $F^u_1,\dots,F^u_{{d_0}-k-1}$, $F'$ in $T_zM$ such that
\begin{itemize}
\item[--] $F\oplus F_1^s\oplus\dots\oplus F_{k-1}^s=T_z W^s(x)$,
\item[--] $F_1^u\oplus\dots\oplus F_{{d_0}-k-1}^u\oplus F=T_z W^u(x)$, and
\item[--] $F\oplus F_1^s\oplus \dots\oplus F_{k-1}^s\oplus F'\oplus F^u_1\oplus \dots\oplus F^u_{{d_0}-k-1}=T_zM$.
\end{itemize}

Using Theorem~\ref{t.linearize} near a finite piece of the forward orbit of $z$ we perform a local perturbation of $g_2$ with the following properties.
\begin{itemize}
\item[--] After a large positive iterate, the orbit of $z$ remains in the linearizing chart and the dynamics act on
the forward iterates of the space $F\oplus F_1^s\oplus\dots\oplus F_{k-1}^s$ as a homothety. Since $SL(k,\RR)$ is connected, we can  perturb the map so that after a large forward iterate, the image of the spaces $F, F_1^s,\dots, F_{k-1}^s$
coincide with $E_1,\dots, E_k$ respectively.
\item[--] The space $F'\oplus F_1^u\oplus\dots\oplus F_{{d_0}-k-1}^u$ is transverse to the stable space at $z$, hence,
it converges under forward iteration to the unstable space.
After a perturbation, a large forward iterate sends
this space on the unstable direction $E_{k+1}\oplus\dots\oplus E_{{d_0}}$.
Then the tangent dynamics acts by homothety inside this space.
Therefore, a perturbation near finitely many further iterates will ensure that
the spaces $F', F_1^u,\dots, F_{{d_0}-k-1}^u$ are eventually mapped
on  $E_{k+1},\dots, E_{{d_0}}$ respectively.
\end{itemize}
Thus, the splitting 
$$F\oplus F_1^s\oplus\dots\oplus F_{k-1}^s\oplus F'\oplus  F_1^u\oplus\dots\oplus F_{{d_0}-k-1}^u$$
 at $z$ is eventually mapped to $E_{1}\oplus\dots\oplus E_{{d_0}}$ by any sufficiently large iterate.
Working in a similar way along the backward orbit, one can ensure that after a large backward iterate,
the splitting 
$$F_1^s\oplus\dots\oplus F_{k-1}^s\oplus F'\oplus F_1^u\oplus\dots\oplus F_{{d_0}-k-1}^u\oplus F$$ 
is mapped on
$E_1\oplus \dots\oplus E_{d_0}$.

Possibly after reducing $W$ we know $\chi\circ g\circ\chi^{-1}$ is locally linear on its domain.
One can choose the coordinate axis $\{0\}^{j-1}\times \RR\times \{0\}^{d_0-j}$ to be preserved and identified with $E_j$.
Theorem~\ref{t.linearize} gives an arbitrarily small perturbation at the finitely many iterates $f^n(z)$, $|n|\leq\ell$, for some large $\ell\geq1$,
such that the new diffeomorphism $g_3$ satisfies:
$g_3^n(z)\in W$ for any $n$ satisfying $|n|\geq \ell$
and $\chi\circ g_3^{2\ell}\circ \chi^{-1}$ is linear in a neighborhood of $g_3^{-\ell}(z)$.
By construction:
\begin{itemize}
\item[--] $Dg_3(f^{j}(z))$ preserves each direction $E_1,\dots, E_{d_0}$ for $|j|\geq \ell$.
\item[--] $Dg_3^{2\ell}(f^{-\ell}(z))$ sends $E_1$,\dots, $E_{d_0}$ on $E_2$,\dots, $E_{d_0}$, $E_1$.
\end{itemize}
Setting $\tau=g_3^{-\ell}(z)$ and $m=2\ell$ gives all the required properties.
\medbreak

\medbreak
\noindent {\bf Conclusion in the dissipative case.}
Note that the perturbations $g_2,g_3$ performed in the last two steps are supported in a small neighborhood of $\cO(x)$; hence, these perturbations may be chosen to preserve a homoclinic relation between $\cO(p)$ and another periodic point $q$. 
By taking $g=g_3$, Proposition~\ref{p.preparation}
is now proved in the dissipative case.

\subsubsection{The conservative case}

\paragraph{ The volume-preserving case.}
In this setting, it is enough to use the volume-preserving cases of Theorems \ref{t.linearize}, \ref{l.gourmelon}, \ref{t.BochiBonatti}, and \ref{t.gourmelon} and to note that linear perturbations (of the differential) can  be made volume-preserving in an obvious way.

\paragraph{The symplectic case.}
When $f$ preserves a symplectic form $\omega$, only the last step has to be modified.
We denote $d_0=2d$ and the stable and unstable dimensions coincide: $k=d_0-k=d$.
We refer to Section~\ref{s.prelim} for standard definitions and basic results about symplectic linear algebra.

Let us choose any line $F'$ transverse to $T_zW^s(x)+T_zW^u(x)$.

\begin{lemma}
One can choose $F^s_1,\dots,F^s_{d-1},F^u_1,\dots,F^u_{d-1}$
such that the planes $F\oplus F'$, $F^s_1\oplus F^u_1$,\dots, $F^s_{d-1}\oplus F^u_{d-1}$
are symplectic and $\omega$-orthogonal to each other.
\end{lemma}
\begin{proof}
Note that $T_zW^s(x)$ and $T_zW^u(x)$ are both Lagrangian spaces which do not contain $F'$. Consequently the symplectic complement $(F')^\omega$ of $F'$
contains neither $T_zW^s(x)$ nor $T_zW^u(x)$.
Note also that $F^\omega$ contains, hence coincides with, $T_zW^s(x)+T_zW^u(x)$;
in particular, $(F')^\omega$ is transverse to $F$.
The spaces 
$$H^s:= (F')^\omega\cap T_zW^s(x)\textrm{ and }H^u:= (F')^\omega\cap T_zW^u(x)$$ thus satisfy:
$T_zW^s(x)=F\oplus H^s$ and $T_zW^u(x)=F\oplus H^u$.
One then inductively chooses $d-1$ lines $F^s_1,\dots,F^s_{d-1}$ which span $H^s$ and $d-1$ lines $F^u_1,\dots,F^u_{d-1}$ which span $H^u$: one first takes $F_1^s\subset H^s$ and $F_1^u\subset H^u$ such that $F^u_1$ is transverse to $(F^s_1)^\omega$;
the plane $F^u_1\oplus F^s_1$ is thus symplectic. One then repeats the construction replacing $H^s$ and $H^u$ by
$H^s\cap (F_1^u)^\omega$ and $H^u\cap (F_1^s)^\omega$.
\end{proof}

Since $E^s(x)$ and $E^u(x)$ are two transverse Lagrangian spaces, one can choose
the decomposition $E_1\oplus \dots\oplus E_{2d}$ such that each plane
$E_i\oplus E_{i+d}$, $1\leq i\leq d$ is symplectic and $\omega$-orthogonal to the other ones.
After a large iterate, $F\oplus F^s_1\oplus\dots\oplus F^s_{d-1}$ is mapped on the stable space $E^s(x)$
and $F'\oplus F^u_1\oplus\dots\oplus F^u_{k-1}$ is sent to a Lagrangian space close to the unstable space $E^u(x)$.
After composing by a symplectomorphism close to the identity, the second space is sent to $E^u(x)$ and the first still coincides with $E^s$.
Indeed if $A$ is a linear map whose restriction to $E^s(x)$ is the identity and which sends $F'\oplus F^u_1\oplus\dots\oplus F^u_{k-1}$
to $E^u(x)$, then writing that these spaces are Lagrangian gives that $A$ is symplectic.

The connected component of the identity of the group of symplectic maps which preserve the Lagrangian space $E^s(x),E^u(x)$
acts transitively on the decompositions of $E^s(x)$ (one can identify $E^s(x)\oplus E^u(x)$ with the decomposition $\RR^d\times \RR^d$
endowed with the standard symplectic form, then any product $A=B^T\times B^{-1}$, where $B\in GL(d,\RR)$ is symplectic).
Recall also that $Dg_2(x)$ acts as a homothety along the space $E^s(x)$. One can thus decompose $A$ as a product of linear maps close to the identity
and realize it along an arbitrarily large piece of the orbit.
Hence, we obtain symplectic perturbations at finitely many forward iterates ensuring that the decomposition
$F\oplus F^s_1\oplus\dots\oplus F^s_{d-1}$ is eventually mapped to $E_1\oplus\dots\oplus E_d$ as required.
Note that the image of $F'$ is in the symplectic complement of $E_2\oplus\dots \oplus E_d$,
hence coincides with $E_{d+1}$. With the same argument one checks that the already chosen perturbation must send the decomposition $F'\oplus F^u_1\oplus\dots\oplus F^u_{d-1}$ to $E_{d+1}\oplus \dots \oplus E_{2d}$.

One perturbs similarly along the backward orbit so that $$F_1^s\oplus\dots\oplus F_{d-1}^s\oplus F'\oplus F_1^u\oplus\dots\oplus F_{d-1}^u\oplus F$$ is mapped on
$E_1\oplus \dots\oplus E_{2d}$ as before. This concludes the proof of Proposition \ref{p.preparation} in the symplectic case.

\subsection{The horseshoe}

We now prove Theorem \ref{t.newhouse} by constructing the horseshoe with the desired entropy. Let $d_0,C,\eps$ and $f\in\Diff^1(M)$ be as in the statement.
Fixing integers $N$ and $T$ as in Proposition~\ref{p.preparation} we let $\cO(p)$ be a $T,N$-weak periodic orbit,
and let $V$ be a neighborhood of $\cO(p)$.

Let $g\in B_{C^1}(f,\eps/2)$ be the perturbation and $O$ be the hyperbolic periodic orbit given by Proposition~\ref{p.preparation}.
Using Theorem~\ref{l.gourmelon}, one can slightly increase $\Delta(O)$ by an arbitrarily small perturbation so that
$\Delta(O)>\Delta(\cO(p))$.
This compensates for the arbitrarily small loss encountered when converting exponents to entropy.  Furthermore, this can be done while keeping the homoclinic tangency and the given homoclinic relations.
The properties stated in Proposition~\ref{p.preparation} are still satisfied with the chart $\chi\colon W\to \RR^{d_0}$,
the coordinate axis $E_j:=\RR^{j-1}\times \RR\times \{0\}^{d_0-j}$, the homoclinic point $\tau$ and the integer $m$.
\medskip

The proof of Theorem~\ref{t.newhouse} will be concluded by applying the following proposition.
Using a sufficiently small $C^1$-perturbation, one preserves any finite number of homoclinic relations between $O$ and other periodic orbits.

\begin{proposition}\label{p.newhouse}
Given $g\in \Diff^1(M)$, $\eps'>0$, a hyperbolic periodic orbit $O$, a chart $\chi\colon W\to \RR^{d_0}$ with $O\subset W$,
a homoclinic point $\tau$ of $O$ and $m\geq 1$ satisfying items 1, 2, 3 of Proposition~\ref{p.preparation},
there exist an $(\eps', W, O)$-perturbation $h$ of $g$ and a (linear and conformal)
horseshoe $K$ in the $\eps'$-neighborhood of
$O\cup\operatorname{orbit}(\tau)$ such that
$h_\top(h,K)\geq \Delta(g,O)-\eps'$ and
$K$ and $O$ are homoclinically related.
Moreover, the perturbation preserves a volume or symplectic form if $g$ does.
\end{proposition}

\begin{proof}
By increasing $m$, one can assume that $\tau$ and $g^m(\tau)$ belong respectively to the unstable manifold and the stable manifold
of the \emph{same point} of the orbit $O$. 
We denote by $k$ the number of negative exponents of $O$.

{Our construction will depend on parameters $\rho,\eta>0$, $\ell,n,L\in\NN$.
We will specify a rectangle $R$ close to $g^{n+m}(\tau)$ and a perturbation
supported near
$g^{n+m-1}(\tau)$, $g^{m-1}(\tau)$ and $g^{-n-1}(\tau)$ performing:
\begin{itemize}
\item[--] $L$ oscillations near $g^{m-1}(\tau)$ creating the horseshoe and
\item[--] small shears near $g^{-n-1}(\tau)$ and $g^{n+m-1}(\tau)$ making  map $g^{\ell+m+2n}$ linear on the horseshoe.
\end{itemize}

More precisely, we  choose first a small number $\rho>0$ (only depending on $g$) controlling the
$C^0$-size of the perturbation, then $\eta>0$ small enough and a large integer multiple $n$ of $\pi(O)$ in order to make the $C^1$-size
of the perturbation smaller than $\eps'$.
Also the diameter $\delta$ of $R$ will be directly related to $n$.
At last, one chooses large integers $\ell$ (a multiple of $\pi(O)$) and $L$
in order to control the size of the iterates of $R$ and the entropy of the horseshoe.}

\paragraph{The rectangle $R$.}
We pick a point $z^s$ arbitrarily close to $g^{m+n}(\tau)$ whose iterates $g^j(z^s)$, $0\leq j\leq \ell$, belong to $W$
and such that $g^{\ell+n}(z^s)$ is arbitrarily close to $\tau$. Observe that $\tilde z^s:=g^{\ell+2n+m}(z^s)$ is close to $z^s$, exponentially in $\ell$.

Let $\lambda_1\leq \dots\leq \lambda_{d_0}$ be the Lyapunov exponents of $O$:
the linear map which coincides with $\chi\circ g\circ \chi^{-1}$ near $O$
expands $E_i$ by a factor $e^{\lambda_i}$.
The linear map which coincides with $\chi \circ g^m\circ \chi^{-1}$ near $\tau$ sends the axis $E_i$ to $E_{i+1}$
and expands it by some factor~$\mu_i$.
We choose the scale
$$\delta =C^{-1} e^{(\lambda_1+\lambda_{d_0})n}\rho,$$
where $C>0$ only depends on $g$ and bounds $\max(|\mu_i|,|\mu_i|^{-1})$.

Then, for $i=1,\dots,d_0$, one defines the following:
\begin{itemize}
\item[--] $\Lambda_i:=\mu_ie^{(\ell+n)\lambda_i+n\lambda_{i+1}}$ (indices to be understood modulo $d_0$),
\item[--]
$\delta_i:=\Lambda_1\cdots\Lambda_{i-1} \cdot \delta \text{ for }i=1,\dots,k$ (in particular $\delta_1:=\delta$),
\item[--] $\delta_i:=(\Lambda_i\cdots\Lambda_{d_0})^{-1} \cdot \delta \text{ for }i=k+1,\dots,{d_0}$, and
\item[--]    $R:=\chi(z^s)+Q(\delta_1,\dots,\delta_{d_0})$ where $Q(\delta_1,\dots,\delta_{d_0}):=\prod_{i=1}^{d_0} [-\delta_i,\delta_i]. $
\end{itemize}
Hence $\Lambda_1,\dots,\Lambda_k$, $\Lambda_{k+1}^{-1},\dots,\Lambda_{d_0}^{-1}$, are exponentially small in $\ell$ and $\delta_i\leq\delta$ for all $i$.

As $\delta$ is small enough (though independent of $\ell$), the $\ell+n$ first iterates of $\chi^{-1}(R)$ are contained in $W$
and the last one of these, very close to $\tau$, is contained in the linearization domain of $\chi\circ g^m\circ \chi^{-1}$.
The form of $g$ near $p$ and of $g^m$ near $\tau$ implies 
that
$$
 (\chi\circ  g^{\ell+m+2n}\circ\chi^{-1}) (R) = \chi(\widetilde z^s) +
   Q(\Lambda_{d_0}\delta_{d_0},\Lambda_{1}\delta_{1},\dots,\Lambda_{d_0-1}\delta_{d_0-1}).$$
Remark that $\Lambda_{i}\delta_{i}=\delta_{i+1}$ for $i\ne k$, and 
 $
 \Lambda_k\delta_k=\Lambda_1\dots\Lambda_k\cdot\delta = \Lambda_1\dots\Lambda_{d_0}\cdot \delta_{k+1}.
 $
 These numbers are all less than $\delta$.

\paragraph{The construction.} The perturbation of $g$ will be given as $h=\psi\circ g$ where $\psi$ coincides with the identity outside the
$\rho$-neighborhoods of $g^{-n}(\tau)$, $g^{m}(\tau)$, $g^{n+m}(\tau)$.
One will consider the coordinates defined by $\chi$
and centered at $z^s$, $g^\ell(z^s)$, $g^{\ell+n}(z^s)$,
$g^{\ell+n+m}(z^s)$ and $\tilde z^s$ respectively.
Let $(a_1,\dots,a_{d_0})=\chi(z^s)-\chi(\widetilde z^s)$.

We fix once and for all a smooth map $\Phi\colon \RR\to \RR$ such that
\begin{itemize}
\item[--] $\Phi(x)=0$ for $x\leq-1/2$ or $x\geq L-1/2$,
\item[--] $\Phi(x)=x$ for $|x|\leq 1/4$,
\item[--] $\Phi(x+j)=\Phi(x)$ for $x\in[-1/2,1/2]$ and $j=0,1,\dots,L-1$.
\end{itemize}
Note that its $C^0$ and $C^1$ sizes do not depend on the choice of $L$.

On the ball $B(g^{-n}(\tau),\rho)$,
the map $\psi$ creates a first shear: denoting by $(x_i)$ the coordinates of a point $x$,
and by $(x'_i)$ the coordinates of the image $x'=\psi(x)$, the map is defined
with the following formulas:
\begin{equation}\label{e.pertg-n}
x'_i=x_i \text{ if } i\neq d_0\quad \text{ and }\quad  x'_{d_0}=x_{d_0}+\eta^{-1}e^{n(\lambda_{k}-\lambda_{d_0})} \frac{\mu_k}{\mu_{d_0}}.x_k.
\end{equation}
On the ball $B(g^{m}(\tau),\rho)$, we define the oscillations by defining $\psi$ through:
 \begin{equation}\label{e.pertgm}
 x'_i=x_i \text{ if } i\neq k+1\quad \text{ and } \quad
x'_{k+1} = x_{k+1}-\eta e^{-n\lambda_1}\frac{\delta}{L}\Phi\left(e^{n\lambda_1}\frac{L x_{1}}{\delta }\right).
\end{equation}
Finally, on $B(g^{n+m}(\tau),\rho)$ we create a second shear by defining $\psi$ through:
 \begin{equation}\label{e.pertgmn}
 x'_i=x_i+a_i \text{ if } i\neq 1\quad \text{ and } \quad x'_{1}=x_{1}+\eta^{-1}e^{n(\lambda_1-\lambda_{k+1})} \cdot x_{k+1}+a_{1}.
\end{equation}
These three maps are isotopic to the identity and
can be chosen conservative.
\medbreak

We ensure that the $C^1$-size of the perturbation is less than $\eps'$ by taking $\eta$ small and then $n$ large.
Note that these choices are independent of $\ell$ and $L$.
We pick an integer $L$ within a constant factor of the following upper bound:
\begin{equation}\label{e.choice-L}
L<\frac{\eta}8e^{n(\lambda_{k+1}-\lambda_{1})}\min({\Lambda^{-1}_{1}\cdots\Lambda^{-1}_{k}},{\Lambda_{k+1}\cdots\Lambda_{d_0}}).
\end{equation}

\begin{figure}[ht!]
\begin{center}
\includegraphics[width=140mm]{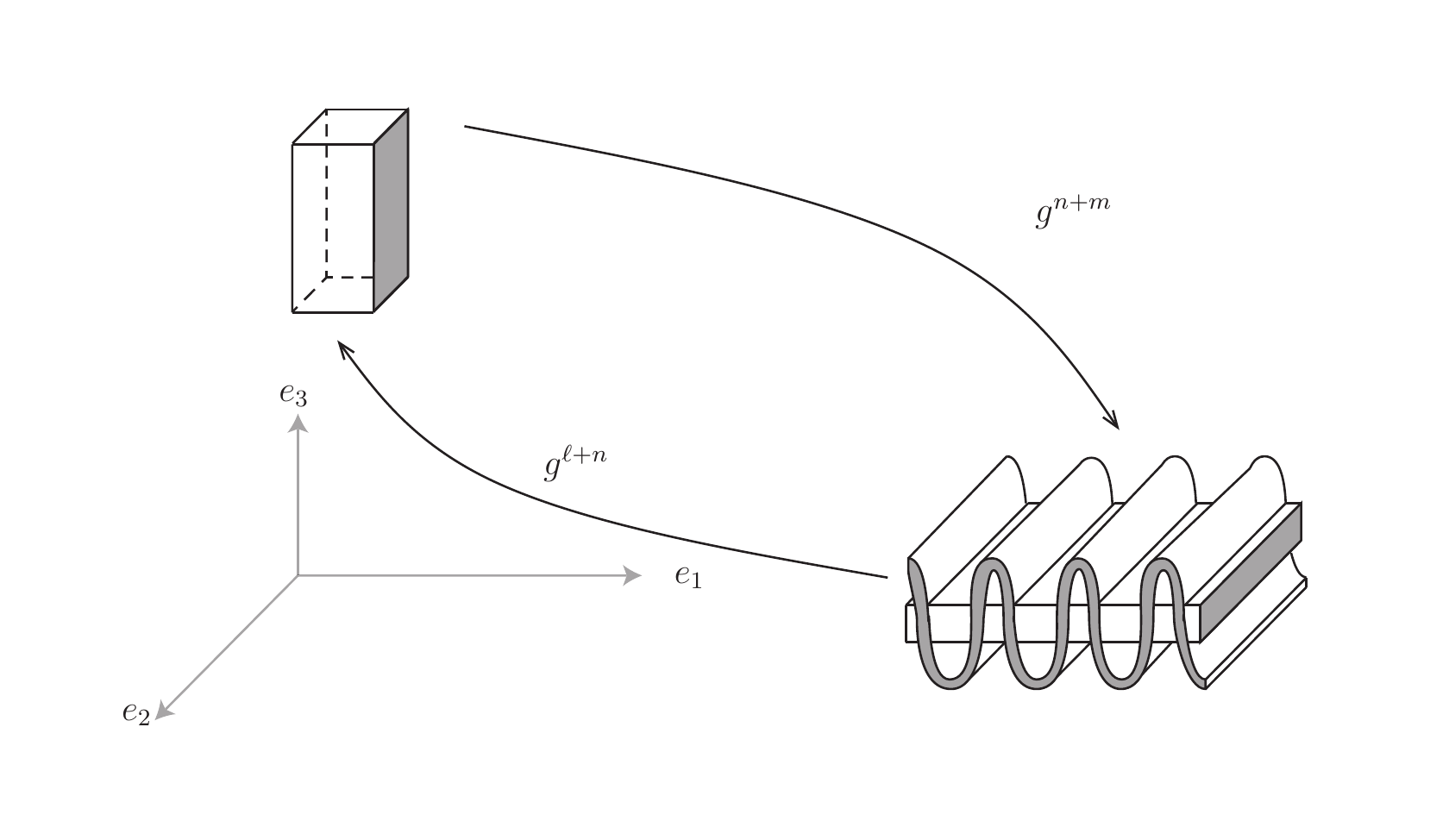}
\caption{{The rectangle $R$ and its images by $g^{\ell+n}$ and $g^{\ell+2n+m}$ when $d_0=3$ and $k=2$. 
One face (together with its images) is shaded gray.}}\label{f.hs1}
\end{center}
\end{figure}

Figure \ref{f.hs1} represents the rectangle $R$ and its images by the perturbation $g$.

\paragraph{Localization of the iterates of $R$.}
Observe that the iterates of the rectangle $R$ remain small.
More precisely,
all $g^j(R)$, for $0\leq j<\ell+n$,
lie in the linearization domain and
$g^{\ell+n}(R)$ lies in
the domain of eq. \eqref{e.pertg-n} since, for all $0\leq j\leq\ell$
 \begin{equation}\label{e.inclusion}
   \chi(g^j(R))=\chi(g^iz^s)+Q(e^{\lambda_1j}\delta_1,\dots,e^{\lambda_{d_0}j}\delta_{d_0}) \subset B(\chi(g^jz^s),\rho).
 \end{equation}
Indeed, $e^{\lambda_ij}\delta_i\leq\delta$  except possibly for $i=d_0$ where 
$$e^{\lambda_{d_0}j}\delta_{d_0}=e^{\lambda_{d_0}j}\Lambda_{d_0}^{-1}\delta \leq \mu_{d_0}^{-1}e^{-(\lambda_1+\lambda_{d_0})n}\delta\leq\rho.$$
In order to control the perturbation given by eq. \eqref{e.pertg-n}, note that 
$$e^{\lambda_k\ell}\delta_k/e^{\lambda_{d_0}\ell}\delta_{d_0}=e^{\lambda_k\ell}\Lambda_1\dots\Lambda_{k-1}\cdot \Lambda_{d_0}e^{-\lambda_{d_0}\ell}=e^{\lambda_k\ell}\Lambda_1\dots\Lambda_{k-1}\mu_{d_0}e^{(\lambda_{d_0}+\lambda_1)n} $$
is exponentially small in $\ell$.
Since $\ell$ is large once $n$ has been fixed, the scale of $g^\ell(R)$ along the $k^\text{th}$-direction is much smaller than scale along the $d_0^\text{th}$ direction.
Hence, the above inclusion~\eqref{e.inclusion} extends to $0\leq j\leq\ell+n+m$.

 Let $(x_i)$ be the coordinates of some point $x$ in $R$.
 Then the coordinates $(y_i)$ of $h^{m+n+\ell}(x)=g^{m+n}\circ\psi\circ g^\ell(x)$ near $g^{\ell+n+m}(z^s)$ are
 \begin{equation}\label{e-formulay}\begin{aligned}
   &y_1 = \mu_{d_0} e^{\lambda_{d_0}(n+\ell)}x_{d_0} + \eta^{-1}\mu_k e^{\lambda_k(\ell+n)} x_k\textrm{ and }\\
   &y_i \,= \mu_{i-1} e^{\lambda_{i-1}(n+\ell)} x_{i-1} \qquad (i\ne 1).
 \end{aligned}\end{equation}

Note that, in eq.~\eqref{e.pertgm} we have 
$$\Phi(e^{n\lambda_1}Ly_1/\delta)=(e^{n\lambda_1}L/\delta)y_1-j$$ 
if $|e^{n\lambda_1}(L/\delta)y_1-j|\leq 1/4$ for some $j\in\{0,\dots,L-1\}$.
This holds in particular if
 \begin{equation}\label{e.xwell}
   \left|\frac{L}\delta \Lambda_{d_0} x_{d_0}-j\right|<1/8  \text{ for some }j\in\{0,\dots,L-1\}.
 \end{equation}
Indeed from~\eqref{e.choice-L},
the quantity $|e^{n\lambda_1}\tfrac{L}\delta.\eta^{-1}\mu_ke^{\lambda_k(\ell+n)}x_k|$ is bounded by 
 $$
   e^{\lambda_1n}\frac L\delta \eta^{-1}\Lambda_k e^{-\lambda_{k+1} n}\cdot \delta_k 
     < \frac{1}{8\delta} (\Lambda_1\cdots\Lambda_k)^{-1} \cdot \Lambda_k(\Lambda_1\cdots\Lambda_{k-1}\delta) = \frac18.
 $$

Arguing as before, $g^j(h^{\ell+n+m}(R))$ has diameter
smaller than $\rho$ for $0\leq j\leq n$.

Under condition \eqref{e.xwell}, the coordinates $(z_i)$
of $h^{\ell+m+2n}(x)=\psi\circ g^n\circ\psi(y)$ satisfy, 
 \begin{equation}\label{e-formulaz}\begin{aligned}
   &z_1 = \eta^{-1} e^{-(\lambda_{k+1}-\lambda_1)n}\Lambda_k x_k + j\delta/L\\
   &z_{k+1} = - \eta e^{(\lambda_{k+1}-\lambda_1)n}(\Lambda_{d_0}  x_{d_0} - j \delta/L)\\
   &z_i = \Lambda_{i-1}x_{i-1}  \qquad (i\ne 1,k+1).
 \end{aligned}\end{equation}
Under condition $\Lambda_{d_0}x_{d_0}<-\delta/2L$, then the last term in eq.~\eqref{e.pertgm} vanishes
and
the coordinates $(z_i)$ are given by:
 \begin{equation}\label{e-formulazz}\begin{aligned}
   &z_1 = \Lambda_{d_0}x_{d_0} +{2}\eta^{-1}e^{-(\lambda_{k+1}-\lambda_1)n}\Lambda_k x_k\\
   &z_i = \Lambda_{i-1}x_{i-1}  \qquad (i\ne 1).
 \end{aligned}\end{equation}

\paragraph{Dynamics of the construction.}
For each $j\in\{0,\dots,L-1\}$ we define the subrectangle
$$
   R^s_j:=\chi(z^s)+\prod_{i=1}^{d_0-1} [-\delta_i,\delta_i]\times \left[\frac{(j-1/8)}{L}\delta_{d_0},\frac{(j+1/8)}{L}\delta_{d_0}\right]\subset R.
   $$
In particular, the condition \eqref{e.xwell} is satisfied and
$h^{\ell+m+2n}|R^s_j$ is given by eq.~\eqref{e-formulaz}. Its image is therefore the rectangle
$$R^u_j:=\chi(z^s)+
\bigg(\frac{j\delta_1}{L}+[-\delta'_1,\delta'_1]\bigg)\times
\prod_{i=2}^{d_0} [-\delta'_i,\delta'_i],$$
where $\delta'_i=\Lambda_{i-1}\delta_{i-1}=\delta_i$ for $i\not\in\{ 1,k+1\}$ and
$$\delta'_1=\eta^{-1}e^{n(\lambda_1-\lambda_{k+1})}\Lambda_k\delta_{k},\quad
\delta'_{k+1}=\frac\eta{8L} e^{n(\lambda_{k+1}-\lambda_1)}\delta_1.$$
The condition~\eqref{e.choice-L} ensures that $\delta'_1< \frac{\delta}{8L}$
and that $\delta'_{k+1}>\delta_{k+1}$. 

Observe that the space $\RR^k\times \{0\}$ (resp.\ $\{0\}\times \RR^{d_0-k}$) is invariant
and contracted (resp.\ expanded).
Hence the subrectangles $R^s_j$ of $R$ (with full stable length) have images which intersect
$R$ as subrectangles of $R$ with full unstable length.
One deduces that the maximal invariant set of $h$ in
$$\bigcup_{i=0}^{\ell+m+2n-1}h^i(R^s_{0}\cup\dots R^s_{L-1})$$
is a horseshoe $K$ with topological entropy $h_\top(h,K)=\log(L)/(\ell+m+2n)$.
\medskip

Remember that  $\ell$ can be taken arbitrarily large, while keeping $\eta,m,n$ fixed,
and that $L$ has been taken within a constant factor of the bound in \eqref{e.choice-L}.
One thus sees that the topological entropy of the horseshoe $K$ is arbitrarily close to $\Delta(g,O)$.
\medskip

Let us check the positions of the stable and unstable manifolds near $g^{n+m}(\tau)$. We use the coordinates $\chi$ (centered at the fixed point $z^s$).
\begin{itemize}
 \item[--]  The local component of $W^s(O)$ at $g^{m+n}(\tau)$ is a preimage of some local stable manifold at $O$, hence it it does not depend on $\psi$ (because of its support).
 \item[--]  The segment defined by $-\delta<x_1<-\delta e^{-n\lambda_1}/2L$, $x_i=0$ for $i\ne1$
 in the coordinates centered at $g^{n+m}(\tau)$,
is contained in the image by $h^{m+2n+1}$ of a segment of $W^u(O)$ centered at $g^{-n-1}(\tau)$ (see the definition of $\Phi$).
 \item[--] The segments defined by $-\delta<x_1<\delta$, $x_i=\text{cte}$, $i\ne1$,
 in the coordinates centered at $z_s$,
and which meet $K$ are contained in $W^s(K)$ since they are mapped to stable segments in $R^s_j$ for some $j$. 
 \item[--]  The segments defined by $-\tfrac9{10}\delta<x_1<-\tfrac35\delta/L$ with constant $x_i$ ($i\ne1$), in the coordinates centered at $z_s$, which are included in the image by $h^{\ell+m+2n}$ of unstable line segments through $R$ parallel to ${d_0}$-axis meeting $K$ are contained in $W^u(K)$.
\end{itemize}
Observe that the intersection of the horseshoe $K$ with $R$ (near $g^{m+n}(\tau)$) is contained in the domain
where $x_1>-\delta/L$.
Moreover, to create the intersections between $W^s(O)$ and $W^u(K)$, resp. $W^u(O)$ and $W^s(K)$, we see that we need to push $W^u(K)$, resp. $W^u(O)$, by an amount exponentially small in $\ell$. Thus we can use a perturbation with support of size $\delta/3$
near the point with coordinate $x_1=\tfrac \delta 2$, independent of $\ell$. Hence its $C^1$-size can be made arbitrarily small. Inside $R$, this perturbation modifies the stable segments introduced above
but not the unstable ones, allowing the creation of the transverse intersections.

 Consequently, for the diffeomorphism $h$, the set $K$ and the orbit $O$ are homoclinically related.
This concludes the proof of Proposition~\ref{p.newhouse} and Theorem \ref{t.newhouse}.
\end{proof}

\subsection{Hausdorff dimension}
\newcommand\HD{\dim_H}
We explain how the construction proving Theorem \ref{t.newhouse} can yield a horseshoe with a large Hausdorff dimension in a $C^1$-robust way. This will lead to the proof of Theorem~\ref{t.dimension} in Section~\ref{ss.dimension}.

\begin{proposition}\label{p.dimension}
{In the setting of Theorem \ref{t.newhouse},}
if $\Delta(f,p)=\Delta^+(f,p)$,
the horseshoe $K$ can be assumed to have Hausdorff dimension
$\dim_H(K)$ larger than the unstable dimension of $p$.
 Also, the diffeomorphism $g$ is a continuity point of the map
$h\mapsto \dim_H(K_h)$, where $K_h$ is the hyperbolic continuation of
$K$ for diffeomorphisms $h$ that are $C^1$-close to $g$.
\end{proposition}

\begin{proof}
In the constructions in the proof of Theorem \ref{t.newhouse}, the diffeomorphism $g$
is affine in a neighborhood of $K$ (using the chart $\chi$).
Hence it is conformal along the stable and the unstable directions.
The Hausdorff dimension is then equal to $d^s+d^u$ where
$d^s=h_{top}(g_{|K})/|\lambda^s|$ and $d^u=h_{top}(g_{|K})/|\lambda^u|$.
Here $\lambda^u$ is the Lyapunov exponent in the unstable directions
and is close to $\Delta^+(f,p)/(d_0-k)$
and $\lambda^s$ is the Lyapunov exponent in the
stable direction and is close to $\Delta^-(f,p)/k$, see~\cite[Theorem 22.2]{pesin}.
As ${h_{top}(g_{|K})\geq}\min(\Delta^-(g,p), \Delta^+(g,p))$, { 
 $$d^u\geq \Delta^+(g,p)/|\lambda^u|=d_0-k \text{ and } d^s>0.$$
Thus $\HD(H(p,g))\geq \HD(K)>d_0-k$, which is the claim for $g$.}

In order to prove the second assertion of the statement,
one considers for each diffeomorphism $h$ that is $C^1$-close to $g$,
the (unique) homeomorphism $\phi$ close to the identity which conjugates $(K,g)$
with $(K_h,h)$. The following claim will imply that $\HD(K)\alpha\leq \HD(K_h)\leq\HD(K)/\alpha$ and hence the continuity.

\medbreak
\noindent{\bf Claim.} \emph{$\phi$ and $\phi^{-1}$ are H\"older with an exponent $\alpha\in (0,1)$ arbitrarily close to $1$
as $h$ gets closer to $g$}.  

Indeed, the distance $d(x,y)$ between two points $x,y$ in $K$ (or $K_h$)
is equivalent to $\max(d(x,[x,y]), d(y,[x,y]))$, where $[x,y]=W^s_{loc}(x)\cap W^u_{loc}(y)$ is the local product. Hence it suffices to prove that the restriction of $h,h^{-1}$ the local stable and unstable manifolds
is H\"older with an exponent $\alpha$ close to $1$.
This property has been proved in~\cite{PV} for horseshoes on surfaces,
but the same proof extends in any dimension when the horseshoe is conformal.
In our case, it is enough to note that for any $\eps>0$ close to $0$,
if $h$ is close enough to $g$, then the distance inside the local unstable manifolds
increases after $n$ iterates by a factor in $[\exp(n(\lambda^u-\eps)), \exp(n(\lambda^u+\eps))]$ and the distance inside the local stable manifolds decreases by a factor in
$[\exp(n(\lambda^s-\eps)), \exp(n(\lambda^s+\eps))]$.
\end{proof}

\section{Entropy formulas}\label{s.entropy}

In this section we prove Theorem \ref{t.entropycons} and its Corollaries~\ref{c.continuity} and~\ref{c.unstability}.  These results concern various ways to compute the entropy for generic maps in $\NDS$ and the dynamical consequences.

\subsection{Periodic points and horseshoes}\label{ss.periodichorseshoes}
We first prove the items~\eqref{i.entHS} and~\eqref{i.entPer} of Theorem \ref{t.entropycons}, relating the topological entropy to that of horseshoes and to the exponents of the periodic orbits.
From Ruelle's inequality we know
for any $f$ in $\Diff^1(M)$ and any $\mu\in \Probe(f)$ that
$$h(f,\mu)\leq \sum_{i=1}^{\dim(M)} \lambda_i^+(\mu).$$
Applying this result to
$f$ and $f^{-1}$, one gets $h(f,\mu)\leq \Delta(f,\mu)$.
The variational principle then gives $h_\top(f)\leq \sup_{\mu\in \Probe(f)}\Delta(f,\mu)$.
Combined with the ergodic closing lemma
(Corollary~\ref{c.ergodic-closing}), one obtains:
\begin{lemma}
For any $f$ in a dense G$_\delta$ subset of $\mathrm{Diff}^1_\omega(M)$,
$$h_\top(f)\leq \sup_{p\in \mathrm{Per}(f)}\Delta(f,p).$$
\end{lemma}}
Clearly for any diffeomorphism we have
$$\sup_{ \text{horseshoe } K}\; h_\top(f,K)\leq h_\top(f).$$
The two first items of Theorem \ref{t.entropycons} follow from these inequalities and from the next results.

\begin{lemma}\label{l.Delta-lsc}
The map  $f\mapsto \Delta(f):=\sup_{p\in \mathrm{Per}(f)}\Delta(f,p)$ defined over $\Diff^1_\omega(M)$ has a  dense G$_\delta$ subset of continuity points. The same holds over $\Diff^1(M)$.
\end{lemma}

\begin{proof}
The following proof applies to both the dissipative and conservative settings.
 For any $a\in \QQ$, let $U^+_a$ (resp. $U^- _a$) be the set of diffeomorphisms $f$ such that for any $g$ that is $C^1$-close to $f$, one has $\Delta(g)>a$ (resp. $\Delta(g)<a$). Franks' Lemma (Theorem~\ref{t.linearize}) implies that any $f$ such that $\Delta(f)\geq a$ lies in the closure of $U^+_a$. Hence $U^-_a\cup U^+_a$ is open and dense in $\NDS$.
\end{proof}

\begin{lemma}\label{l.ineq-entropy}
For any $f$ in a dense G$_\delta$ subset of $\NDS$ we have
$$\Delta(f)=\sup_{p\in \mathrm{Per}(f)}\Delta(f,p)\leq 
\sup_{ \text{horseshoe } K}\; h_\top(f,K).$$
\end{lemma}

\begin{proof} There exists a dense set of diffeomorphisms $f_0$ which belong to the intersection
of the sets $\cG_1,\cG_2,\cG_3\subset \NDS$ defined as follow.

\begin{enumerate}
\item $f\in\cG_1$ if it is a continuity point of $\Delta$ over $\NDS$.

\item $f\in \cG_2$ if any periodic orbit $O$ is hyperbolic, and if
its homoclinic class $H(O)$ is the whole manifold $M$. From~\cite{BC,ABC2},
it contains a dense G$_\delta$ subset of $\NDS$.

\item $f\in \cG_3$ if for any hyperbolic periodic orbit $O$ and any $\eps>0$,
there exists a periodic orbit $O'$ which is $\eps$-dense
in $H(O)$ and whose collection of Lyapunov exponents is $\eps$-close to
the Lyapunov exponents of $O$. {It is dense in $\cG_2$;} indeed, there exists a horseshoe containing $O$
that is $\eps/2$-dense in $H(O)$; then one applies Corollary~\ref{c.ergodic-closing}.
\end{enumerate}

Let us fix a neighborhood $\cU$ of $f_0$ and consider the integers $T,N$ given
by Theorem~\ref{t.newhouse}.
For $\eps>0$ small,
one chooses a periodic orbit $O$ such that $\Delta(f_0,O)>\Delta(f_0)-\eps/3$.
Since $f_0\in \cG_2\cap \cG_3$, one can replace $O$ by a periodic orbit which
is $\eps$-dense in $H(O)=M$.
Hence it has a large period and (since $f_0\in \NDS$), it does not have an $N$-dominated splitting.
{Theorem \ref{t.newhouse} yields} 
a diffeomorphism $g$ in $\cU$ having a horseshoe $K$
whose topological entropy is larger than $\Delta(f_0)-\eps/2$.
This property is open.
Since $f_0\in \cG_1$, one gets a non-empty open set of diffeomorphisms $g$ close to $f_0$, having
a horseshoe $K$ such that
\begin{equation}\label{e.ineq-entropy}
h_\top(g,K)>\Delta(g)-\eps.
\end{equation}
Taking the intersection of the sets of diffeomorphisms satisfying~\eqref{e.ineq-entropy} for each $\eps=1/n$, yields the required dense G$_\delta$ set.
\end{proof}

\begin{remark}\label{rem-small-hs}
\emph{The first two items of Theorem \ref{t.entropycons} can be strengthened:}
\smallskip

\noindent
For any $\delta>0$, one can restrict the supremum in
item~\eqref{i.entHS} to horseshoes $K$ with dynamical diameter smaller than $\delta$,
i.e., admitting an invariant partition $K:=K_1\cup\dots \cup K_\ell$ into compact subsets with diameters smaller than $\delta$.
\smallskip

\noindent
One can restrict the supremum in~\eqref{i.entPer} to periodic orbits with period larger than $\delta^{-1}$.
\smallskip

\emph{Indeed, in the proof of Lemma~\ref{l.ineq-entropy}, starting with a initial periodic orbit $O$,
one can find $O'$ arbitrarily dense in $M$ and $K$ contained in an arbitrarily small
of $O'$ such that $\Delta(O')$ and $h_\top(K)$ are larger than $\Delta(O)-\eps$.
In particular, $O'$ has large period and $K$ has small diameter.}
\end{remark}

\def\iLsc{(a)}
\def\iPtc{(b)}

\subsection{Continuity, stability}
We now prove Corollaries~\ref{c.continuity} and~\ref{c.unstability}.

\begin{proof}[Proof of Corollary \ref{c.continuity}]
From item~\eqref{i.entHS} of Theorem~\ref{t.entropycons} and {Lemma~\ref{l.Delta-lsc}},
there exists a dense G$_\delta$ set of diffeomorphisms $f\in \NDS$ satisfying:
\begin{enumerate}
\item[(a)] $h_\top(f)= \sup\{h_\top(f,K):\; K \text{ horseshoe of $f$}\}=\Delta(f)$.

\item[(b)] $f$ is a continuity point of $g\mapsto \Delta(g)$.
\end{enumerate}
From (a), a generic $f\in\NDS$ is a point of lower semicontinuity for the topological entropy by structural stability of  horseshoes.

Let $\eps>0$.
Let us consider a sequence $f_n\to f$ and ergodic measures $\mu_n$
for $f_n$ with $h(f_n,\mu_n)>h_\top(f_n)-\eps/4$. Ma\~n\'e's ergodic closing lemma
(Theorem~\ref{t.ergodic-closing})
gives a $C^1$-perturbation $\tilde f_n$ of $f_n$ and a hyperbolic periodic orbit $O_n$
with exponents close to those of $\mu_n$. Using Ruelle's inequality, one gets
$$\Delta(\tilde f_n, O_n)\geq \Delta(f_n,\mu_n)-\eps/4 \geq h(f_n,\mu_n)-\eps/4\geq h_\top(f_n)-\eps/2.$$
From (a) and (b), one also has
$$h_\top(f)+\eps/2=\Delta(f)+\eps/2\geq \Delta(\tilde f_n)\geq \Delta(\tilde f_n, O_n).$$
All these inequalities together give the upper semi-continuity
{of the topological entropy at $f$}
and conclude the proof of
Corollary \ref{c.continuity}.
\end{proof}

\begin{proof}[Proof of Corollary \ref{c.unstability}]
Theorem \ref{t.entropycons} provides a dense G$_\delta$ subset $\cG\subset\NDS$ on which the entropy is equal to $\Delta$.  Now, any $f\in\NDS$ can be perturbed to $f_1\in\cG$ and Franks' Lemma (Theorem \ref{t.linearize}) yields a further perturbation $f_2$ for which, robustly, $\Delta(f_2)>\Delta(f_1)$. Choosing $f_2\in\cG$ yields $h_\top(f_2)\ne h_\top(f_1)$.
\end{proof}

\subsection{Submultiplicative entropy formula} \label{s.sub}

We now prove item~\eqref{i.entSubmultiplicative} of Theorem~\ref{t.entropycons}, expressing the entropy as the maximum growth rate of the Jacobian. We consider
any $C^1$-diffeomorphism $f$ of a compact, $d_0$-dimensional, Riemannian manifold $M$
and for each integer $0\leq k\leq d_0$, define 
 \begin{equation}\label{def.sigmak}
 \sigma_k(f):=\lim_{n\to\infty}   \; \sup _{E\in\Grass_k(TM)}  \;\frac 1 n \log\Jac(f^n,E).
 \end{equation}

We first state some elementary properties.
\begin{lemma}\label{l.sigmak} The following hold for any $f\in\Diff^1(M)$:
\begin{enumerate}
\item $\sigma_0(f)=0$ (and $\sigma_{d_0}(f)=0$ if $f$ is conservative).
\item The limit defining $\sigma_k(f)$ exists and
 \begin{equation}\label{e.fekkel}
 \begin{split}
    \sigma_k(f)&=\inf_{n\geq1} \;\sup_{E\in\Grass_k(TM)}   \; \frac 1 n\log\operatorname{Jac} (f^n,E)\\
    &=\sup_{E\in\Grass_k(TM)} \inf_{n\geq 1} \;\frac 1 n \log  \operatorname{Jac} (f^n,E).
    \end{split}
 \end{equation}
  \item Each map $f\mapsto \sigma_k(f)$ is upper semicontinuous in the $C^1$ topology.
{More precisely,
for any $f$ and any $\alpha>0$, there exists $N_1$, such that for every
$1\leq k\leq d_0$, every $g$ $C^1$-close to $f$ we have
 \begin{equation}\label{e.sigmausc}
    \forall n\geq N_1,\; \forall E\in\Grass_k(TM),\quad
    \log \Jac(g^n,E) \leq (\sigma_k(f)+\alpha)n.
  \end{equation}}
 \end{enumerate}
\end{lemma}
\begin{proof}
The first item is obvious.
Note that 
$$\left(\sup_{E\in\Grass_k(TM)} \Jac(f^n,E)\right)_{n\geq1}$$ 
is submultiplicative.
This gives the convergence in the definition of $\sigma_k$ and first equality of the second item.
The second equality in eq.\ \eqref{e.fekkel} is a consequence of the continuity of $E\mapsto \operatorname{Jac} (f^n,E)$
over the compact space $\Grass_k(TM)$.
The last item is again a consequence of the submultiplicativity and of the continuity
of $(g,E)\mapsto \Jac(g,E)$.
\end{proof}

We now relate $\sigma_k(f)$ to Lyapunov exponents. For $\mu\in \Probe(f)$ we define 
 $$
  \sigma_k(f,\mu):= \sum_{j=d_0-k+1}^{d_0} \lambda_j(f,\mu).
 $$
 We also use the notation $\sigma_k(f,O)$ for a periodic orbit $O$. 

 Note that if $\mu$ has $i$ nonnegative Lyapunov exponents and thus $d_0-j$ negative ones, we have
 $\Delta(f,\mu)=\min(\sigma_i(f,\mu), \sigma_j(f^{-1},\mu))$. One deduces:
$$\Delta(f,\mu)=\max_k\min(\sigma_k(f,\mu), \sigma_{d_0-k}(f^{-1},\mu)).$$
 
\begin{lemma}\label{l-usc1}
The following functions are upper semi-continuous: $(f,\mu)\mapsto \sigma_k(f,\mu)$,
$(f,\mu)\mapsto \Delta(f,\mu)$ and $f\mapsto \sup_{\nu\in \Probe(f)} {\Delta}(f,\nu)$ where $f\in\Diff^1(M)$.
\end{lemma}
 
\begin{proof}{
Oseledets theorem implies that, for $f\in\Diff^1(M)$ and $\mu\in\Prob(f)$,
 $$
   \sigma_k(f,\mu)=\inf_{n\geq1} \int_M \frac1n\log\sup_{E\in\Grass_k(T_xM)} J(f^n,E)\, d\mu.
 $$
This integral is a bounded, continuous function of $(f,\mu)$. The upper semicontinuity of $(f,\mu)\mapsto\sigma_k(f,\mu)$ follows.

The function $(f,\mu)\mapsto\Delta(f,\mu)$ is upper semicontinuous because this property is stable by taking  maximum or minimum over finite families. 

Finally, let $f_n\to f$. Pick $\mu_n\in\Prob(f_n)$ with $\Delta(f_n,\mu_n)>{\sup_{\nu\in \Probe(f_n)} \Delta(f_n,\nu)}-1/n$. By compactness, we can assume { that $(\mu_n)$ converges to some $\mu\in\Prob(f)$} and get, from the semicontinuity of $(f,\mu)\mapsto\Delta(f,\mu)$:
 $$
    \limsup_n \sup_{\nu\in \Probe(f_n)} \Delta(f_n,\nu)=\limsup_n \Delta(f_n,\mu_n)\leq\Delta(f,\mu)\leq{\sup_{\nu\in \Probe(f)} \Delta(f,\nu)}.
 $$}
 \end{proof}

\begin{lemma}\label{l.deltak}
For all $f\in\Diff^1(M)$ and $0\leq k\leq d_0$, we have
$$\sigma_k(f) = \sup_{\mu\in\Probe(f)} \sigma_k(f,\mu).$$
\end{lemma}
\begin{proof}
Let  $\mu\in\Probe(f)$.
Applying the Oseledets theorem at $\mu$-almost every point $x$,
one finds a $k$-dimensional subspace of $T_xM$ whose volume grows at the rate given by the $k$ strongest exponents of $\mu$.
Hence:
\begin{equation}\label{e.oseledets-mu}
\sigma_k(f,\mu)=\;\sup_{E\in\Grass_k(T_xM)}   \; \frac 1 n\log\operatorname{Jac} (f^n,E)\leq \sigma_k(f).
\end{equation}

For the converse inequality, on the compact metric space ${X:=}\Grass_k(TM)$ one considers the homeomorphism
$f_k\colon E\mapsto Df(E)$ and the continuous function $\phi_k\colon E\mapsto\log\operatorname{Jac}(f,E)$.
Note the obvious factor map $\pi_k\colon (\Grass_k(TM), f_k)\to(M,f)$. By a well-known argument, one gets:
 $$
    \sigma_k(f) =  \lim_{n\to\infty}\left[\max_{{E}\in X}\left(\frac1n\sum_{i=0}^{n-1}\phi_k(f_k^i{(E)})\right)\right]
    =\sup_{\nu\in\Probe(f_k)} \nu(\phi_k).$$
Let $\nu\in\Probe(f_k)$. The measure $\mu=(\pi_k)_*(\nu)$ belongs to $\Probe(f)$.
The ergodic theorem gives $x\in M$ satisfying~\eqref{e.oseledets-mu} and $E_0\in \Grass_k(T_xM)$ such that
$$\begin{aligned}
    \nu(\phi_k) &= \lim_{n\to\infty}\frac1n\log \Jac(f^n,E_0)\\
       &\leq \sup_{E\in\Grass_k(T_xM)} \lim_{n\to\infty}\frac1n\log \Jac(f^n,E)  = \sigma_k(f,\mu).
  \end{aligned}$$
This implies $\sigma_k(f)\leq \sup_{\mu\in\Probe(f)} \sigma_k(f,\mu)$.
\end{proof}

Finally, we show the following generic formula, which immediately implies item \eqref{i.entSubmultiplicative} of Theorem \ref{t.entropycons}. 

\begin{proposition}\label{p.EntSubMul}
For $f$ in a dense G$_\delta$ subset of $\Diff^1_\omega(M)$
 \begin{equation}\label{e.entropysigma}
  \Delta(f) = \max_{0<k<d_0} \sigma_k(f).
 \end{equation}
\end{proposition}

{Notice that we do not assume that there is no dominated splitting, and also that this formula fails on some obvious open sets in $\Diff^1(M)$.}

\begin{proof}
{
Let $f\in \Diff^1_\omega(M)$ be generic. As $f$ is conservative we know
 $$
   \Delta(f) = \sup_O\Delta(f,O) = \sup_O\max_{0<k<d_0} \sigma_k(f,O).
 $$
By Corollary~\ref{c.ergodic-closing}
any $\mu\in\Probe(f)$ is approximated by periodic orbits $O$ with arbitrarily close Lyapunov exponents and thus we have
 $$
    \sup_{O} \sigma_k(f,O) \geq  \sup_{\mu\in\Probe(f)}  \sigma_k(f,\mu).
  $$
As any periodic orbit defines an ergodic invariant measure, this is in fact an equality.

Together with Lemma \ref{l.deltak} we have
 $$
    \Delta(f)= \max_{0<k<d_0} \sup_{\mu\in\Probe(f)}\sigma_k(f,\mu) = \max_{0<k<d_0} \sigma_k(f).
 $$
}
\end{proof}

\section{No measures of maximal entropy}\label{s.non}

In this section we prove Theorem~\ref{t.nomax}.

\subsection{A concentration phenomenon for high entropy measures}
For $x\in M$, $f\in \Diff(M)$, $\eps>0$ and $N\geq 1$,
{corrected from $(n,\epsilon)$}
the \emph{$(N, \epsilon)$-Bowen ball around $x$} is
$$
B_f(x,\eps,N):=\{ y\, :\, \forall\,  0\leq k< N,\; d(f^k(x),f^k(y))<\eps\}.
$$

\begin{proposition}\label{p.HEconcentration}
Consider any function $\varphi\colon \NN\to\NN$ 
and
fix $\eps,\alpha\in (0,1)$. Then for any $f_0$ in a dense subset of $\NDS$,
there exist a constant $\delta>0$ and a finite set $X\subset M$ such that,
for any $f\in \NDS$ close to $f_0$ and any ergodic  measure $\mu$ for $f$, we have
\begin{equation}\label{e.concentration}
    h(f,\mu)>h_\top(f)-\delta \implies \mu\left(M\setminus \bigcup_{x\in X} B_f(x,\eps,\varphi(\#X))\right)<\alpha.
 \end{equation}
\end{proposition}

\begin{proof}
From Theorem~\ref{t.entropycons} and Corollary~\ref{c.continuity} there exists a dense G$_\delta$ subset $\cR\subset \NDS$ of diffeomorphisms $g$ such that
\begin{itemize}
\item[--] $g$ is a continuity point of $h_\top$, and
\item[--]
$
h_\top(g)=\max_k\sigma_{k}(g)
=\Delta(g).
$
\end{itemize}

Before stating the main perturbation lemma, we need the following fact and number $C_0>0$. For any $z\in M$ and any linear spaces $L,L'\subset T_zM$ having equal dimensions,
there exists a linear map $D$ on $T_zM$ such that $D(L)=L'$. 
In the symplectic case, if $L,L'$ are Lagrangian, then one can choose $D$ symplectic with $\|D\|,\|D^{-1}\|\leq C_0$ for some uniform constant $C_0>0$,  see Proposition~\ref{p.change-basis2}.
In the dissipative or volume-preserving case, one can choose $D$ orthogonal and  set $C_0\geq1$.
\medskip

\begin{lemma}\label{l.perturb}
For any $g\in \cR$, any $\eta$ in some interval
$(0,\eta_0(g))$, any $\delta>0$, and any periodic saddle $O$ for $g$,
 there exists $N_0\geq 0$ with the following property.

For any $\rho>0$, there is an $f_0\in \cR$ that is $\eta$-close to $g$ and satisfies:
\begin{enumerate}
\item\label{i.l1} $f_0$ is arbitrarily $C^1$-close to $g$
outside the $\rho$-neighborhood of $O$;
\item\label{i.l2} $h_\top(f_0)\geq \Delta(g,O)+\frac{1}{10C^2_0(\|Dg\|+\|Dg^{-1}\|)}\eta$; and
\item\label{i.l3} for any $0\leq k \leq d_0$, any $E\in \Grass_k(TM)$, and any $n\geq N_0$,
$$\log |\Jac(f_0^{n},E)|\leq (h_\top(f_0)+\delta)n.$$
\end{enumerate}
\end{lemma}
\begin{proof} We pick $\eta_0(g)>0$ small enough so that 
$\|Df\|_\infty\leq 2\|Dg\|_\infty$ and $\|Df^{-1}\|_\infty\leq 2\|Dg^{-1}\|_\infty$ 
for all $f$ which are $\eta_0(g)$-close to $g$. 
The proof requires several steps.

\paragraph{Control of the $C^1$-size and of the entropy (items 1 and 2).}
Let
$$b=\frac{\eta}{8C^2_0(\|Dg\|+\|Dg^{-1}\|)}.$$
For each point $x\in O$, there exists a linear map $C\colon T_xM\to \RR^{d_0}$
which sends the stable space $E^s(x)$ to $\RR^{\dim E^s(x)}\times \{0\}^{d_0-\dim E^s(x)}$, and satisfies $\|C\|, \|C^{-1}\|\leq C_0$.
We introduce :
 $$
   U_x=C^{-1}\circ  \exp\begin{pmatrix} -b/\dim(E^s) & 0\\ 0 & b/\dim(E^u)\end{pmatrix}\circ C
 $$
(a homothety on $E^s$). We note that $\|U_x\circ Dg(x)-Dg(x)\|$ and $\|(U_x\circ Dg(x))^{-1}-(Dg(x))^{-1}\|$
are bounded by 
$2bC_0^2\max(\|Dg\|,\|Dg^{-1}\|).$
Let $\eta:=8bC_0^2(\|Dg\|+\|Dg^{-1}\|)$ and $V$ be the $\rho$-neighborhood of $O$.
The Franks lemma (Theorem~\ref{t.linearize}) yields a $(\eta/2,V,O)$-perturbation
$h$ with $Dh(x)=U_xDg(x)$ for all $x\in O$.
In particular, $\Delta(h,O)=\Delta(g,O)+b$.
One then chooses $f_0\in \cR$ arbitrarily close to $h$.
We have $h_\top(f_0)= \Delta(f_0)\geq \Delta(g,O)+8b/10$.
This gives the two first items of the lemma.

\paragraph{Intermediate constructions.}
{The third item requires a more precise construction.
For that we need to introduce some preliminary objects for each $0\leq k\leq d_0$.
\begin{itemize}
\item[--] An integer $N_1$ such that (see~\eqref{e.sigmausc}) for any diffeomorphism $f$
$C^1$-close to $g$, any $E\in \Grass_k(TM)$, and any $n\geq N_1$,
\begin{equation}\label{e.growth-g1}
 \log |\Jac( f^{n},E)|< (\max_k\sigma_k(f)+\delta/3)n.
\end{equation}
\item[--] A diffeomorphism $G$ on the tangent bundle $\TMO$,
defined as follow.
Let $V_r$ be the union of the $r$-balls at the origin in each $T_xM$, $x\in O$.
Then one can find a diffeomorphism $G$ of $\TMO$ and $0<r_1<1<r_2<\infty$ such that:
\begin{enumerate}[ (a)]
\item $G$ coincides with $Dg(x)$
outside the unit balls of each space $T_xM$, $x\in O$,
and with $U_xDg(x)$ on $V_{r_1}\cap T_xM$,
\item $G$ is $\eta/2$-close to $Dg_{|O}$, and
\item if $\|u\|\geq r_2$, then $\|G^n(u)\|\geq 1$ for all $n\geq0$ or all $n\leq 0$. 
\end{enumerate}
Let $\Lambda$ be the maximal invariant set of $G$ in $V_{r_2}$.
Property (c) implies that any orbit segment for $G$ can be split into at most three subsegments contained in $V_{r_2}$ or outside of $V_1$. Applying this to arbitrarily long orbit segments, we obtain
that $\sigma_k(G)$ is well defined and satisfies:
 $$
   \sigma_k(G)=\max(\sigma_k(G|\Lambda),\sigma_k(Dg|O))=\sigma_k(G|\Lambda).
 $$

\item[--] An integer $N_2$ such that (see Lemma \ref{l.sigmak}) for any $E\in \Grass_k(TM)$ and $n\geq N_2$,
\begin{equation}\label{e.growth-g2}
 \log |\Jac(G^{n},E)|< (\sigma_k(G)+\delta/3)n.
\end{equation}
\end{itemize}

\paragraph{Construction of $f_0$.}
The map $f_0$ is obtained as follow:
\begin{enumerate}
\item \emph{Linearization.} Theorem~\ref{t.linearize} gives a diffeomorphism $g'$
which is linear in a small neighborhood of $O$ and arbitrarily $C^1$-close to $g$.
\item \emph{Deformation.} The dynamics of $g'$ near $O$  may be identified with the linear cocycle $Dg$
over the tangent bundle $\TMO$.
One chooses $R>0$ small and one defines the perturbation $h$ of $g$ by replacing $g'$ in a small neighborhood
of $O$ by the diffeomorphism $G_R\colon z\mapsto RG(R^{-1}z)$.
In particular $h$ preserves a set $\Lambda_R=R.\Lambda$
and $\sigma_k(h|\Lambda_R)=\sigma_k(G|\Lambda)$.
Note also that the Jacobians are not modified after conjugacy by an homothety
and $G_R$ still satisfies~\eqref{e.growth-g2}.
\item \emph{Genericity.} As in the first step of the construction, one then approximates $h$ by a diffeomorphism
$f_0\in \cR$. One can require that $\Delta(f_0)\geq \max_k\sigma_k(h)$ for each $k$.
Indeed, one can consider a measure $\mu\in \Probe(h)$ and $k$ such that
$\sigma_k(h,\mu)$ approximates $\max_k\sigma_k(h)$. One then applies Theorems
\ref{t.ergodic-closing} and \ref{t.gourmelon} to get a diffeomorphism $h'$ satisfying
$\Delta(h')> \max_k\sigma_k(h)$ and one chooses $f_0$ close to $h'$.
\end{enumerate}
The $C^1$-perturbations in steps 1 and 3 are arbitrarily small
and the support of perturbations in step 2 is contained in an arbitrarily small neighborhood of $O$.

\paragraph{Bound on the Jacobian.}
We now set $N=\max(N_1,N_2)$.
Since $R$ can be chosen arbitrarily small once $N$ is given, any piece of orbit of $h$ of length $N$ of $h$
coincides with a piece of orbit of $g'$ or of $G_R$.
Hence, for any $E\in \Grass_k(TM)$,
$$
 \log |\Jac(h^{N},E)|< (\max(\sigma_k(G),\sigma_k(g'))+\delta/3)N
 \leq (\sigma_k(h|\Lambda)+\delta/3)N.
$$
Note that the integer $N$ only depends on $g$ and $G$,
but not on the choices made in steps 1--3 above.
Since  $\sigma_k$ is lower semicontinuous,
for any diffeomorphism $f_0$ $C^1$-close to $h$
and any $E\in \Grass_k(TM)$,
$$
 \log |\Jac(f_0^{N},E)|< \log|\Jac(h^N,E)|+\tfrac\delta{12}N
   \leq (\sigma_k(h|\Lambda)+\tfrac{5}{12}\delta)N\leq (\sigma_k(f_0|\Lambda)+\delta/2)N.
$$
One deduces that there exists $N_0$, which only depends
on $N$ and on $\|Dg\|$, such that the item 3 holds.

The proof of Lemma~\ref{l.perturb} is complete.}
\end{proof}
\bigskip

Let us continue with the proof of the Proposition~\ref{p.HEconcentration}.
We fix $\varphi,\eps,\alpha$ and let $g\in\cR$.
We have to build $f_0\in\NDS$ arbitrarily close to $g$, a number $\delta>0$, 
and a finite set $X\subset M$ as in the conclusion of the proposition.
We pick an arbitrarily small number $0<\eta\leq\eta_0(g)$ with $\eta_0(g)$ given by Lemma \ref{l.perturb} and set
$$\delta=\alpha\eta/(100C^2_0(\|Dg\|+\|Dg^{-1}\|)).$$

Using Lemma~\ref{l.sigmak} and $g\in \cR$, one chooses $N_1$ such that
for any $0\leq k \leq d_0$, any $E\in \Grass_k(TM)$, and $n\geq N_1$, we have
\begin{equation}\label{e.growth-g}
\log |\Jac(g^{n},E)|\leq (h_\top(g)+\delta)n.
\end{equation}
From the choice of $\cR$, there exists a periodic orbit $O$ such that
$$\Delta(g,O)\geq h_\top(g)- \delta.$$
Lemma \ref{l.perturb} also gives $N_0=N_0(g,O,\delta)$ and we set $N=\max(N_0,N_1,\varphi(\#O))$.
We also choose $\rho>0$ such that for any $f$ that is $2\rho$-close to $g$
for the $C^0$-topology  we have
 \begin{equation}\label{e.piste}
    \bigcup_{0\leq k< N} f^{-k}B_f(O,\rho)\subset \bigcup_{x\in O} B_g(x,\eps,N)
   \subset \bigcup_{x\in O} B_g(x,\eps,\varphi(\# O)).
 \end{equation}
Lemma \ref{l.perturb} now gives a diffeomorphism $f_0$ that is $\eta$-close to $g$ in $\NDS$
and $2\rho$-close to $g$ for the $C^0$-distance.
We check that~\eqref{e.concentration} holds with $X=O$
for any $f\in \NDS$ close to $f_0$ and any ergodic measure $\mu$ for $f$.
Let us assume that $h(f,\mu)\geq h_\top(f)-\delta$.

Note that it is enough to estimate the proportion of time spent
inside $\cup_{x\in X} B_f(x,\eps,\varphi(\#X))$ by the forward
orbit under $f^N$ of $\mu$-almost every point.

Since $f$ is close to $f_0\in \cR$, 
$$
  h(f,\mu)\geq h_\top(f)-\delta> h_\top(f_0)-2\delta.
 $$
By Ruelle's inequality, there exists $0\leq k \leq d_0$ such that
$\sigma_k(\mu)\geq h(f,\mu).$
By Oseledets theorem, for $\mu$-almost every $z$,
there is $E\subset  \Grass_k(T_zM)$ with
\begin{equation}\label{e.limit-h}
\frac 1 n\log \Jac(f^n,E)\underset{n\to+\infty}\longrightarrow \sigma_k(\mu)\geq h_\top(f_0)-2\delta.
\end{equation}
When $f^n(z)$ is not in $\cup_{x\in X} B_f(x,\eps,\varphi(\#X))$,
eq.\ \eqref{e.piste} shows that $f^N(f^nz)$ and $g^N(f^nz)$ are arbitrarily close from item~\ref{i.l1}
of Lemma~\ref{l.perturb}. Together with~\eqref{e.growth-g}
we get 
$$\frac 1 N\log \Jac(f^N,Df^n(E))\leq h_\top(g)+2\delta.$$
When $f^n(z)$ is in $\cup_{x\in X} B_f(x,\eps,\varphi(\#X))$,
item~\ref{i.l3} of the same lemma gives
 $$
  \frac 1 N\log \Jac(f^N,Df^n(E))\leq h_\top(f_0)+2\delta.
 $$

For each $m\geq 1$, we define:
$$p_m=\frac 1 m\#\big\{0\leq \ell\leq m-1,\; f^{\ell N}(z)\not \in \cup_{x\in X} B_f(x,\eps,\varphi(\#X))\big\}.$$
This gives
$$\frac 1 {mN}\log \Jac(f^{mN},E)\leq
(1-p_m) (h_\top(f_0)+2\delta)+ p_m(h_\top(g)+2\delta),$$
and, using~\eqref{e.limit-h},
 $$
    p_m\leq \frac{4\delta}{h_\top(f_0)-h_\top(g)}\leq {100\delta C_0^2(\|Dg\|+\|Dg^{-1}\|)/\eta}\leq\alpha
 $$
from item~\ref{i.l2} of Lemma~\ref{l.perturb} and the choice of $\delta$.
\end{proof}

\subsection{Non-existence of measures of maximal entropy}

With an immediate Baire argument, Proposition~\ref{p.HEconcentration} gives the next result.

\begin{corollary}\label{c.HEconcentration}
For any function $\varphi\colon \NN\to\NN$ there exists a
dense G$_\delta$ set $\cG= \cG(\varphi)\subset \NDS$
such that for any $\eps,\alpha\in (0,1)$ and $f\in \cG$,
 there exist $\delta>0$ and a finite set $X\subset M$ satisfying~\eqref{e.concentration}.
\end{corollary}

\begin{proof}[Proof of Theorem~\ref{t.nomax}]
We choose $\varphi=\operatorname{Id}_\NN$ and
consider $f$ in the dense G$_\delta$ subset $\cG\subset \NDS$ given by Corollary~\ref{c.HEconcentration}.
Let us assume by contradiction that there is an ergodic measure $\mu$ maximizing the entropy:
 \begin{equation}\begin{aligned}\label{e.KatokEntFor}
   &h_\top(f)=h(f,\mu)=\lim_{\eps\to 0} \limsup_{n\to\infty}\frac1n\log r_f(\mu,\eps,n),
 \end{aligned}\end{equation}
 where $r_f(\mu,\eps,n)$ is the minimal number of sets $B_f(x,\eps,n)$
 needed to cover a set of $\mu$-measure larger than $1/2$ (this is Katok's formula for the Kolmogorov-Sinai entropy \cite{Katok}).
By  Corollary \ref{c.unstability}, we can assume that $h_\top(f)>0$.

Let us fix $\eps>0$ and some $\eps$-dense finite subset $\cA\subset M$.
Let $0<\alpha\ll 1/\log(\# \cA)$.
Corollary~\ref{c.HEconcentration} gives $\delta>0$ and a
set $X\subset M$ (arranging $\cA$ to be disjoint from $X$). We set $N=\# X$.
Notice that eq.\ \eqref{e.concentration} implies that $N=\#X$ is large since  $h(f,\mu)=h_\top(f)>0$ and $\eps$ is small.

Let us bound $r_f(\mu,2\eps,n)$ for $n$ large. The Birkhoff Ergodic Theorem gives $Y\subset M$ measurable with $\mu(Y)\geq 1/2$ and an integer $n_1$ such that, for all $y\in Y$ and $n\geq n_1$,
 $$
    \#\{0\leq k<n:\;f^k(y)\notin \cup_{x\in X}B_f(x,\eps,N)\} < n\alpha.
 $$ 
We associate to each $y\in Y$ the following data:
 \begin{itemize}
  \item[--] integers $a_1,a_2,\dots$ defined inductively as:
 $$a_{i+1}:=\min\{k\geq a_i+N,\; f^k(y)\in \cup_{x\in X}B(x,\eps,N)\}$$
 (by convention $a_0:=-N$),
  \item[--] points $q_0,q_1,\dots\in\cA\sqcup X$ satisfying:
  \begin{equation*}
  \begin{split}
&q_{a_i}=q_{a_i+1}=\dots=q_{a_i+N-1}\in X \text{ with } f^{a_i}(y)\in B_f(q_{a_i},\eps,N) \text{ for each }i\geq 1, \\
 &q_k\in \cA \text{ with } f^k(y)\in B(q_k,\eps) \text{ if } k\in \NN\setminus \bigcup_{i\geq 1} [a_i,a_i+N-1].
  \end{split}
  \end{equation*}
  \end{itemize}
Note that if $y,z$ share the same sequence $(q_0,\dots,q_{n-1})$,
then  $d(f^k(y),f^k(z))<2\eps$ for each $0\leq k<n$.
The number of such sequences $(q_0,\dots,q_{n-1})$ is bounded by 
 $$
    [n/N]\times \binom{n}{[n/N]} \times (\#\cA)^{\alpha.n}\times (\# X)^{n/N+1}.
 $$
Denoting $H(t)=-t\log(t)-(1-t)\log(1-t)$ and letting $n\to\infty$,
Stirling's formula gives
 $$
   \frac 1 n \log r_f(\mu,2\eps,n) \leq \frac{\log n}{n}+
   H(1/N) + \alpha.\log(\#\cA)+\frac 1 N \log(\# X).
 $$
 Since $\alpha.\log(\#\cA)$ is arbitrarily small and
 $N=\# X$ is arbitrarily large, one gets $h_\top(f)=h(f,\mu)=\lim_{\eps\to0}\lim_{n\to\infty}h(f,\mu,\eps) =0$ for all generic $f$ in the dense G$_\delta$ set $\cG$, a contradiction.
\end{proof}

\subsection{Borel classification}

We prove the following stronger version of Corollary \ref{c.almostBorel}. {The fundamental tool is the uniform Borel generator theorem of Hochman \cite{Hochman2}.}

\begin{proposition}\label{p.Borel}
There is a dense G$_\delta$ set of diffeomorphisms $f\in\NDS$ such that the free part of $f$ is Borel conjugate to the free part of any  mixing Markov shift with no measure maximizing the entropy and with Gurevi\v{c} entropy equal to $h_\top(f)$. In particular, the free parts of two diffeomorphisms in $\mathcal G$ are Borel conjugate if and only if they have equal topological entropy. 
\end{proposition}

We recall from \cite{GurevichSavchenko,Sarig} that the Gurevi\v{c} entropy of a Markov shift is the supremum of the entropies of its invariant probability measures and coincides with the supremum of the topological entropies of the shifts of finite type that it contains. Also, for any $0\leq h<\infty$, there exist mixing Markov shifts with Gurevi\v{c} entropy $h$ and no measure maximizing the entropy \cite{RuetteNote}.

\begin{proof}
Let $\mathcal G$ be the set of diffeomorphisms $f\in\NDS$ such that  $h_\top(f)$ is the supremum of the topological entropy of its mixing horseshoes and $f$ has no measure maximizing the entropy. 

Then $\mathcal G$ contains a dense G$_\delta$ subset of $\NDS$. Indeed,
by Theorem \ref{t.entropycons}, generic diffeomorphisms $f\in \NDS$ admit horseshoes with entropy arbitrarily close to $h_\top(f)$ and \cite{AbdenurCrovisier} yields larger, topologically mixing horseshoes. By Theorem~\ref{t.nomax}, the generic diffeomorphisms have no measure of maximal entropy. 

The free parts of any $f\in\mathcal G$ and any Markov shift as described are Borel conjugate, as follows from an easy adaptation of Corollary 1.3 of Hochman \cite{Hochman2} using the techniques of \cite{Hochman1} (see also \cite{BuzziLausanne} for an exposition).
\end{proof}

\section{Tail entropy and entropy structures} \label{s.entropystructure}

We draw the consequences of Theorem \ref{t.entropycons} for the symbolic extension theory of Boyle and Downarowicz \cite{BoyleDownarowicz}. We prove Proposition \ref{p.ualpha} below which is a stronger, more technical version of Theorem~\ref{t.ualphageneral}.  We then discuss the value
of the tail entropy for $C^1$-generic conservative diffeomorphisms allowing a dominated splitting
and prove Theorem~\ref{t.tail-general}.

\subsection{Entropy structures of general homeomorphisms}
We recall some basic definitions and results. Let $f:M\to M$ be a continuous map of a compact metric space. We often omit $f$ from the notation when there is no confusion on the map. 
Downarowicz \cite{DownarowiczEntropyStructure} introduced the \emph{entropy structure} as a ``master invariant for the theory of entropy" in topological dynamics. According to \cite[Theorem 8.4.1]{DownarowiczBook}, it can be defined\footnote{To be precise: the entropy structure is the equivalence
class of the sequence $(h_k)_{k\geq1}$ for a relation called \emph{uniform equivalence}, see \cite[section 8.1.4]{DownarowiczBook}.} as the following sequence of Romagnoli's entropy functions  $h_k:\Prob(f)\to[0,\infty]$:  
 \begin{equation}\label{e.Romagnoli}
    h_k(f,\mu):=\limsup_{n\to\infty} \frac1n \inf_{\mathcal A\in\mathcal P(1/k,n)} -\sum_{A\in\mathcal A} \mu(A)\log\mu(A)
 \end{equation}
where $\mathcal P(\eps,n)$ is the set of finite measurable partitions $\mathcal A$ of $M$ satisfying $\diam(f^jA)<\eps$ for all $A\in\mathcal A$ and $0\leq j<n$.
\medskip

For a function $u:\Prob(f)\to [0, \infty]$ we define
$$\widetilde{u}(\mu):=
\left\{
\begin{array}{llll}
\limsup_{\nu\to\mu} u(\nu)&\textrm{ if }u\textrm{ is bounded},\\
\infty&\textrm{ otherwise.}
\end{array}
\right.
$$
So  $\widetilde{u}$ is the smallest upper semicontinuous function $\Prob(f)\to [0, \infty)$ larger than or equal to $u$ (by convention the constant function $\infty$ is the only unbounded upper semicontinuous function).

According to \cite{BoyleDownarowicz}, the entropy structure determines  functions $u_\beta:\Prob(f)\to[0,\infty]$ for all ordinals $\beta$ by the following transfinite induction:
\begin{itemize}
\item[--] $u_0:=0$, 
\item[--] $u_{\beta+1}:=\lim_{k\to\infty} \widetilde{(u_\beta+h-h_k)}$, and
\item[--] if $\beta$ a limit ordinal, 
then $u_\beta:=\widetilde{\sup_{\gamma<\beta} u_\gamma}. $
\end{itemize}

The entropy structure allows one to recover the tail entropy.

\begin{theorem}[\cite{DownarowiczEntropyStructure}]\label{t.DownarowiczEntropyStructure}
The tail entropy $h^*(f)$ coincides with $\|u_1\|_\infty$.
\end{theorem}

There is a smallest ordinal $\alpha$ such that $u_{\alpha+1}=u_\alpha$. It is known to be a countable (possibly finite) ordinal.
It is a topological invariant, called the \emph{order of accumulation of the entropy structure}.
The function $u_\alpha$ describes the ``best" possible symbolic extensions. 

\begin{theorem}[\cite{BoyleDownarowicz}]\label{t.BoyleDownarowicz}
If $\alpha$ is the order of accumulation of the entropy structure,
then $u_\alpha=\infty$ if and only if there is no symbolic extension.
\end{theorem}

For systems with bounded topological entropy
and no symbolic extension, the order of accumulation $\alpha$ has to be infinite. 

\subsection{Entropy structure of diffeomorphisms in $\NDS$}\paragraph{a-- Inequalities for entropy structures.}
For $\nu\in \Prob(f)$, not necessarily ergodic, and almost any point $x$, we let $\Delta^+(f, x)$ be the sum of the positive Lyapunov exponents at $x$ and we define 
$$
\Delta^+(f, \nu)=\int \Delta^+(f, x) d\nu.
$$

\begin{proposition}\label{p.ualpha}
For any $C^1$-diffeomorphism $f$ of a compact manifold $M$, for all $n\in\NN$ and all $\mu\in \Prob(f)$ we have
 \begin{equation}\label{e.unhdelta}
        u_n(\mu) \leq n \cdot \widetilde{h}(f,\nu) \leq n \cdot \widetilde{\Delta^+}(f,\nu).
 \end{equation}
If $f\in\NDS\subset\Diff^1_\omega(M)$ is generic, these are equalities.
\end{proposition}

One can view this as a rigidity result: generically, away from dominated splitting, all the functions $u_\beta$ are determined by the entropy function. This is in contrast to the general case:
according to \cite{BurguetMcGoff}, all countable ordinals are realized as the order of accumulation of the entropy structure of homeomorphisms of compact metric spaces. 

\begin{proof}[Proof of Theorem \ref{t.ualphageneral} from Proposition \ref{p.ualpha}]
{Let $f$ be a generic element in $\NDS$.}
Proposition~\ref{p.ualpha} implies $\sup u_n=n\cdot h_\top(f)$. By Theorem \ref{t.DownarowiczEntropyStructure}, the equality for $n=1$ yields $h^*(f)=h_\top(f)$.

As $h_\top(f)\ne0$, the functions $u_n$, $n\in\NN$, are pairwise distinct and the order of accumulation $\alpha$ is infinite. Now, $\sup u_\omega=\sup_\mu \sup_{n\in\NN} u_n(\mu)=\infty$, so $u_\omega\equiv\infty$ and $\alpha=\omega$.

From Theorem~\ref{t.BoyleDownarowicz}, one deduces that $f$ has no symbolic extension.
\end{proof}

\paragraph{b-- Perturbations of entropy structures.}
The next lemma will follow from results in \cite{ABC} 
and Theorem \ref{t.newhouse} and is the key step to the proof of Proposition~\ref{p.ualpha}.
Let $d_*$ be a distance on the space $\Prob(M)$ of probability measures on $M$ compatible with the vague topology. 

\begin{lemma}\label{l.smallapprox}
Consider a generic $f\in\NDS$ and $(h_k)_{k\geq1}$ its entropy structure as in \eqref{e.Romagnoli}. For any $\mu\in \Prob(f)$ and $\eps>0$,
there is $\nu\in\Prob(f)$ such that
\begin{enumerate}
\item $d_*(\mu,\nu)<\eps$,
\item 
$|h(f,\nu)-\Delta^+(f,\mu)|<\eps,$ and
\item $h_k(f,\nu)=0$ for any $k\leq 1/\eps$.
\end{enumerate}
\end{lemma}

\begin{proof}
We first explain how to robustly approximate a given measure. 

\begin{claim}\label{cl.stronglyapproximated}
For any $f$ generic in $\NDS$, $\mu\in\Prob(f)$ and $\eps>0$, there are $g\in \NDS$ arbitrarily close to $f$, some horseshoes $K_0,\dots,K_{\ell}$ for $g$, and $\alpha_0,\dots,\alpha_{\ell}>0$ such that \begin{enumerate}
  \item $|\sum_{i=0}^{\ell} \alpha_i\cdot h_\top(g,K_i)-\Delta^+(f,\mu)|<\eps$;
 
   \item each $K_i$ can be written as $\bigcup_{i=0}^{p-1} g^i(\Lambda)$ for some $p\geq1$ and some compact subset $\Lambda$ satisfying $g^p(\Lambda)=\Lambda$ and $\diam(g^i(\Lambda))<\eps$ for all $i\geq0$; and
   
 \item for any $\nu_i\in\Prob(f,K_i)$, if one denotes $\nu:=\sum_{i=0}^{\ell} \alpha_i\nu_i$, then
  $d_*(\mu,\nu)<\eps$.

 \end{enumerate}
 In particular, from the definition~\eqref{e.Romagnoli},
the measure $\nu$ in the third item satisfies $h_k(f,\nu)=0$ for $k<\eps$.
\end{claim}

\begin{proof}
From the ergodic decomposition theorem
and Corollary~\ref{c.ergodic-closing}, there exists a finite family of hyperbolic periodic orbits
$\cO_1,\dots,\cO_\ell$ and $\alpha_0,\dots,\alpha_\ell>0$
such that $\sum_{i=0}^\ell \alpha_i\cO_i$ is arbitrarily close to the measure $\mu$
and $\sum_{i=0}^\ell \alpha_i\Delta^+(f,\cO_i)$ is arbitrarily close
to $\Delta^+(f,\mu)$.
Since $f$ is generic, from~\cite[th\'eor\`eme 1.3]{BC} or~\cite[Theorem 1]{ABC},
the homoclinic class of each $\cO_i$ coincides with $M$.
Since $f$ has no dominated splitting,
from Proposition~\ref{p.control-exponents}
one can replace each $\cO_i$ by another periodic orbit whose
period is arbitrarily weak and which has no $N$-dominated splitting, for some $N$
arbitrarily large.

Since $f$ is conservative one has $\Delta^+(f,\cO_i)=\Delta(f,\cO_i)$.
One can thus apply Theorem~\ref{t.newhouse} independently in a neighborhood of each $\cO_i$
in order to build horseshoes $K_i$ with arbitrarily small dynamical diameter and whose
topological entropy is close to $\Delta^+(f,\cO_i)$. The conclusion follows.
\end{proof}

We conclude the proof of the lemma by a Baire argument.
Let us consider an open set $U$ in $\Prob(M)$ and an open interval $I\subset (0,+\infty)$.
Let $\cZ(U,I,n)$ be the (open) set of diffeomorphisms $f\in \NDS$
such that any diffeomorphism $g$ close admits
an invariant probability measure $\nu\in U$ with
$h(g,\nu)\in I$ and $h_k(g,\nu)=0$ for $k\leq n$.
We define the open and dense set
$$\cV(U,I,n):=\; [\NDS\setminus \operatorname{closure}(\cZ(U,I,n))]\;
\cup\; \cZ(U,I,n).$$

Let $\cG_0$ be a dense G$_\delta$ subset of $\NDS$ whose elements satisfy
the Claim~\ref{cl.stronglyapproximated}.
Let $(U_n,I_n)$ be a family of pairs of open set $U_n\subset\NDS$
and open intervals $I_n\subset (0,+\infty)$ such that $\{U_n\times I_n\}$
is a basis of the topology of $\NDS\times (0,+\infty)$.
Then $\cG:=\cG_0\cap \bigcap_n \cV_n(U_n,I_n,n)$ is a dense G$_\delta$
subset of $\NDS$.

Let $f$ be a diffeomorphism in $\cG$. For any $\mu\in \Prob(f)$ and $\eps>0$,
there is $n>\eps^{-1}$ such that $(\mu,\Delta^+(f,\mu))\in U_n\times I_n$
and the diameters of $U_n$ and $I_n$ are smaller than $\eps$.
Since $f\in \cG_0$, the Claim~\ref{cl.stronglyapproximated}
can be applied. It shows that $f$
is limit in $\NDS$ of diffeomorphisms $g$
satisfying $h(g,\nu)\in I$ and $h_k(g,\nu)=0$ for $k\leq n$.
Since the horseshoes admit a hyperbolic continuation for nearby diffeomorphisms,
the conclusion of the Claim~\ref{cl.stronglyapproximated} holds for an open
set of diffeomorphisms $g$. Hence $f$ belongs to the closure of $\cV(U_n,I_n,n)$.
Since $f\in \cG$, one deduces that $f$ belongs to $\cV(U_n,I_n,n)$
and the conclusion of Lemma~\ref{l.smallapprox} holds.
\end{proof}

We now easily get the stated properties of the functions $u_n$, $n\in\NN$.

\begin{proof}[Proof of Proposition~\ref{p.ualpha}]
We first show $u_n\leq n\cdot \widetilde{h}$. The case $n=0$ is trivial. Assuming the inequality for some $n\geq0$, we get:
$$u_{n+1}\leq \widetilde{u_n+h} \leq \widetilde{u_n}+\widetilde{h} \leq \widetilde{n\cdot\widetilde {h}}+\widetilde{h} = (n+1)\widetilde{h}.
$$
Ruelle's inequality  $h\leq \Delta^+$ yields inequality \eqref{e.unhdelta}.

We now prove inductively $u_n(\mu)=n\cdot\widetilde{\Delta^+}(\mu)$ for $f\in\NDS$ generic and $\mu\in \Prob(f)$ (not necessarily ergodic). This is obvious for $n=0$.
Pick $\mu_N\to\mu$ such that $\Delta^+(\mu_N)\to\widetilde{\Delta^+(\mu)}$. Lemma~\ref{l.smallapprox} yields $\nu_N$ such that
\begin{itemize}
\item[--]  $d_*(\mu_N,\nu_N)<1/N$, 
\item[--] {$|h(\nu_N)-\Delta^+(\mu_N)|<1/N$,  and}
\item[--] $h_k(\nu_N)=0$ for all $k\leq N$. 
\end{itemize}
{By the second item and Ruelle's inequality we have
$$\Delta^+(\nu_N)\geq h(\nu_N)\geq \Delta^+(\mu_N)-1/N.$$}

Assume $u_n(\mu)=n\cdot\widetilde{\Delta^+}(\mu)$ for some $n\geq 0$.  Then 
 $$\begin{aligned}
   u_{n+1}&(\mu) = \lim_{k\to\infty} \widetilde{u_n+\tau_k}(\mu) 
   = \lim_{k\to\infty} \limsup_{\nu\to\mu} n\cdot\widetilde{\Delta^+}(\nu)+h(\nu)-h_k(\nu)\\
   &\geq \lim_{k\to\infty} \limsup_{N\to\infty} n\cdot\Delta^+(\nu_N)+h(\nu_N)-h_k(\nu_N)\\
   &\quad\quad= \lim_{k\to\infty} \limsup_{N\to\infty} n\cdot\Delta^+(\mu_N)+\Delta^+(\mu_N)
   =(n+1)\widetilde{\Delta^+(\mu)}.
 \end{aligned}$$
Hence, the inequality \eqref{e.unhdelta} is an equality: $u_n=n\cdot\widetilde{\Delta^+}$ for  all $n\in\NN$.
\end{proof}

\subsection{Tail entropy}\label{ss.tailentropy}
In this section we prove Theorem \ref{t.tail-general} and Corollaries \ref{c.unstability-hstar} {and~\ref{c.loc-cte}} for the tail entropy of a {generic conservative diffeomorphism}.  However, the main result, given next, is valid for \emph{any} $C^1$-diffeomorphism.

\begin{theorem}\label{t.bound-tail}
Let $f\in \Diff^1(M)$ with a dominated splitting
$TM=E_1\oplus\dots\oplus E_\ell$. Then, the tail entropy is bounded by
$$h^*(f)\leq \sup \{\Delta_{E_i}(f,\mu),\; \mu\in \Probe(f), 1\leq i\leq \ell\}.$$
\end{theorem}

We will bound the tail entropy by a local entropy $h_\loc(f,\mu,\eps)$ defined with respect to ergodic measures $\mu$ and a small, but uniform scale $\eps$. We will then bound the latter quantity using the Lyapunov exponents of $\mu$.

\subsubsection{Local entropy}

Let us fix $\eps>0$.  For $x\in M$, define the closed bi-infinite Bowen $\eps$-ball as:
$$\bar B_f(x,\eps,\pm\infty)
:=\{ y\in M\, :\, \forall\,  n\in \ZZ,\; d(f^k(x),f^k(y))\leq\eps\}.$$

\begin{definition}
The \emph{local entropy at scale $\eps$} of an ergodic measure $\mu$ is:
 $$
   h_\loc(f,\mu,\eps):=\lim_{\sigma\to1}\; \inf_{\mu(Y)>\sigma}\;
     \lim_{\delta\to 0}\;  \limsup_{n\to+\infty}\; \max_{x\in M}\; \frac1n\log \;   r(\delta,n,\bar B_f(x,\eps,\pm\infty)\cap Y \cap f^{-n}(Y)),
  $$
where $Y\subset M$ is a measurable subset and $r(\delta,n,Z)$ denotes the minimal cardinality of an open cover $\mathcal U$ of a set $Z$ where each $U\in\mathcal U$ satisfies $\diam(f^k(U))\leq \delta$ for any $k=0,\dots,n-1$.
\end{definition}

\begin{remark}\label{r.inverse}  \rm The definition of the local entropy is time symmetric, so that:
$$h_\loc(f^{-1},\mu,\eps) = h_\loc(f,\mu,\eps).$$
\end{remark}

The above local entropy is a variant of the one defined by Newhouse \cite{Newhouse1989}:
$$
   h^{New}_\loc(f,\mu,\eps):=\lim_{\sigma\to 1}\; \inf_{\mu(\Lambda)>\sigma}\;
     \lim_{\delta \to0}\;  \limsup_{n\to+\infty}\; \max_{x\in \Lambda}\; \frac1n\log \;   r(\delta,n,\bar B_f(x,\eps,n)\cap \Lambda)
$$
where $\Lambda$ ranges over the compact subsets with $\mu(\Lambda)>\sigma$ {and $\bar B_f(x,\eps,n)$ is the closure of $B_f(x,\eps,n)$.}
We show that the two are closely related:

\begin{lemma}\label{l-New-new}
For any measure $\mu\in\Proberg(f)$, any $\eps>0$,
 $$h^{New}_\loc(f,\mu,\eps)\leq h_\loc(f,\mu,\eps)
 \leq h^{New}_\loc(f,\mu,2\eps).$$
\end{lemma}

\newcommand\Ho{H}
\newcommand\Ht{h_\sigma}

\begin{proof}
We fix $\eps>0$ and $\mu\in \Proberg(f)$.
The second inequality follows from the fact that for each $x\in M$, there exists $y\in \Lambda$ such that $\bar B_f(x,\eps,\pm\infty)\cap \Lambda\cap f^{-n}(\Lambda)\subset \bar B_f(y,2\eps,n)\cap \Lambda$. We focus on the first inequality.

\medbreak
%
%
%
We fix $\eta>0$ and consider $\Ho=h^{New}_\loc(f,\mu,\eps)$. By definition, there is $0<\sigma<1$, arbitrarily close to $1$, such that, for all compact subsets $\Lambda\subset M$ with $\mu(\Lambda)>\sigma$, all {small enough} $\delta>0$, there is a sequence $(x_k,m_k)_{k\geq1}$ with $x_k\in \Lambda$ and $m_k\to+\infty$ such that,
 \begin{equation}\label{eq-lower-New}
  \forall k\geq1\quad r(\delta,m_k,\bar B_f(x_k,\eps,m_k)\cap \Lambda) \geq \exp ((H-\eta)m_k).
 \end{equation}
Now, by definition of the local entropy, there is a Borel subset $Y$ with $\mu(Y)>\sigma$ such that, for all $\delta>0$ and $x\in M$,
 \begin{equation}\label{eq-upper-loc}
  \forall n\geq n_2(Y,\delta)\quad
   r(\delta,n,\bar B_f(x,\eps,\pm\infty)\cap Y\cap f^{-n}Y)
  \leq \exp ((h_\loc(f,\mu,\eps)+\eta)n).
 \end{equation}
We will compare these two bounds for the following choice of compact set $\Lambda$. By regularity of the measure, there is a compact subset $Y'\subset Y$ with $\mu(Y')>\sigma$. By Birkhoff's theorem, there is a compact set $\Lambda$ with $\mu(\Lambda)>\sigma$ and numbers $\sigma'>\sigma$ and $N\geq1$ such that for any $m\geq N$ and $y\in \Lambda$, the sequence $y,f(y),\dots,f^{m-1}(y)$ contains at least $\sigma'.m$ points in $Y'$.
\medbreak

Let us choose $n\geq 1$ large.
For each $m\gg n$, we consider an $m$-cut $\cC$, that is, a collection of integers
$0=a_0<a_1<\dots<a_r=m-1$ such that:
\begin{itemize}
\item[--] $a_{s+1}=a_s + 1$ or $a_s+n$ for each $0\leq s<r$,
\item[--] $\{s\; : \; a_{s+1}=a_s+1\}$ has cardinality at most $2(1-\sigma).m$.
\end{itemize}
The number of possible $m$-cuts is bounded by
$C_\sigma e^{\Ht m}$ where $C_\sigma,\Ht$ only depend on $\sigma$ with $\lim_{\sigma\to0} \Ht=0$.
From the properties of $\Lambda$ and $N$, we note that for each $y\in \Lambda$ and $m\geq N$, one can find an $m$-cut $\cC=(a_s)$ such that
$f^{a_s}(y)\in Y'$ and $f^{a_s+n}(y)\in Y'$ for each $s$ such that $a_{s+1}=a_s + n$.

For each $k$ large enough, one considers all the possible $m_k$-cuts $\cC=(a_s)$ and one partitions
 $\Lambda$ into sets $\Lambda_\cC$ of points associated to the $m_k$-cut $\cC$.
There is a cut $\cC$ and a set $\Lambda_\cC\subset \Lambda$ such that (using eq. \eqref{eq-lower-New})
$$r(\delta, m_k,\bar B_f(x_k,\eps,m_k)\cap \Lambda_\cC)\geq \frac{r(\delta, m_k,\bar B_f(x_k,\eps,m_k)\cap \Lambda)}{C_\sigma\exp{\Ht.m_k}}\geq
C_\sigma^{-1}\exp((\Ho-\eta-\Ht).m_k).
$$
Since $f$ is a diffeomorphism, there exists $K>0$ such that
for any $r>0$ small, the images by $f$ and $f^{-1}$ of any ball of radius $r$
can be covered by at most $K$ balls of radius $r$.
For each $a_s=a_{s+1}- n$, one can consider
a cover realizing $r(\delta,n, f^{a_s}(\bar B_f(x_k,\eps,m_k)\cap \Lambda_\cC))$.
Pulling back by $f^{a_s}$ each of these covers, taking their join,
and subdividing again each element into $K$ pieces according to each iterate $f^{a_s}$ such that $a_s=a_{s+1}-1$,
one gets a cover of $B_f(x_k,\eps,m_k)\cap \Lambda_\cC$. This proves:
$$r(\delta, m_k,\bar B_f(x_k,\eps,m_k)\cap \Lambda_\cC)\leq
K^{2(1-\sigma).m_k}
\prod_{s:a_s=a_{s+1}- n} r(\delta, n,f^{a_s}(\bar B_f(x_k,\eps,m_k)\cap \Lambda_\cC)).$$
One deduces that there exists $a_s=a_{s+1}-n$ such that (provided $n$ has been chosen large enough):
\begin{equation*}
\begin{split}
\frac 1 n \log r(\delta, n,f^{a_s}(\bar B_f(x_k,\eps,m_k)\cap \Lambda_\cC))
&\geq \Ho-\eta-2(1-\sigma)\log K-\Ht- \frac 1 {m_k} \log C_\sigma\\
&> \Ho-2\eta-2(1-\sigma)\log K-\Ht.
\end{split}
\end{equation*}
Recall that $f^{a_s}(\Lambda_\cC)$ and $f^{a_s+n}(\Lambda_\cC)$ are both contained in $Y'$.
When $m_k\to \infty$, one can furthermore assume that $a_s\to \infty$ and $m_k-a_s-n\to \infty$,
and (up to passing to a subsequence) that $a_s$ and $m_k-a_s-n$ are both larger than $k$.
One has thus obtained a sequence of points $z_k=f^{a_s-k}(x_k)$
such that
$$\frac 1 n \log r(\delta, n,f^{k}(\bar B_f(z_k,\eps,2k+n))\cap Y'\cap f^{-n}(Y'))
> \Ho-2\eta-2(1-\sigma)\log K-\Ht.$$
As $k\to \infty$, one can assume that $(f^k(x_k))$ converges toward a point $x\in M$. Since the set $Y'$ is compact
and contained in $Y$,
 taking a limit of maximal $(\delta,n)$-separated subsets,
one gets that
$$\frac 1 n \log r(\delta, n,\bar B_f(x,\eps,\pm \infty))\cap Y\cap f^{-n}(Y))
> \Ho-2\eta-2(1-\sigma)\log K-\Ht.$$

Combining with~\eqref{eq-upper-loc}, one gets
$$h_\loc(f,\mu,\eps)+\eta>\Ho-2\eta-2(1-\sigma)\log K-\Ht.$$
Taking $\eta$ and $1-\sigma$ arbitrarily small one concludes the proof.
\end{proof}

In the terminology of \cite{DownarowiczEntropyStructure,DownarowiczBook}, $(h-h^{New}_\loc(f,\cdot,1/k))_{k\geq1}$ and $(h-h_\loc(f,\cdot,1/k))_{k\geq1}$ are uniformly equivalent. Therefore the (harmonic extension of the) latter is an entropy structure and the tail entropy satisfies the following variational principle \cite{DownarowiczEntropyStructure}.

\begin{proposition}\label{p.variational}
$
   h^*(f) = \lim_{\eps\to0} \sup_{\mu\in\Proberg(f)} h_\loc(f,\mu,\eps).$
\end{proposition}

\begin{proof}
The variational principle for tail entropy (Theorem~\ref{t.DownarowiczEntropyStructure}) yields:
 $$
   h^*(f) =  \sup_{\mu\in\Prob(f)} \lim_{k\to\infty} \limsup_{\nu\to\mu} h_\loc(f,\nu,1/k)
   \leq \lim_{\eps\to0} \sup_{\mu\in\Proberg(f)} h_\loc(f,\mu,\eps). 
    $$
The converse inequality is immediate from the definitions.
\end{proof}

\subsubsection{Local entropy under a dominated splitting}
Let us now assume that $f$ preserves a dominated splitting
\begin{equation}
TM=E_1\oplus E_2\oplus\dots\oplus E_\ell.
\end{equation}

For each $i$, there is
a backward invariant cone field $\cC_i^-$
which satisfies (see~\cite[Appendix B.1]{BDV}):
$$\bigcap_{n} Df^{-n}(f^{n}(x)).\cC_i^-(f^{n}(x))=E_1\oplus\dots\oplus E_i.$$
We will also need the existence of plaque families~\cite{HPS}:
\begin{theorem}[Plaque families]
If $f$ has a dominated splitting $TM=E\oplus F$, there exists
a continuous family of $C^1$ embeddings
$\psi_x\colon E_x\to M$ such that
\begin{itemize}
\item[--] $\psi_x(0)=x$ and the image of $\psi_x$ is a disc $\cD(x)$ tangent to $E_x$,
\item[--] (local invariance) $f(\psi_{f^{-1}(x)}(B(0,1)))\cup f^{-1}(\psi_{f(x)}(B(0,1)))$ is contained in $\cD(x)$.
\end{itemize}
\end{theorem}

For each index $i$, one considers a locally invariant plaque family $\cD^-_i$ tangent to
$E_1\oplus\dots\oplus E_i$ and a locally invariant plaque family $\cD^+_i$ tangent to
$E_i\oplus\dots\oplus E_\ell$. The intersection $\cD_i(x):=\cD^+_i(x)\cap \cD^-_i(x)$
defines a locally invariant plaque family tangent to $E_i$.

\begin{proposition}\label{p.bowen-domination}
Assume that $f$ has a dominated splitting $TM=E_1\oplus\dots\oplus E_\ell$
and consider locally invariant plaque families $\cD_1,\dots,\cD_\ell$.
Then {for all small enough} $\eps>0$ the following property {holds}.
For any $\mu\in \Proberg(f)$ there exist $\Xi\subset M$ with $\mu(\Xi)=1$ and $i\in \{1,\dots,\ell\}$
such that $$\forall x\in M,\quad \bar B_f(x,\eps,\pm\infty)\cap \Xi\subset \cD_i(x).$$
\end{proposition}

\begin{proof}
From the dominated splitting,
{we} can choose $N\geq 1$ such that, for each $x$ and $i$, one has:
\begin{equation}
\label{e.domination}
{5}\|Df^{N}(f^{-N}(x))|E_{1}\oplus\dots\oplus E_{i-1}\| \leq \|Df^{-N}(x)|E_{i}\|^{-1}.
\end{equation}
We fix $\lambda\in (1/2,1)$. Then we choose $\eps>0$ small.

Consider any $\mu\in \Proberg(f)$.
For any $i$, one defines the measurable set $Z_i^-$ of points $y$ satisfying
$$\forall k\geq 1, \quad \prod_{s=0}^{k-1}\|Df^N(f^{sN}(y))|E_1\oplus \dots\oplus E_{i}\|\leq \lambda^k.$$
One also defines the measurable set $Z_i^+$ of points $y$ satisfying
$$\forall k\geq 1, \quad \prod_{s=0}^{k-1}\|Df^{-N}(f^{-sN}(y))|E_i\oplus \dots\oplus E_{\ell}\|\leq \lambda^k.$$
By Birkhoff's ergodic theorem, there exists an invariant measurable set $\Xi$ with full $\mu$-measure
which satisfies for each $Z=Z^-_i$ or $Z^+_i$: depending on whether $\mu(Z)=0$,
either the orbit of $\Xi$ is disjoint from $Z$ or the orbit of each point $y\in \Xi$ intersects $Z$ infinitely many times in the future and in the past.

\begin{claim}\label{c.inclusion}
Assume $\mu(\Xi\cap Z^-_{i-1})>0$. Then
for any $x\in M$ we have $\bar B_f(x,\eps, \pm\infty)\cap \Xi\subset \cD_{i}^+(x)$.
\end{claim}
\begin{proof}
We take $y\in \bar B_f(x,\eps, \pm\infty)\cap \Xi$.
By definition of $\Xi$, its orbit meets $Z^-_{i-1}$ arbitrarily many times in the past.
From~\cite[Proposition 8.9]{ABC}, each point in $Z^-_{i-1}$ has a uniform neighborhood in $\cD^-_{i-1}$
which is contracted under forward iterations (and which is contained in the stable manifold).
One may consider arbitrarily large backwards iterates $f^{-k}(y)\in \bar B_f(f^{-k}(x),\eps, \pm\infty)$
such that $f^{-k}(y)\in Z^-_{i-1}$: the local stable manifold of $f^{-k}(y)$ in $\cD^-_{i-1}(f^{-k}(y))$ then intersects the plaque $\cD^+_{i}(f^{-k}(x))$.
Provided $\eps$ has been chosen small enough and by local invariance of the plaques, one deduces that
$y$ is arbitrarily close to $\cD_{i}^+(x)$.
\end{proof}

\begin{claim}\label{c.exponent}
If $\bar B_f(x,\eps, \pm\infty)\not \subset \cD^-_{i}(x)$,
then the orbit of any $z\in \bar B_f(x,\eps, \pm\infty)$ intersects $Z_{i}^-$.
Similarly, if $\bar B_f(x,\eps, \pm\infty)\not \subset \cD^+_{i}(x)$,
then the orbit of any $z\in \bar B_f(x,\eps, \pm\infty)$ intersects $Z_{i}^+$.
\end{claim}
\begin{proof}
There exists $K>0$ such that for any two points $x\in M$ and $z\in \cD^-_{i+1}(x)$,
close enough, there is
$y\in \cD_{i+1}(x)$ satisfying:
\begin{itemize}
\item[(a)] $\max(d(y,x), d(z,y))\leq K d(z,x)$,
\item[(b)] there is an arc tangent to $\cC^-_{i}$ in $\cD^-_{i+1}(x)$ that connects $z$ to $y$.
\end{itemize}

Note that from the backward invariance of the cone-field $\cC^-_{i}$
and {(locally)} of the  plaque family, if $d(f^{-j}(x),f^{-j}(y))$ remains small for each $1\leq j\leq N$, then
one still has $f^{-N}(y)\in \cD_{i+1}(f^{-N}(x))$ and the two properties (a) and (b) above are still satisfied by
$f^{-N}(x)$, $f^{-N}(z)$ and $f^{-N}(y)$.

Considering charts, one can take the arcs between $f^{-j}x$ and $f^{-j}y$ almost linear.
This shows that if $\eps>0$ is small enough{,} then for any $x\in M$ and $z\in{{ \cD_{i+1}^-(x)}}$ that are $\eps$-close, {there is $y\in\cD_{i+1}(x)$ satisfying:}
\begin{equation}\label{e.tail2}
d(x,y)\geq (2\|Df^{-N}(x)|E_{i+1}\|)^{-1}d(f^{-N}(x),f^{-N}(y)).
\end{equation}

Let us now consider a Bowen ball
$\bar B_f(x,\eps,\pm\infty)$ centered at some point $x\in M$.
Up to increasing $i$ {(note $Z_{i+1}^-\subset Z_i^-$)}, one can assume that $\bar B_f(x,\eps, \pm\infty)\subset \cD^-_{i+1}(x)$.
For each $z\in \bar B_f(x,\eps, \pm\infty)$ and each multiple $k$ of $N$, one can find
$y_k\in \cD_{i+1}(f^k(x))$ such that for each $0\leq j\leq k$,
$f^{-j}(y_k)$ belongs to $\cD_{i+1}(f^{k-j}(x))$ and may be connected to $f^{k-j}(z)$
{inside} $\cD_{i+1}^-(f^{k-j}(x))$ by an arc tangent to $\cC^{-}_{i}$.
Taking a limit of $f^{-k}(y_k)$ as $k\to+\infty$, one gets a point $y\in \cD_{i+1}(x)$ satisfying,
for any $j\geq 0$:
\begin{itemize}
\item[--] $\max(d(f^j(z),f^j(y)),\; d(f^{j}(y),f^{j}(x)))\leq K d(f^{j}(z),f^{j}(x))$,
\item[--] there is an arc tangent to $\cC^-_{i}$ in $\cD_{i+1}^-(f^{j}(x))$ that connects $f^{j}(z)$ to $f^{j}(y)$.
\end{itemize}

By definition of $i$, there exists
$z_0\in \bar B_f(x,\eps, \pm\infty)\setminus  \cD^-_{i}(x)$. Let $y_0\in \cD_{i+1}(x)$
{satisfy (a) and (b) with respect to $z_0$ and $x$. We can also impose $d(y_0,x)\neq 0$}
since $z_0\notin \cD^-_{i}(x)$. We set $$C_x:= \frac{\eps K}{d(x,y_0)}.$$
For any $z\in \bar B_f(x,\eps,\pm\infty)$ and $k\geq 1$,
one gets from~\eqref{e.domination}, \eqref{e.tail2} and the continuity of $Df^{N}$:
\begin{equation*}
\begin{split}
\eps &\geq d(f^{kN}(x),f^{kN}(z_0))\geq d(x,y_0)\; K^{-1}\; \prod_{s=0}^{k-1}(2 \|Df^{-N}(f^{{(s+1)}N}(x))|E_{i+1}\|)^{-1}\\
&\geq d(x,y_0)\; K^{-1}\; \prod_{s=0}^{k-1}2 \|Df^{N}(f^{sN}(z))|E_{1}\oplus \dots\oplus E_{i}\|.
\end{split}
\end{equation*}
Our choice of $C_x$ gives:
$$
\forall k\geq 1,\; \prod_{s=0}^{k-1} \|Df^{N}(f^{sN}(z))|E_{1}\oplus \dots\oplus E_{i}\|\leq C_x 2^{-k}.
$$
If one assumes by contradiction that the forward orbit of $z$ does not intersects $Z^-_i$, then
for each iterate $f^{jN}(z)$, there exists $k>j$ such that
$$
\forall k\geq 1,\; \prod_{s=j}^{k-1} \|Df^{N}(f^{sN}(z))|E_{1}\oplus \dots\oplus E_{i}\|> \lambda^{k-j}.
$$
One deduces that there exist  arbitrarily large integers $k$ satisfying
$$
\prod_{s=0}^{k-1} \|Df^{N}(f^{sN}(z))|E_{1}\oplus \dots\oplus E_{i}\|> \lambda^k.
$$
This contradicts our choice $\lambda> 2^{-1}$ and proves the first claim. The second one follows by applying the first one to $f^{-1}$.
\end{proof}

To finish the proof of the proposition, we consider the smallest $i$ such that $\Xi\cap Z^-_{i}=\emptyset$
and the largest $j$ such that $\Xi\cap Z^+_{j}=\emptyset$: one can assume that they exist
since otherwise $\Xi$ intersects $Z_\ell^-$ or $Z_1^+$; the measure $\mu$ is then a sink or a source
and the proposition holds trivially. {We can also assume the existence of $x_0\in M$ such that $\bar B_f(x_0,\eps,\pm\infty)\cap\Xi\not\subset\{x_0\}$.}

Claim~\ref{c.inclusion} gives $\bar B_f(x_0,\eps, \pm\infty)\cap \Xi\subset \cD_{i}^+(x_0)\cap \cD_{j}^-(x_0)$.
Since $\cD_i^+$ is transverse to $\cD^-_{i-1}$, one deduces $j\geq i$.
The claim~\ref{c.exponent} and the definition of $Z,i,j$ imply
$\bar B_f(x_0,\eps, \pm\infty)\subset \cD^-_{i}(x_0)\cap \cD^+_j(x_0)$.
Since $\cD_{i+1}^+$ is transverse to $\cD^-_{i}$, one deduces $j\leq i$.
We have shown that $i=j$.  For any $x\in M$ we now see that 
$$
\bar B_f(x,\eps, \pm\infty)\cap \Xi\subset \cD_{i}^+(x)\cap \cD_{j}^-(x)=\cD_i(x).
$$
\end{proof}

\begin{proposition}\label{p.bound-exponents}
Assume that $f$ has a dominated splitting $TM=E_1\oplus\dots\oplus E_\ell$ and consider corresponding locally invariant plaque families $\cD_1,\dots,\cD_\ell$. Let us fix  $\mu\in\Proberg(f)$ and numbers $\eps>0$ and $i\in \{1,\dots,\ell\}$ such that the following holds for some $\Xi\subset M$ with $\mu(\Xi)=1$:
 $$
   \forall x\in M\quad \bar B_f(x,\eps,\pm\infty)\cap \Xi\subset \cD_i(x).
 $$
Then,
 $$
h_\loc(f,\mu,\eps) \leq \Delta^+_{E_i}(f,\mu).
 $$
\end{proposition}

\begin{proof}
We are going to bound $h_\loc(f,\mu,\eps)$ by selecting a suitable measurable set $Y$ with $\mu(Y)>\sigma$ where $\sigma<1$ is arbitrarily close to $1$.
We denote $\gamma:=1-\sigma$.

Denote the Lyapunov spectrum of $\mu$ restricted to $E_i$ by $(\lambda_1,\dots,\lambda_d)$ where $d:=\dim(E_i)$. The Oseledets theorem provides a first measurable set $Y_0\subset \Xi$ and an integer $N>\gamma^{-1}$ such that $\mu(Y_0)>\sigma$, and, for every $y\in Y_0$,
the singular values of $Df^N(y)|E_i(y)$ are $e^{N(\lambda_i+\gamma_i)}$ with $|\gamma_i|<\gamma/2d$.
If $x\in M$ and $y\in \bar B_f(x,\eps,\pm\infty)\cap Y_0$, then
$Df^N(y)|T_y\cD_i(x)$ can be deduced from $Df^N(y)|E_i(y)$ by composition with two bounded linear maps, hence its singular values 
are  $e^{N(\lambda_i+\gamma'_i)}$ with $|\gamma'_i|<\gamma/d$.

The Birkhoff ergodic theorem gives a subset $Y\subset Y_0$ with $\mu(Y)>\sigma$ and an integer $n_0$
such that  for all $y\in Y$ and $n\geq n_0$, the sequence $y,f(y),\dots,f^n(y)$ contains at least $\sigma.n$ points in $Y_0$.

By a well-known argument (used for proving Ruelle's inequality), there exist constants $C_M$ (depending only on $\dim M$) and $\rho_0>0$
such that for any $0<\rho<\rho_0$ the following property holds:

\begin{claim}\label{c-Ruelle-expansion}
Consider any $x\in M$,  $y\in \bar B_f(x,\eps,\pm\infty)\cap Y_0$, and  open subset $U\subset \cD_i(x)$
which contains $y$ such that $U$ has diameter smaller than $\rho$.
Then one can subdivide $U$ into
at most $C_M\exp(N(\Delta_{E_i}^+(f,\mu)+\gamma))$ open subsets $U'$ such that $f^N(U')$ has diameter smaller than $\rho$.
\end{claim}

Since $f$ is Lipschitz and perhaps after reducing $\rho_0>0$, there exists $K>0$ such that for any $0<\rho<\rho_0$, any $x\in M$ and any open subset $U$ of $\cD_i(x)$
with diameter smaller than $\rho$, one can subdivide $U$ into at most $K$ open subsets whose images by $f$ have diameters
smaller than $\rho$.

We now fix $0<\delta<\rho_0$ small and choose $\delta'>0$ smaller such that any set with diameter smaller than $\delta'$
has its first $N$ images with diameter smaller than $\delta$. We also choose $n\geq n_0$ large.
\medskip

We consider an \emph{$n$-cut} $\cC$.  Similarly to the proof of Lemma \ref{l-New-new}, this is a collection of integers $0=a_0<a_1<\dots<a_r=n-1$ such that:
 \begin{itemize}
  \item[--] $a_{s+1}=a_s+1$ or $a_s+N$ for $0\leq s<r$; and
  \item[--] $\cC^*:=\{s:a_{s+1}\ne a_s+1\}\cup\{r\}$ has cardinality at most $(1-\sigma).n=\gamma.n$.
\end{itemize}

Recall that the number of $n$-cuts is bounded by $C_\sigma e^{\Ht n}$,
where $C_\sigma,\Ht$ only depend on $\sigma$ and $\lim_{\sigma\to1}\Ht=0$. We let 
 $$
   Y_\cC:=\{x\in Y:\forall 1\leq s<r,\; s\notin\cC^*\implies f^{a_s}(x)\in Y_0\}.
 $$
 From the choice of $Y$ and $N$, we note that $Y$ is the union of the sets $Y_\cC$. 
 
Fix some point $x\in M$ and $n$-cut $\cC$. We will find a cover of $\bar B_f(x,\eps,\pm \infty)\cap Y_\cC$ by open sets $U$ satisfying $\diam(f^k(U))<\delta$ for any $k=0,\dots,n-1$. Note that it suffices for the open sets $U$ to satisfy  $\diam(f^{a_s}(U))< \delta'$ for any $0\leq s\leq r$.

For $s=0$, it is enough to select a finite set of $\delta'{/2}$-balls covering $Y_\cC$
with cardinality $C_\delta$ independent of $x$ and $n$. We then build inductively the cover for $s>0$
by refining the cover for $s-1$. Recall that for each element of this cover,
the image by $f^{a_{s-1}}$ has diameter smaller than $\delta'$. Two cases occur:
\begin{itemize}
\item[--] If $a_{s}=a_{s-1}+1$, we subdivide each open set of the cover for $s-1$ into at most $K$ open subsets in order to get a cover for $s$.
(Recall our choice of $K$ above.)
\item[--] If $a_{s}=a_{s-1}+N$, Claim~\ref{c-Ruelle-expansion} shows that each open set of the cover for $s-1$
may be subdivided into $C_M\exp(N(\Delta_{E_i}^+(f,\mu)+\gamma))$ open subsets, giving a cover for $s$.
\end{itemize}
We thus get:
$$
r(\delta,n,\bar B_f(x,\eps,\pm \infty)\cap Y_\cC) \leq 
C_\delta\;.\;K^{(1-\sigma)n}\;.\; C_M^{n/N} e^{n(\Delta^+_{E_i}(f,\mu)+\gamma)}.
$$
Since there are at most $C_\sigma e^{\Ht n}$ $n$-cuts and $N>\gamma^{-1}$, we get
 $$
  \limsup_{n\to\infty}\frac1n\log \max_{x\in M} r(\delta,n,\bar B_f(x,\eps,\pm\infty)\cap Y) \leq
     \Delta_{E_i}^+(f,\mu)+(1+\log C_M)\gamma+ (1-\sigma)\log K  +\Ht.
 $$
As this holds for any $\sigma<1$, $\gamma:=2(1-\sigma)$ and some set with $\mu(Y)>\sigma$,
 $$
   h_\loc(f,\mu,\eps)\leq \Delta_{E_i}^+(f,\mu).
 $$
\end{proof}

\subsubsection{Proof of Theorem~\ref{t.bound-tail}}
{To begin with, we observe that} the conclusion of Proposition~\ref{p.bound-exponents}
can be strengthened into
\begin{equation}\label{e.bound-exponents}
h_\loc(f,\mu,\eps) \leq \Delta_{E_i}(f,\mu)=\min({\Delta^-_{E_i}(f,\mu)}, \Delta^+_{E_i}(f,\mu)).
\end{equation}
Indeed, {note that the assumptions of Proposition~\ref{p.bound-exponents} are time symmetric, as is} the local entropy} (Remark~\ref{r.inverse}), and use  $\Delta^+_{E_i}(f^{-1},\mu)=-\Delta^-_{E_i}(f,\mu)$.
\medbreak
{Theorem~\ref{t.bound-tail} follows from Propositions~\ref{p.variational} and~\ref{p.bowen-domination} and \eqref{e.bound-exponents}.}

\subsubsection{Proof of Theorem~\ref{t.tail-general}}
For a generic conservative diffeomorphism $f$, the Main Theorem (revisited) (together with Remark \ref{rem-small-hs}) yields horseshoes with arbitrarily small dynamical diameter and topological entropy arbitrarily close to {$\Delta^*(f)$}  where, we recall,
 \begin{equation}\label{e-Deltastar0}
   \Delta^*(f):=\sup\{\Delta_{E_i}(f,p)\; : \; p \text{ periodic },\;  1\leq i \leq \ell\}
 \end{equation}
with $TM=E_1\oplus\dots\oplus E_{\ell}$ the finest dominated splitting for $f$.
{Hence, $h^*(f)\geq\Delta^*(f)$.}

For the converse inequality, note that according to Theorem~\ref{t.bound-tail}, {for all $f\in\Diff(M)$, $h^*(f)\leq\hat\Delta^*(f)$ where
 $$
   \hat\Delta^*(f):=\sup\{\Delta_{E_i}(f,\mu):\mu\in\Proberg(f), \;1\leq i\leq\ell\}
    \geq\Delta^*(f).
 $$
Moreover, for generic $f$, $\hat\Delta^*(f)=\Delta^*(f)$ by Corollary~\ref{c.ergodic-closing}. 

We turn to the generic upper semicontinuity of  the tail entropy.}
Note that the finest dominated splitting depends semicontinuously on $f$, hence continuously at generic $f$. Then a variant of Lemma \ref{l-usc1} applied to subbundles implies that a generic $f$ is a continuity point of $\Delta^*(f)$. 

{
Assume by contradiction that there are  $f_n\to f_0$ in $\Diff^1_\omega(M)$ with $f_0$  generic and $h^*(f_n)>h^*(f_0)+\eps_0=\Delta^*(f_0)+\eps_0$ for some $\eps_0>0$. 
 By Theorem \ref{t.bound-tail},   $\hat\Delta^*(f_n)>\Delta^*(f_0)+\eps_0$. 

 Using Corollary \ref{c.ergodic-closing}, we can perturb each $f_n$ to $g_n$ with $\Delta^*(g_n)>\hat\Delta^*(f_n)-\eps_0/2$. This contradicts the upper-semicontinuity of $\Delta^*(f)$ at $f=f_0$.
}


\subsubsection{Proof of Corollary~\ref{c.unstability-hstar}}
Item (i) is clearly open and by Theorem~\ref{t.bound-tail} implies that the tail entropy vanishes.
Consequently for \emph{any} diffeomorphism (i) implies (ii).
\medskip

If {item (iv) fails to hold (i.e.,} $h^*(f)\neq 0$), Theorem~\ref{t.bound-tail} yields a periodic point $p$ and some bundle $E_i$ {of the finest dominated splitting} such that $\Delta_{E_i}(f,p)$ {is nonzero and $\eps/2$ close} to $h^*(f)$. Franks lemma {gives}
a $2\eps$-perturbation $g$ of $f$ in the $C^1$-topology
{with} $\Delta_{E_i}(g,p)> h^*(f)+\eps/2$. Note that this is robust.
Choosing $g$ in a residual set, Theorem~\ref{t.tail-general} applies so
$h^*(g)>h^*(f)+\eps/2$: the tail entropy is not locally constant
{at} $f$. Hence item (ii) implies item (iv). 
\medskip

{By Theorem \ref{t.tail-general}, there is a dense G$_\delta$ set of diffeomorphisms $f\in\Diff^1_\omega(M)$ such that $h^*(f)=\Delta^*(f)$ {(see \eqref{e-Deltastar0})}.
Hence to show the implication (iv)$\implies$(i), it is enough to check that the following is dense and open in $\mathcal G$:
 $$
   \mathcal U:=\{ f\in\mathcal G: \Delta^*(f)>0 \text{ or $f$ satisfies (i)}\}.
 $$
{As noted in the proof of} Theorem \ref{t.tail-general}, $\Delta^*(f)$ {and the finest dominated splitting of $f$} are continuous at generic conservative diffeomorphisms $f$. Hence $\mathcal U$ can be assumed to be open in $\mathcal G$. To see its density, let $f\in\mathcal G$.    Assume that (1) $\Delta^*(f)=0$; and (2) the finest dominated splitting contains a subbundle $E_i$ with dimension at least $2$, since otherwise $f$ itself belongs to $\mathcal U$.
Consider now the Lyapunov exponents along $E_i$. Note that $\Delta_{E_i}(f)=0$ implies that, for each periodic point,  the Lyapunov exponents are all nonnegative or all nonpositive (depending on the point). 
If there exists a periodic orbit with an arbitrarily small Lyapunov exponent along $E_i$, then there is an arbitrarily small perturbation $g$ of $f$ with $\Delta_{E_i}(g)>0$ hence $g\in\mathcal U$. Thus we can assume that there is $\delta>0$ such that, for any periodic points, the Lyapunov exponents along $E_i$ all belong to $(\delta,+\infty)$ or all belong to $(-\infty,-\delta)$. 
If both kinds of periodic points occur, the barycenter property for Lyapunov exponents proved in \cite{ABCDW} yields periodic orbits with an arbitrarily small exponent, and we conclude as previously.

We are reduced to the case where all periodic points have all their Lyapunov exponents along $E_i$ greater than $\delta$ (say). This also holds for all ergodic invariant measures by Corollary \ref{c.ergodic-closing}. A standard argument then shows that $E_i$ is uniformly expanded (or contracted).
Arguing similarly for each sub bundle, one gets $f\in \mathcal U$.
This proves that (iv) implies (i).
}

\medbreak

{It remains to show the equivalence of items (iii) and (iv). By Theorem \ref{t.tail-general},  the tail entropy $h^*(f)$ is upper semicontinuous at a generic diffeomorphism $f$. If $h^*(f)=0$, the lower semicontinuity and hence the continuity is obvious: (iv) implies (iii). For the converse, observe that the tail entropy vanishes by \cite{buzzi} on the dense subset $\Diff^\infty_\omega(M)\subset\Diff^1_\omega(M)$ (the density has been proved by \cite{avila}). }

\subsubsection{Proof of Corollary~\ref{c.loc-cte}}
Let us consider a generic $f\in \Diff^1_\omega(M)$ such that $h_\top(f)=h^*(f)$. Recall that $h_\top(f)>0$.
By Theorem~\ref{t.tail-general}, there exist a subbundle $E_i$ of the finest dominated splitting
and a periodic point $p$ such that $\Delta_{E_i}(f,p)$ is arbitrarily close to $h^*(f)$, hence positive.
Franks lemma yields a perturbation increasing $\Delta_{E_i}(f,p)$ by a quantity
which only depends on the $C^1$-distance between $f$ and that perturbation $g$.
One deduces that there exists an open set of diffeomorphisms $\mathcal{U}$ which contains $f$ in its closure
such that $\Delta^*(g)>\Delta^*(f)$ for any $g\in \mathcal{U}$. For the generic diffeomorphism $g$ in $\mathcal{U}$ we have
$$h_\top(g)\geq h^*(g)=\Delta^*(g)>\Delta^*(f)=h_\top(f).$$
Hence the topological entropy is not constant on any neighborhood of $f$ in $\Diff^1_\omega(M)$.

\section{Support of measures and homoclinic classes}
\label{s.nonconservative}

{In this section we prove Theorems \ref{t.infiniteHC} and \ref{t.katok} for homoclinic classes with no dominated splitting in the dissipative setting. }

\subsection{Infiniteness of classes with large entropy}

{To prove Theorem \ref{t.infiniteHC}, we establish the following consequence of
 the Main Theorem and Proposition~\ref{p.control-exponents}.}

\begin{lemma}\label{l.expell}
If $f$ is $C^1$-generic and $H(p)$ has no dominated splitting,
then, for any $\eps>0$, there exist $g$ arbitrarily close to $f$ in $\Diff^1(M)$,
and a horseshoe $K$ of $g$ such that:
\begin{itemize}
\item[--] $K$ is contained in the $\eps$-neighborhood of $H(p_g)$
of the homoclinic class of the continuation $p_g$,
\item[--] $K$ does not intersect the chain-recurrence class of $p_g$,
\item[--] the topological entropy of $K$ is larger than $\Delta(g,p)$.
\end{itemize}
Moreover these properties are still satisfied by diffeomorphisms close to $g$.
\end{lemma}

{We note that the following argument is classical (see, e.g., \cite{bonatti-diaz-aperiodic}), even though this consequence was not observed before.}

\begin{proof}
Let us fix $\delta>0$.
From Proposition~\ref{p.control-exponents} there exists a periodic orbit $\cO$
in $H(p)$ such that $\Delta(f,p)$ and $\Delta(f,\cO)$ are arbitrarily close
(by genericity, one can even assume $\Delta(f,\cO)>\Delta(f,p)$).
Moreover $\cO$ is arbitrarily dense in $H(p)$. In particular the period of $\cO$ is arbitrarily large
and there is no $N$-dominated splitting on $\cO$, where $N$ can be chosen arbitrarily large.

Now we apply
Franks lemma and results in~\cite{BochiBonatti} to obtain a $\delta/3$-perturbation
$f_1$ of $f$ which turns the orbit $\cO$
into a sink or a source. {In particular}, there exists a small neighborhood $U$ of $\cO$ disjoint from the orbit of $p$
which satisfies $f_1(\overline U)\subset U$ or $f_1^{-1}(\overline U)\subset U$.

Applying Franks lemma again, one gets a new diffeomorphism $f_2$ that is $\delta/3$-close to $f_1$
such that $f_2$ {coincides with $f_1$ outside of $f_1(U)\cap f_1^{-1}(U)$ and with} the initial diffeomorphism $f$ on a smaller neighborhood of $\cO$.
Consequently, the neighborhood $U$ still separates the orbit of $p$ and $\cO$,
and we are in condition to apply the Main Theorem to the orbit of $\cO$.

Since $f_2$ coincides with $f$ along the orbit of $\cO$,
the Main Theorem provides us with a diffeomorphism $g$ having a horseshoe $K$
with entropy larger than $\Delta(f,p)$ contained in a small neighborhood of $\cO$.
The diffeomorphisms $f$ and $g$ are $\delta$-close.
Moreover the chain-recurrence classes of $p$ and $K$ are distinct since they are separated
\end{proof}

\begin{proof}[Proof of Theorem~\ref{t.infiniteHC}]
There exists a dense G$_\delta$ set $\cG\subset \Diff^1(M)$
such that for any $f\in \cG$ and any periodic point $p$,
\begin{itemize}
\item[--] the periodic point $p$ is hyperbolic,
\item[--] the homoclinic class $H(p)$ and the chain-recurrence class of $p$
coincide,
\item[--] the homoclinic class $g\mapsto H(p_g)$ of the continuation of $p$
for diffeomorphisms $g$ close to $f$ is continuous at $f$ (in the Hausdorff topology).
\end{itemize}
The first property is classical \cite{Kupka,Smale}, the second is proved in~\cite{BC},
the third one in~\cite{CMP}.

In particular for any $n\geq 1$, the periodic points of period smaller than $n$
are all hyperbolic, their number is finite and these periodic points can be locally {continued}.
Let us fix such an integer $n$.
The continuity of $g\mapsto H(p_g)$ above
and the robustness of the dominated splitting imply:

\begin{claim}
There exists a dense and open subset $\cU(n)\subset \cG$ such that
for {any} $f\in \cU(n)$ and any periodic point $p$ of period smaller than $n$
either $H(p)$ has a dominated splitting,
and this holds also for the continuation $H(p_g)$ for all $g$ $C^1$-close to $f$,
or $H(p)$ has no dominated splitting,
and this holds also for the continuation $H(p_g)$ for all $g$ $C^1$-close to $f$.
\end{claim}

We now conclude the proof of the theorem.
For $\eps>0$, we denote by $\cV(n,\eps)$
the set of diffeomorphisms $f\in \cU(n)$ such that
for any periodic point $p$ with period smaller than $n$,
if $H(p)$ has no dominated splitting, then
for any diffeomorphism $g$ $C^1$-close to $f$,
the $\eps$-neighborhood of $H(p_g)$ contains
a horseshoe $K$ disjoint from the chain-recurrence class of $p_g$,
with topological entropy larger than $\Delta(g,p)$. 
Lemma~\ref{l.expell} shows that $\cV(n,\eps)$
 is a dense and open subset of $\cU(n)$.

Any diffeomorphism $f$ in the dense G$_\delta$ subset
$\bigcap_n \cV(n,1/n)$ satisfies the conclusion of Theorem~\ref{t.infiniteHC}.
\end{proof}

\subsection{Approximation of ergodic measures by horseshoes}

{Given an ergodic measure of a generic diffeomorphism, Corollary \ref{c.ergodic-closing} approximates it by a periodic orbit and, when there is no dominated splitting, the perturbation given by our Main Theorem produces a horseshoe with the required properties. Theorem~\ref{t.katok} claims that the generic diffeomorphism already such a horseshoe. This will follow from} a Baire argument similar to the proof of~\cite[Theorem 3.8]{ABC}.

\begin{proof}[Proof of Theorem \ref{t.katok}]
Note that it is enough to fix $L>0$ and to prove the theorem in restriction to
the set of diffeomorphisms $f$ such that $\|Df\|$ and $\|Df^{-1}\|$ are bounded by $L$:
in particular the topological entropy is bounded by $L\dim(M)$.

Let $\cK(M)$ denote the collection of the compact subsets of $M$
endowed with the Hausdorff metric, a compact metric space.
We define a map $\Phi$ from $\Diff^1(M)$ to
the space $\cK(\Prob(f)\times \cK(M)\times [0,L\dim(M)])$ of compact subsets of $\Prob(f)\times \cK(M)\times [0,L\dim(M)]$
(also endowed with the Hausdorff distance).  This map
associates to each diffeomorphism $f$ the closure of the set of triples
$(\mu,K,h)$ where $K$ is a horseshoe for $f$, $\mu$ is an ergodic measure for $f$ supported on $K$
and $h=h(f, \mu)$.
Since each horseshoe has a hyperbolic continuation, the map $\Phi$
is lower semicontinuous, hence admits a dense G$_\delta$ set $\cG_0$ of continuity points.

Let $\cG_1$ be the dense G$_\delta$ set of diffeomorphisms satisfying
Corollary~\ref{c.ergodic-closing}.
For $f\in \cG_0\cap \cG_1$ and $\mu$ an ergodic measure whose support does not have
any dominated splitting, there exists a sequence of periodic orbits $\cO_n$ which converge
to $\mu$ for the Hausdorff topology {and the weak star} topology, and whose Lyapunov exponents
converge to those of $\mu$. Ruelle's inequality gives $\Delta(f,\mu)\geq h(f,\mu)$.

One deduces from the Main Theorem that there exist a sequence of diffeomorphisms
$(g_n)$ converging to $f$ and a sequence of horseshoes $\Lambda_n$
with topological entropy equal to $h(f,\mu)$
and contained in a small neighborhood of $\cO_n$.
Katok's approximation theorem (\cite{Katok}, see also~\cite[Theorem 1.5]{ACW})
{yields} a sub-horseshoe $K_n\subset \Lambda_n$
whose invariant measures are all close to
$\mu$ and whose topological entropy is close to $h(f,\mu)$.

In particular, considering the measure of maximal entropy $\mu_n$ on $K_n$,
one gets a sequence of triples $(\mu_n,K_n,h_n)\in \Phi(g_n)$ which converge to
$(\mu,\operatorname{Supp}(\mu),h(f,\mu))$.
Since $\Phi$ is continuous, we {obtain}
horseshoes for $f$ close to $\operatorname{Supp}(\mu)$, supporting an ergodic measure $\nu$ weak star close to
$\mu$ with entropy $h(f, \nu)$ close to $h(f, \mu)$.
\end{proof}

\subsection{Dimension of homoclinic classes} \label{ss.dimension}
{We prove Theorem~\ref{t.dimension}, relating the stable and unstable dimensions of $p$ 
and the Hausdorff dimension of the homoclinic class $H(p)$ when the latter has no dominated splitting.
{Recall that Proposition} \ref{p.control-exponents} {yields a periodic orbit without strong dominated splitting allowing our} Main Theorem  {to create a linear horseshoe inside the continuation $H(p_g)$. The estimate} in Proposition~\ref{p.dimension} for {the Hausdorff dimension of further perturbations $H(p_g)$ easily gives:}
\begin{lemma}
If $f$ is $C^1$-generic, if $H(p)$ has no dominated splitting
and if $\Delta^+(f,p)\geq \Delta^-(f,p)$
then, for any $\eps>0$, there exists $g$ arbitrarily close to $f$ in $\Diff^1(M)$
such that $H(p_g)$ has Hausdorff dimension larger than the unstable dimension of $p$.
Moreover this still holds for diffeomorphisms close to $g$.
\end{lemma}

The proof of Theorem~\ref{t.dimension} is then concluded by
a routine Baire argument very similar to the proof of Theorem~\ref{t.infiniteHC}.}

\noindent
\emph{J\'er\^ome Buzzi}\\
{\small Laboratoire de Mathématiques d'Orsay, CNRS - UMR 8628\\
Universit\'e Paris-Sud 11, 91405 Orsay, France}\\
{\tt jerome.buzzi@math.u-psud.fr}\\

\noindent
\emph{Sylvain Crovisier}\\
{\small Laboratoire de Mathématiques d'Orsay, CNRS - UMR 8628\\
Universit\'e Paris-Sud 11, 91405 Orsay, France}\\
{\tt sylvain.crovisier@math.u-psud.fr}\\

\noindent
\emph{Todd Fisher}\\
{\small Department of Mathematics, Brigham Young University,
Provo, UT 84602, USA}\\
{\tt tfisher@math.byu.edu}

\end{document}